\documentclass[12pt]{article}
\usepackage{graphicx}
\usepackage{amsmath}
\usepackage{amssymb}
\usepackage{theorem}
\usepackage{enumerate} 
\usepackage{color}

\numberwithin{equation}{section}

\newtheorem{theorem}{Theorem}[section]
\newtheorem{lemma}[theorem]{Lemma}

\newtheorem{prop}[theorem]{Proposition}
\newtheorem{cor}[theorem]{Corollary}
{\theorembodyfont{\rmfamily}

\newtheorem{example}[theorem]{Example}

\newtheorem{rmk}[theorem]{Remark}
}

\renewcommand{\thesubsection}{\arabic{section}.\arabic{subsection}}

\def\subjclass#1{{\renewcommand{\thefootnote}{}%
\footnote{\emph{Mathematics Subject Classification (2010):} #1}}}
\def\keywords#1{\par\medskip
\noindent\textbf{Keywords.} #1}

\newcommand{\qed}{\hfill \mbox{\raggedright \rule{.07in}{.1in}}}
 
\newenvironment{proof}{\vspace{1ex}\noindent{\bf
Proof}\hspace{0.5em}}{\hfill\qed\vspace{1ex}}
\newenvironment{pfof}[1]{\vspace{1ex}\noindent{\bf Proof of
#1}\hspace{0.5em}}{\hfill\qed\vspace{1ex}}

\newcommand{\R}{{\mathbb R}}
\newcommand{\N}{{\mathbb N}}
 \newcommand{\C}{{\mathbb C}}
 \newcommand{\Z}{{\mathbb Z}}
\newcommand{\cB}{{\mathcal B}}
 \newcommand{\cF}{{\mathcal cF}}
 \newcommand{\cG}{{\mathcal G}}
\renewcommand{\H}{{\mathbb H}}
\newcommand{\barH}{{\overline{\H}}}

\newcommand{\tY}{{\widetilde Y}}
\newcommand{\tZ}{{\widetilde Z}}
\newcommand{\hY}{{\widehat Y}}

\newcommand{\esssup}{\operatorname{ess\, sup}}
\newcommand{\essinf}{\operatorname{ess\, inf}}
\newcommand{\infa}{\SMALL \inf_a}
\newcommand{\suma}{{\textstyle \sum_a}}
\newcommand{\supa}{\SMALL \sup_a}
\newcommand{\sumI}{\SMALL \sum_I}
\newcommand{\supI}{\SMALL \sup_I}
\newcommand{\infI}{\SMALL \inf_I}
\newcommand{\Var}{\operatorname{Var}}
\newcommand{\Leb}{\operatorname{Leb}}
\newcommand{\diam}{\operatorname{diam}}
 
 \newcommand{\spec}{\operatorname{spec}}
 \newcommand{\supp}{\operatorname{supp}}
 \renewcommand{\Re}{\operatorname{Re}}

\newcommand{\eps}{{\epsilon}}

\newcommand{\SMALL}{\textstyle}

\newcommand{\vertiii}[1]{{\left\vert\kern-0.25ex\left\vert\kern-0.25ex\left\vert #1
    \right\vert\kern-0.25ex\right\vert\kern-0.25ex\right\vert}}

\title{Rates of mixing for nonMarkov \\ infinite measure  semiflows}

\author{
Henk Bruin\thanks{Faculty of Mathematics, University of Vienna, 1090 Vienna, AUSTRIA}
\and
Ian Melbourne\thanks{Mathematics Institute, University of Warwick, Coventry, CV4 7AL, UK}
\and
Dalia Terhesiu\thanks{Mathematics Department, University of Exeter, EX4 4QF, UK}
}

\date{22 May 2017, updated 22 September 2018}

\begin{document}

\maketitle
 
\begin{abstract}
We develop an abstract framework for obtaining optimal rates of mixing and higher order asymptotics for
infinite measure semiflows.
Previously, such results were restricted to the situation where there is a first return Poincar\'e map 
that is uniformly expanding and Markov.
As illustrations of the method, we consider semiflows over nonMarkov Pomeau-Manneville intermittent maps 
with infinite measure, and we also obtain mixing rates for semiflows over Collet-Eckmann maps with 
nonintegrable roof function.
\end{abstract}


\section{Introduction}\label{sec:intro}

Decay of correlations is a delicate phenomenon for continuous time dynamical systems.  Exponential decay of correlations has been established  for certain classes of Anosov flows~\cite{BW,Dolgopyat98a,Liverani04,Tsujii},
and the techniques have been extended to
various (non)uniformly hyperbolic flows~\cite{ABV16,AraujoMapp, AGY, BDL, BL14, BMMW}.
Nevertheless, the class of flows for which exponential decay has been established is very restricted.

The situation for superpolynomial decay of correlations (rapid mixing) is somewhat better.  
Rapid mixing for (nontrivial) basic sets for typical Axiom A flows
was established in~\cite{Dolgopyat98b,FMT07},
and was extended in~\cite{M07} to nonuniformly hyperbolic flows given by a suspension over a Young tower with exponential tails~\cite{Young98}. 

For slowly mixing nonuniformly hyperbolic flows (suspensions over Young towers with 
polynomial tails~\cite{Young99}), the method of~\cite{Dolgopyat98b,M07} was used
in~\cite{M09} to establish polynomial decay of correlations.

Recently~\cite{MT17} developed operator renewal theory for continuous time dynamical systems, extending the discrete time theory of~\cite{Gouezel04, Sarig02}.
The framework in~\cite{MT17} applies to slowly mixing nonuniformly hyperbolic semiflows that can be modelled as first return suspensions over full branch Gibbs-Markov maps 
(uniformly expanding Markov maps with an at most countable Markov partition
satisfying bounded distortion, see Remark~\ref{rmk:GM} for more details).  For this class of continuous time systems,~\cite{MT17} shows that
the polynomial decay rates in~\cite{M09} are sharp.

The paper of~\cite{MT17} also addresses mixing and rates of mixing for infinite measure nonuniformly expanding semiflows, extending the work of~\cite{Gouezel11,MT12} in the discrete time setting.
Again, the results in~\cite{MT17} are restricted to the Gibbs-Markov setting.

In the current paper, we introduce a functional analytic framework that
dispenses with the Gibbs-Markov structure in~\cite {MT17}.
The statement and proof of our first main result, Theorem~\ref{thm:rate} below, is somewhat technical, but the main conclusions are easily described.
The prototypical class of dynamical systems with infinite measure are those that
are intermittent 
in the sense of Pomeau-Manneville~\cite{PomeauManneville80}.  
Previous results~\cite{Gouezel11,MT12,MT17} deal with such systems for either discrete time 
or in the continuous time Markov case.  For continuous time nonMarkov examples, 
the previous theory is insufficient.  The new methods in this paper are able to handle such cases, 
as exemplified by the example below.

\begin{example} \label{ex:AFN}
Let $X=[0,1]$ and consider the map $f:X\to X$ given by
\begin{align}\label{eq:f-neutral}
\textstyle f(x) = x(1+a x^{1/\beta}) \bmod 1 \quad\text{where
$\beta \in (\frac12,1),\, a >0$}.
\end{align}
This is an example of an AFN map~\cite{Zweimuller98}, namely a nonuniformly expanding one-dimensional map with 
at most countably (in this case finitely) many branches with finite images and satisfying Adler's 
distortion condition $\esssup |f''|/|f'|^2 < \infty$.
Up to scaling, there is a unique absolutely continuous invariant measure
$\mu_X$.  The measure $\mu_X$ is infinite and the density has a singularity at the neutral fixed point $0$.

Let $\tau_0:[0,1]\to[2,\infty)$ be a continuous roof function
and let $f_t$ denote the suspension semiflow on $X^{\tau_0}$ with
invariant measure $\mu_X^{\tau_0}=\mu_X\times{\rm Lebesgue}$.
Note that there is now a neutral periodic solution of period $\tau_0(0)$.

If $a$ is a positive integer, then $f$ is Markov and the semiflow $f_t$ is covered by the framework in~\cite{MT17}.
Otherwise, we are in the nonMarkov setting and new methods are required.
Roughly speaking, we show that for any $\eps>0$, almost any sufficiently regular roof function $\tau_0$, and
sufficiently regular observables $v,$ $w:X^{\tau_0}\to\R$ supported away from the neutral periodic solution,
there exist explicit constants
$d_1>0$ and $d_2,d_3,\ldots \in\R$ (typically nonzero), such that
 \begin{align*}
\textstyle \int v\,w\circ f_t=
\sum_j d_jt^{-j(1-\beta)}
\int v\int w +O(t^{-(\frac12-\eps)}).
\end{align*}
Here, 
the sum is over those $j\ge1$ with $j(1-\beta)\le \frac12-\eps$.
\end{example}

A more challenging class of examples is provided by suspensions over unimodal maps with 
nonintegrable roof functions. Here there does not exist a uniformly expanding first return 
Poincar\'e map, and the functional-analytic setting is considerable more complicated. 
Our second main result, Theorem~\ref{thm:rateA}, allows us to establish rates of mixing in such examples:

\begin{example} \label{ex:Collet}
Let $X=[0,1]$ and let $f:X\to X$ be a $C^2$ unimodal map with unique non-flat critical point $x_0\in(0,1)$.   
We suppose that $f$ satisfies Collet-Eckmann and slow recurrence conditions with a mixing acip $\mu_X$.   
(See Section~\ref{sec:Collet} for precise formulations of these assumptions.)

Consider a roof function $\tau_0:X\to\R^+$ of the form
$\tau_0=g(x)|x-x_0|^{-1/\beta}$ where $\beta\in(\frac12,1)$ and $g:[0,1]\to(1,\infty)$ is differentiable.
Form the suspension semiflow $f_t:X^{\tau_0}\to X^{\tau_0}$ as in Example~\ref{ex:AFN}.
Define $\kappa_0=\begin{cases}
\hphantom{XXXX} 2\beta(1-\beta) & \beta \ge \frac12(\sqrt 5-1) \\ 
\frac12(\sqrt 5-1)-\beta(\sqrt 5-2) & \beta \le\frac12(\sqrt 5-1)
\end{cases}$.
Then we show that for typical choices of~$g$ and all $\eps>0$,
\[
\textstyle \int v\,w\circ f_t=
d_1t^{-(1-\beta)} \int v\int w +O(t^{-(\kappa_{{\textsc{\tiny 0}}}-\eps)}),
\]
for sufficiently regular observables $v,\,w:X^{\tau_0}\to\R$ supported in $X\times[0,1]$,
where $d_1>0$ is an explicit constant.
(Note that $\kappa_0>1-\beta$ for all $\beta \in (\frac12,1)$.)
\end{example}

\vspace{1ex}
There are two ingredients that make possible the generalisation to semiflows that do not possess a Gibbs-Markov first return Poincar\'e map:
\begin{itemize}
\item[(a)] In Section~\ref{sec:small}, we incorporate ideas
from~\cite{LT} (based on~\cite{KellerLiverani99}) for dealing with perturbation theory of transfer operators, thereby significantly relaxing the functional-analytic hypotheses.
\item[(b)] In Sections~\ref{sec:large} and~\ref{sec:UNI}, we incorporate the idea of
using a second (reinduced) suspension semiflow model for the study of high Fourier modes.
This method was introduced in~\cite{MTtoralsub} for
the study of toral extensions
of nonMarkov slowly mixing dynamical systems  (with finite and infinite measure).   As in~\cite{MTtoralsub}, reinducing facilitates the use of Dolgopyat-type arguments.

In Sections~\ref{sec:abstr} to~\ref{sec:rateA}, we work with a ``superpolynomial'' Dolgopyat assumption~\cite{Dolgopyat98b,M07}, condition (H4) below, whereas in Section~\ref{sec:UNI} we consider an ``exponential'' Dolgopyat assumption (UNI)~\cite{Dolgopyat98a}.   The former has the advantage of applying to a much larger class of dynamical systems (given the current technology), whereas the latter permits a much larger class of observables $v$, $w$.
\end{itemize}

The focus in this paper is on rates of mixing for infinite measure semiflows.  
As pointed out to us by Dima Dolgopyat, P\'eter N\'andori and Doma Sz\'asz, mixing itself does not require Dolgopyat-type arguments, so reinducing is not required for mixing without rates.  However, ingredient (a) remains useful for studying mixing of infinite measure semiflows, see~\cite{MTsub}.  
Also, we anticipate that ingredient~(b) will be useful for
future work on
rates of mixing (correlation decay) for finite measure semiflows, extending~\cite{MT17} from the Gibbs-Markov setting.

The remainder of the paper is organised as follows.
In Section~\ref{sec:abstr} we introduce the abstract functional analytic setting, and state our main results, Theorems~\ref{thm:rate} and~\ref{thm:rateA}.
An outline of the proof of Theorem~\ref{thm:rate} is presented in
Section~\ref{sec:strategy}.  
Sections~\ref{sec:small} and~\ref{sec:large} contain proofs of
the main lemmas (Lemma~\ref{lem:small} and~\ref{lem:large}) dealing with small and large Fourier modes respectively.
In Section~\ref{sec:rateA}, we prove Theorem~\ref{thm:rateA}.
Section~\ref{sec:UNI} shows how to enlarge the class of observables $v$ and~$w$
under the stronger assumption (UNI) for the underlying dynamics.
In Section~\ref{sec:Young}, we show how the hypotheses for Theorem~\ref{thm:rateA} can be verified in examples, such as those in Example~\ref{ex:Collet}, which can be modelled as a suspension over an exponential Young tower.
In Sections~\ref{sec:AFN} and~\ref{sec:Collet}, we apply our main results to 
Examples~\ref{ex:AFN} and~\ref{ex:Collet} respectively.

\vspace{-2ex}
\paragraph{\bf Notation}
We use ``big O'' and $\ll$ notation interchangeably, writing $a_n=O(b_n)$ or $a_n\ll b_n$
if there is a constant $C>0$ such that
$a_n\le Cb_n$ for all $n\ge1$.

\section{Abstract set-up}
\label{sec:abstr}

Let $(Y,d_Y)$ be a bounded metric space with Borel probability measure $\mu$ and let
$F:Y\to Y$ be an ergodic and mixing measure-preserving transformation.
Let $\tau:Y\to\R$ be a nonintegrable roof function bounded away from zero.
For convenience, we suppose that $\essinf \tau>1$.
Throughout we assume that
\begin{itemize}
\item[]
 $\mu(y\in Y:\tau(y)>t)=ct^{-\beta}+O(t^q)$ where $c>0$, $\beta\in(\frac12,1)$
and $q\in(1,2\beta]$.
\end{itemize}
In particular, $\tau\in L^p(Y)$ for all $p<\beta$.

Define the suspension 
$Y^\tau=\{(y,u)\in Y\times\R: 0\le u\le \tau(y)\}/\sim$
where $(y,\tau(y))\sim(Fy,0)$.  The suspension semiflow $F_t:Y^\tau\to Y^\tau$ is  given by $F_t(y,u)=(y,u+t)$, computed modulo identifications. 
The measure 
$\mu^\tau=\mu\times{\rm Lebesgue}$ is ergodic and $F_t$-invariant. 

Next, let $Z\subset Y$ be a subset of positive measure (possibly $Z=Y$).
Let $\sigma:Z\to\Z^+$ be an inducing time such that
$F^{\,\sigma(z)}(z)\in Z$ for all $z\in Z$, yielding the induced map $G=F^\sigma:Z\to Z$.
Throughout we assume that
\begin{itemize}
\item[]
 $\mu(z\in Z:\sigma(z)>n) =O(e^{-dn})$ for some $d>0$.
\end{itemize}

\begin{rmk}
We are specifically interested in the case when $\sigma$ is {\em not\/} the first return time. Otherwise, in light of hypothesis (H2) below, we could induce to a first return semiflow over a Gibbs-Markov map $G$ and proceed as in~\cite{MT17}.
\end{rmk}

Define the induced roof function 
\[
\textstyle \varphi=\tau_\sigma:Z\to\R^+, \qquad \varphi(z)=\sum_{j=0}^{\sigma(z)-1}\tau(F^jz).
\]
Then we can form the suspension semiflow $G_t:Z^\varphi\to Z^\varphi$ with
$Z^\varphi=\{(z,u)\in Z\times\R: 0\le u\le \varphi(z)\}/\sim$
where $(z,\varphi(z))\sim(Gz,0)$,  and $G_t:Z^\varphi\to Z^\varphi$ is  given by $G_t(z,u)=(z,u+t)$, computed modulo identifications. 

\subsection{Rates of mixing and higher order asymptotics}

In this subsection, we state our first main result.
\paragraph{Assumptions on $F$ and $\tau$} \mbox{}
\noindent 
Let $\H=\{\Re s>0\}$ and $\barH=\{\Re s\ge0\}$.
Let $R:L^1(Y)\to L^1(Y)$ denote the transfer operator for $F:Y\to Y$, that is $\int_Y Rv\,w\,d\mu=\int_Y v\,w\circ F\,d\mu$.
Define the twisted transfer operators 
$\hat R(s):L^1(Y)\to L^1(Y)$, $s\in\barH$,
\[
\hat R(s)v= R(e^{-s\tau}v). 
\]

We assume that there exists $p_0\ge1$, and for each $p\in(p_0,\infty)$ and $\eps\in(0,\beta)$ there exists a Banach space $\cB(Y)$ 
containing constant functions with norm $\|\,\|_{\cB(Y)}$, 
and constants $\delta>0$, $\gamma_0\in(0,1)$ and $C>0$ such that
\begin{itemize}

\parskip = -2pt
\item[\textbf{(H1)}]
\begin{itemize}
\item[(i)] 
$\cB(Y)$ is compactly embedded in $L^p(Y)$.
\end{itemize}
\item[(ii)]  
$\|\hat R(s)^n v\|_{\cB(Y)}\le C(|v|_{L^p(Y)}+\gamma_0^n \|v\|_{\cB(Y)})$
for all $s\in\barH\cap B_\delta(0)$, $v\in\cB(Y)$, $n\ge1$.
\item[(iii)] 
$|R(\tau^{\beta-\eps}|v|)|_p\le C\|v\|_{\cB(Y)}$ for all $v\in\cB(Y)$.
\end{itemize}

\begin{rmk} \label{rmk:p}
It is clear from the arguments in this paper that the assumption that~(H1) holds for all $p$ and $\eps$ can be relaxed.  Indeed, there exist $p_0\ge1$ and $\eps_0>0$ depending only on $\beta$ such that (H1) is required to hold only for one value of $p>p_0$ and one $\eps\in(0,\eps_0)$. 
\end{rmk}

\begin{rmk} \label{rmk:H1}
Condition~(H1)(ii) is equivalent to the existence of constants $\gamma_1\in(0,1)$, $n_0\ge1$ and $\tilde C>0$ (with $\gamma_1=\gamma_0^{n_0}$)
such that
$\|\hat R(s)\|_{\cB(Y)}\le \tilde C$ and
$\|\hat R(s)^{n_0} v\|_{\cB(Y)}\le \tilde C|v|_{L^p(Y)}+\gamma_1 \|v\|_{\cB(Y)}$
for all $s\in\barH\cap B_\delta(0)$, $v\in\cB(Y)$.

In particular, if the family $s\mapsto \hat R(s)$ of operators on $\cB(Y)$ is continuous at $s=0$, then (H1)(ii) holds for sufficiently small $\delta$ if and only it holds for $s=0$.
\end{rmk}

\paragraph{Assumptions on $G$ and $\sigma$}

Let $d_Z$ be a metric on $Z$.   
For $\ell\ge1$, define $\tau_\ell=\sum_{j=0}^{\ell-1}\tau\circ F^j$.  We assume
\begin{itemize}
\item[\textbf{(H2)}] 
There is an at most countable measurable partition $\alpha$
of $Z$ with $\mu(a)>0$ for all $a\in\alpha$ such that $\sigma$ is constant on partition elements.
Moreover, there are constants $\lambda>1$, $\eta\in(0,1]$, $C>0$,
such that for each $a\in\alpha$,
\begin{itemize}
\item[(i)] $G=F^\sigma$ restricts to a (measure-theoretic) bijection from $a$  onto $Z$.
\item[(ii)] $d_Z(Gz,Gz')\ge \lambda d_Z(z,z')$ for all $z,z'\in a$.
\item[(iii)] $\xi=\frac{d\mu|_Z}{d\mu|_Z\circ G}$
satisfies  
$|\log \xi(z) - \log \xi(z')|  \le C  d_Z(Gz,Gz')^\eta$
for all $z,z'\in a$.
\item[(iv)] $d_Y(F^\ell z,F^\ell z')\le Cd_Z(Gz,Gz')$ for all $z,z'\in a$,
$0\le \ell <\sigma(a)$.
\end{itemize}
\item[\textbf{(H3)}]
There exists $C>0$ such that 
$|\tau_\ell(z)-\tau_\ell(z')|\le C(\infa\varphi)d_Z(Gz,Gz')^\eta$
for all $a\in\alpha$, $z,z'\in a$, $1\le \ell\le \sigma(a)$.  
\item[\textbf{(H4)}] 
\setcounter{footnote}{0}
Approximate eigenfunction condition~\cite{Dolgopyat98b,M07,rapid}.  It suffices~\cite[Proposition~5.2]{rapid} that
 there exist three fixed points $z_i\in \bigcup_{a\in\alpha}a$, $i=1,2,3$ for $G:Z\to Z$ with periods $p_i=\varphi(z_i)$ for the suspension semiflow $G_t:Z^\varphi\to Z^\varphi$ such that
that $(p_1-p_3)/(p_2-p_3)$ is Diophantine.\footnote{In~\cite{MT17}, the corresponding assumption~(A.2) is misstated in terms of two periodic orbits, and should be replaced by the condition here (see~\cite{rapid} for more details).  The general approximate eigenfunction condition in~\cite[Definition~4.2]{MT17} is stated correctly.}
\end{itemize}

\begin{rmk}
\label{rmk:GM}
Assumptions (H2)(i)-(iii) mean that $G:Z\to Z$ is a Gibbs-Markov map
(standard references for background material on Gibbs-Markov maps
are~\cite[Chapter~4]{Aaronson} and~\cite{AaronsonDenker01}).
Let $\alpha_n$ denote the partition of $Z$ into $n$-cylinders, and define
$\xi_n=\prod_{j=0}^{n-1}\xi\circ G^j$.
The transfer operator $R^G:L^1(Z)\to L^1(Z)$ for $G$ satisfies
$((R^G)^nv)(z)=\sum_{a\in\alpha_n} \xi_n(z_a)v(z_a)$ where $z_a$ is the unique preimage in $a$ of $z$ under $G^n$,  and there exists $C>0$ such that
for all $z,z'\in a$, $a\in\alpha_n$, $n\ge1$,
\begin{align} \label{eq:GM}
\xi_n(z)\le C\mu(a), \quad |\xi_n(z)-\xi_n(z')|\le C\mu(a)d_Z(G^nz,G^nz')^\eta.
\end{align} 

If $F$ is Gibbs-Markov, then 
we can take $Z=Y$, $\sigma\equiv1$, $R^G=R$, and it suffices to verify (H3) with $\ell=1$ and (H4), thereby reducing to the hypotheses in~\cite{MT17}.    Indeed it is immediate that condition (H2) is redundant and that (H3) reduces to the case $\ell=1$.
Taking $\cB(Y)$ to be a H\"older space, (H1)(i,ii) are standard,
 see for instance~\cite[Proposition 3.5]{MT17}.
Also,
$|R(\tau^{\beta-\eps}v)|_\infty\le \sum_{a\in\alpha} \supa\xi\,\supa(\tau^{\beta-\eps}v)
\le C|v|_\infty\sum_a\mu(a)\supa\tau^{\beta-\eps}$.
Assuming (H3) with $\ell=1$, we have $\supa\tau-\infa\tau\le C\infa\tau(\diam Y)^\eta$,
so $\sup_a\tau\ll\infa\tau$ and $|R(\tau^{\beta-\eps}v)|_\infty\ll|v|_\infty
\sum_a\mu(a)\infa\tau^{\beta-\eps}\le |v|_\infty\int_Y\tau^{\beta-\eps}\,d\mu\ll\|v\|_{\cB(Y)}$ verifying (H1)(iii).
\end{rmk}

\paragraph{Observables}
Let $\eta\in(0,1]$.
Define $\tY=Y\times [0,1]$. 
Given $v:\tY\to\R$, we define 
\[
\|v\|_{C^\eta}=|v|_\infty+|v|_{C^\eta}, \quad
|v|_{C^\eta}=\sup_{y,y'\in Y,\,y\neq y'}\sup_{u\in[0,1]} |v(y,u)-v(y',u)|/d_Y(y,y')^\eta.
\]
Let $C^\eta(\tY)$ be the space of observables $v:\tY\to\R$ for which
$\|v\|_{C^\eta(\tY)}<\infty$.

Next, let $\cB(Y)$ be the Banach space in (H1).
Define $\|v\|_{\cB(\tY)}=\sup_{u\in[0,1]}\|v(\cdot,u)\|_{\cB(Y)}$. 
Then $\cB(\tY)$ is the space consisting of those 
$v\in L^1(\tY)$ with $\|v\|_{\cB(\tY)}<\infty$. 

If $v\in \cB(\tY)\cap C^\eta(\tY)$, then we define
$\vertiii{v}  = \|v\|_{\cB(\tY)}+ \|v\|_{C^\eta(\tY)}$.

For $w:Y\times(0,1)\to\R$, $m\ge0$,
set $|w|_{\infty,m}=\max_{j=0,\dots,m}|\partial_u^jw|_\infty$.
We write $w\in L^{\infty,m}(\tY)$ if 
$\supp w\subset Y\times (0,1)$ and
$|w|_{\infty,m}<\infty$.

\vspace{2ex}
Define 
\[
\rho_{v,w}(t)=\int_{Y^\tau} v\,w\circ F_t\,d\mu^\tau.
\]

We can now state the first main result.

\begin{theorem}\label{thm:rate}
Suppose that (H1)--(H4) hold.
Define 
\[
\kappa=\begin{cases} \beta(1-2\beta+q)/q & q<2\beta \\
\frac12-\eps & q=2\beta \end{cases}\, ,
\enspace \qquad \text{where $\eps>0$ is arbitrarily small.}
\]

Then there exist constants
$d_1=\frac{1}{c\pi}\sin\beta\pi$, $d_2,d_3,\ldots \in\R$, and
there exists \mbox{$m\ge2$}, such that
 \begin{align*}
\textstyle \rho_{v,w}(t)=
\sum_j d_jt^{-j(1-\beta)}
\int_{\tY} v\,d\mu^\tau\, \int_{\tY} w \,d\mu^\tau
+O(\vertiii v |w|_{\infty,m}\,t^{-\kappa}),
\end{align*}
for all $v\in\cB(\tY)\cap C^\eta(\tY)$,
$w\in L^{\infty,m}(\tY)$, $t>0$.
Here, the sum is over those $j\ge1$ with $j(1-\beta)< \kappa$.
\end{theorem}

\begin{rmk} As indicated in the introduction, for a more restricted class of dynamical systems satisfying a uniform nonintegrability (UNI) condition, we obtain a stronger result; namely the conclusion of Theorem~\ref{thm:rate} holds for $m=2$.  See Section~\ref{sec:UNI} for a precise statement.
\end{rmk}

\begin{rmk}  
(a) 
Suppose that $\mu(\varphi>t)=ct^{-\beta}+O(t^{-2\beta})$ and that $\beta>\frac34$.
If in addition
$d_2\neq0$ (as is typically the case), then we obtain
second order asymptotics in Theorem~\ref{thm:rate} and the mixing rate is sharp.
These results are identical to the ones obtained in~\cite[Section~9]{MT12} in the discrete time context.

We note that the proof of Theorem~\ref{thm:rate} gives explicit (but not particularly nice) formulas for the constants $d_j$, $j\ge2$ 
(cf.\ \cite[top of p.~89]{MT12}). 
\\
(b) Mixing rates and higher order asymptotics for the case $\beta=1$ can be also obtained  along the
lines of~\cite[Section~9.2]{MT12}. We omit this case here.
\end{rmk}

\subsection{Alternative hypotheses for mixing rates}

Theorem~\ref{thm:rate} gives explicit, often optimal, rates of mixing as well as high order asymptotics.
In this subsection, we state an alternative hypothesis in place of (H1) under which it is still possible to obtain rates of mixing, though the estimates are rougher.

We assume that for every (sufficiently large) $p\in(1,\infty)$, there exists a Banach space $\cB(Y)$ containing constant functions, with norm $\|\,\|_{\cB(Y)}$,
and constants $\delta>0$, $\gamma_0\in(0,1)$ and $C>0$ such that
\begin{itemize}

\parskip = -2pt
\item[\textbf{(A1)}]
\begin{itemize}
\item[(i)]
$\cB(Y)$ is compactly embedded in $L^p(Y)$.
\end{itemize}
\item[(ii)]
$\|\hat R(s)^n v\|_{\cB(Y)}\le C(|v|_{L^1(Y)}+\gamma_0^n \|v\|_{\cB(Y)})$
for all $s\in\barH\cap B_\delta(0)$, $v\in\cB(Y)$, $n\ge1$.
\end{itemize}

It follows from these assumptions (see Lemma~\ref{lem:continfA} below), that
(after possibly shrinking $\delta$) there is a continuous family of simple eigenvalues $\lambda(s)$ for $\hat R(s):\cB(Y)\to\cB(Y)$, $s\in\barH\cap B_\delta(0)$, with $\lambda(0)=1$.
Let $\zeta(s)\in\cB(Y)$ be the corresponding family of eigenfunctions normalized so that $\int_Y\zeta(s)\,d\mu=1$.
We assume further that there exists $\beta_+\in(\beta,1)$ such that
\begin{itemize}
\item[\textbf{(A1)}]
\begin{itemize}
\item[(iii)] $\big|\int_Y(e^{-s\tau}-1)(\zeta(s)-1)\,d\mu\big|\le C|s|^{\beta_+}$ for all $s\in \barH\cap B_\delta(0)$.
\end{itemize}
\end{itemize}

\begin{theorem}\label{thm:rateA}
Suppose that (A1) and (H2)--(H4) hold.
Let $\kappa=\beta(1-2\beta+\beta_+)/\beta_+$,
$d_1=\frac{1}{c\pi}\sin\beta\pi$.
There exists
$m\ge2$, such that for all $\eps>0$,
 \begin{align*}
\textstyle \rho_{v,w}(t)=
d_1t^{-(1-\beta)}
\int_{\tY} v\,d\mu^\tau\, \int_{\tY} w \,d\mu^\tau
+O(\vertiii v |w|_{\infty,m}\, t^{-(\kappa-\eps)}),
\end{align*}
for all $v\in\cB(\tY)\cap C^\eta(\tY)$,
$w\in L^{\infty,m}(\tY)$, $t>0$.
\end{theorem}

\begin{rmk} \label{rmk:rateA}
Note that $\beta(1-2\beta+x)/x=1-\beta$ when $x=\beta$, and hence this expression is strictly greater than $1-\beta$ for $x=\beta_+>\beta$.

If $q=2\beta$ and (A1)(iii) holds with $\beta_+\ge2\beta$ then 
it follows from the methods in this paper that 
we obtain the same error rates and asymptotic expansions as in Theorem~\ref{thm:rate}.  
If $\min\{q,\beta_+\}\in(1,2\beta)$, then 
we obtain essentially the same error rates (up to an~$\epsilon$) and asymptotic expansions as in Theorem~\ref{thm:rate}.  

However, the criteria in Section~\ref{sec:Young} for verifying (A1)(iii) hold only for $\beta_+\in(\beta,1)$, hence our restriction to this range in (A1)(iii) and
Theorem~\ref{thm:rateA}.
\end{rmk}

\subsection{Semiflows on ambient manifolds}
\label{sec:ambient}

In applications, we are often given a semiflow $f_t:M\to M$ on
a finite-dimensional manifold $M$, with codimension one cross-section $X$ and first hit time $\tau_0:X\to \R^+$ and Poincar\'e map $f:X\to X$.
Here $\tau_0(x)>0$ is least such that $f_{\tau_0(x)}(x)\in X$ and
$f(x)=f_{\tau_0(x)}(x)$.

We are particularly interested in the situation where $f:X\to X$ possesses a conservative ergodic absolutely continuous  infinite Borel measure $\mu_X$.
In this case, we fix a subset $Y\subset X$ with $\mu_X(Y)\in(0,\infty)$.
Define the first return time $r:Y\to\Z^+$
and the first return map $F=f^r:Y\to Y$ with ergodic invariant probability
measure $\mu=(\mu_X|Y)/\mu_X(Y)$.
The induced roof function $\tau(y)=\sum_{j=0}^{r(y)-1}\tau_0(f^jy)$
and suspension semiflow $F_t:Y^\tau\to Y^\tau$ is as defined above.

Let $\pi_M:Y^\tau\to M$ denote the semiconjugacy between $F_t$ and $f_t$
given by $\pi_M(y,u)=f_uy$.
We assume that the suspension semiflow $F_t:Y^\tau\to Y^\tau$ falls into the abstract setting above.
Let $\widetilde M=\bigcup_{t\in[0,1]}f_t(Y)=\pi_M(\tY)$.
Let $v,w:\widetilde M\to\R$ be observables such that
$w\in L^{\infty,m}(\widetilde M)$ and 
 $v\circ\pi_M\in\cB(\tY)\cap C^\eta(\tY)$ where $\cB(Y)$ is the Banach space in~(H1).
Then it is immediate that Theorem~\ref{thm:rate} applies
to $\int_M v\,w\circ f_t\,d\mu_X$.

\section{Strategy of proof for Theorem~\ref{thm:rate}}
\label{sec:strategy}

Recall that $F_t:Y^\tau\to Y^\tau$ is a suspension semiflow over a mixing map $F:Y\to Y$ with roof function $\tau:Y\to\R^+$, and that
$G_t:Z^\varphi\to Z^\varphi$ is a suspension semiflow over 
$G=F^\sigma:Z\to Z$ with roof function $\varphi=\tau_\sigma:Z\to\R^+$.
Recall also that $\mu^\tau=\mu\times{\rm Lebesgue}$ is an ergodic $F_t$-invariant measure on $Y^\tau$.

Assumption (H2)(i)-(iii) guarantees that there is a unique ergodic $G$-invariant
probability measure $\mu_Z$ on $Z$ that is absolutely continuous with respect to $\mu|_Z$.  
We obtain
an ergodic $G_t$-invariant measure 
$\mu_Z^\varphi=(\mu_Z\times{\rm Lebesgue})/\int_Z\sigma\,d\mu_Z$ on $Z^\varphi$.

The projection $\pi:Z^\varphi\to Y^\tau$ given by $\pi(z,u)=F_u(z,0)$ defines a semiconjugacy between the suspension semiflows 
$G_t:Z^\varphi\to Z^\varphi$ and $F_t:Y^\tau\to Y^\tau$. 
The following result, proved in Section~\ref{sec:large}, shows that
$\pi$ is measure-preserving.

\begin{prop} \label{prop:mu}
$\pi_*\mu_Z^\varphi=\mu^\tau$.
\end{prop}
It follows that
\[
\rho_{v,w}(t)=\int_{Y^\tau}v\,w\circ F_t\,d\mu^\tau
=\int_{Z^\varphi}\hat v\,\hat w\circ G_t\,d\mu_Z^\varphi \quad\text{
where $\hat v=v\circ\pi$, $\hat w=w\circ\pi$.}
\]

Let $\hat\rho_{v,w}(s)=\int_0^\infty e^{-st}\rho_{v,w}(t)\,dt$ denote the Laplace transform of $\rho_{v,w}(t)$.  This is analytic on $\H$.  We are particularly interested in the behaviour of $\hat\rho_{v,w}(s)$ for $s=ib$ purely imaginary.
As indicated in the introduction (ingredient~(b)) the strategy in this paper is to analyse $\hat\rho_{v,w}(ib)$ using the two different expressions for $\rho_{v,w}(t)$.  For $b$ in a neighbourhood of $0$ ($b$ ``small''), we use the
representation
$\rho_{v,w}(t)=\int_{Y^\tau}v\,w\circ F_t\,d\mu^\tau$.
 For $b$ outside a neighbourhood of $0$  ($b$ ``large'')
we use the representation
$\rho_{v,w}(t) =\int_{Z^\varphi}\hat v\,\hat w\circ G_t\,d\mu_Z^\varphi$.

The resulting estimates are stated in Lemmas~\ref{lem:small} and~\ref{lem:large} below, and are proved in
Sections~\ref{sec:small} and~\ref{sec:large}.

In the remainder of this section, we state the key estimates for small and large $b$ (Lemmas~\ref{lem:small} and~\ref{lem:large})
and use them to prove Theorem~\ref{thm:rate}.
Except in the proof of Lemma~\ref{lem:infty} below, we write $\rho(t)$ and $\hat\rho(s)$, suppressing the dependence on $v$ and $w$.

We assume (H1)--(H4) throughout.  
Let $c_\beta=i\int_0^\infty e^{-i\sigma}\sigma^{-\beta}\,d\sigma$.
Recall that $\mu(\tau>t)=ct^{-\beta}+O(t^{-q})$ where $c>0$, $\beta\in(\frac12,1)$, $q\in(1,2\beta]$.

\begin{lemma} \label{lem:small}
For all $\eps>0$, there exists $C$, $\delta>0$, such that 
for $v\in \cB(\tY)$, $w\in L^\infty(\tY)$,
\begin{itemize}
\item[(a)] 
$|\hat\rho(s)|  \le C |s|^{-\beta}\|v\|_{\cB(\tY)}\,|w|_{L^\infty(\tY)}$
for all $s\in\barH\cap B_\delta(0)$.
\item[(b)] 
$|\hat \rho(i(b+h))-\hat \rho(ib)|\le C 
\{b^{-2\beta}h^\beta +b^{-\beta} h^{\beta-\eps} \}\|v\|_{\cB(\tY)}|w|_{L^\infty(\tY)}$
for all \mbox{$0<h<b<\delta$}.
\item[(c)]
There are constants $c_j\in\C$ with $c_0=c^{-1}c_\beta^{-1}$ such that
for all $a\in(0,\delta t)$, $\eps>0$,
\begin{align*}
\int_0^{a/t}  e^{ibt}\hat\rho(ib)\,db
 =\sum_j c_j & \int_0^{a/t}  b^{-((j+1)\beta-j)}  e^{ibt}\,db\int_{\tY}v\,d\mu^\tau \int_{\tY}w\,d\mu^\tau \\ & +
O\big(\{(a/t)^{1-2\beta+q}+(a/t)^{1-\eps}\}\|v\|_{\cB(\tY)}\,|w|_{L^\infty(\tY)}\big),
\end{align*}
where the sum is over those $j\ge0$ with $(j+1)\beta-j\ge 2\beta-q$.
\end{itemize}
\end{lemma}

\begin{lemma} \label{lem:large}
Let $\delta$, $\eps>0$.  There exists $C$, $\omega>0$ such that
\[
|\hat\rho(i(b+h))-\hat\rho(ib)|\le Cb^\omega h^{\beta-\eps}
\|v\|_{C^\eta(\tY)}\,|w|_{L^\infty(\tY)},
\]
for all $0<h<\delta<b$,
$v\in C^\eta(\tY)$, $w\in L^\infty(\tY)$.
\end{lemma}

Lemmas~\ref{lem:small} and~\ref{lem:large} are proved in Sections~\ref{sec:small} and~\ref{sec:large} respectively.
We now have the necessary prerequisites for completing the proof
of Theorem~\ref{thm:rate}.

\begin{prop}[{{cf.~\cite[Proposition~6.2]{MT17}}}] \label{prop:0cx} 
The analytic function $\hat\rho$ on $\H$ extends to a continuous function
on $\barH\setminus\{0\}$, and
\[
\rho(t)=\frac{1}{2\pi}\int_{-\infty}^\infty e^{ibt}\hat\rho(ib)\,db
= \frac{1}{\pi}\int_0^\infty \Re(e^{ibt}\hat\rho(ib))\,db.
\]
\end{prop}

\begin{proof} 
This is the same as the proof of~\cite[Proposition~6.2]{MT17} with~\cite[Proposition~6.1]{MT17}
replaced by Lemma~\ref{lem:small}(a).
\end{proof}

\begin{lemma}[{{cf.~\cite[Proposition~6.4]{MT17}}}] 
\label{lem:DT}
Let $\eps>0$.
There exists $C$, $\delta>0$ such that 
for all $a\ge1$, $t>(a+\pi)/\delta$, $v\in \cB(\tY)$, $w\in L^\infty(\tY)$,
\[
\Big|\int_{a/t}^{\delta}  e^{ibt}\hat\rho(ib)\,db\Big|
\le C\{ t^{-(1-\beta)}a^{-(2\beta-1)} + t^{-(\beta-\eps)}\}\|v\|_{\cB(\tY)}|w|_{L^\infty(\tY)}.
\]
\end{lemma}

\begin{proof} 
Throughout, we suppress the factor $\|v\|_{\cB(\tY)}|w|_{L^\infty(\tY)}$.
Write
\[
I=\int_{a/t}^{\delta} e^{ibt}\hat\rho(ib)\,db
=-\int_{(a+\pi)/t}^{\delta+\pi/t} e^{ibt}\hat\rho(i(b-\pi/t))\,db.
\]
Then $2I=I_1 + I_2 +  I_3$, where
\begin{align*} 
I_1 & = -\int_{\delta}^{\delta+\pi/t} e^{ibt}\hat\rho(i(b-\pi/t))\,db, \qquad
I_2  = \int_{a/t}^{(a+\pi)/t} e^{ibt}\hat\rho(ib)\,db, \\
I_3 & = \int_{(a+\pi)/t}^{\delta} e^{ibt}(\hat\rho(ib)-\hat\rho(i(b-\pi/t)))\,db.
\end{align*}
Clearly $I_1=O(t^{-1})$. By Lemma~\ref{lem:small}(a),
$|I_2|  \ll \int_{a/t}^{(a+\pi)/t} b^{-\beta}\,db \le (\pi/t)(a/t)^{-\beta}\le
 t^{-(1-\beta)}a^{-\beta}$.
By Lemma~\ref{lem:small}(b) with $h=\pi/t$,
\begin{align*}
|I_3| & \ll 
t^{-\beta}\int_{a/t}^\infty b^{-2\beta}\,db+ t^{-(\beta-\eps)}\int_0^{\delta}b^{-\beta}\,db 
\ll t^{-(1-\beta)}a^{-(2\beta-1)} + t^{-(\beta-\eps)}.
\end{align*}
This completes the proof.
\end{proof}

\begin{lemma}[{{cf.~\cite[Proposition~6.5]{MT17}}}] 
\label{lem:infty}  For any $\delta>0$,  $\eps\in (0,\beta)$, there exists $C>0$, $m\ge2$,
such that
for all $t>0$, $v\in \cB(\tY)\cap C^\eta(\tY)$, $w\in L^{\infty,m}(\tY)$,
\[
\textstyle |\int_{\delta}^\infty  e^{ibt}\hat\rho(ib)\,db|
\le C t^{-(\beta-\eps)}\|v\|_{C^\eta(\tY)}\,|w|_{L^{\infty,m}(\tY)}.
\]
\end{lemma}

\begin{proof} 
Let $\omega$ be as in Lemma~\ref{lem:large}.
Choose $m$ such that $m>\omega+1$.
By~\cite[Proposition~3.7]{MT17},
$\hat\rho_{v,w}(s)=\hat p_m(s)+\hat r_m(s)$, where
$\hat p_m(s)$ is a linear combination of $s^{-j}$, $j=1,\dots,m$, and 
$\hat r_m(s)=s^{-m}\hat \rho_{v,\partial_u^mw}(s)$. 

By the proof of~\cite[Proposition 6.5]{MT17}, 
 $|\int_{\delta}^\infty  e^{ibt}\hat p_m(ib)\,db|\ll t^{-1}|v|_{L^\infty(\tY)}|w|_{L^{\infty,m}(\tY)}$.

By Proposition~\ref{prop:0cx}, $\hat r_m$ is well-defined and continuous on
$\barH\setminus\{0\}$.
 By Lemma~\ref{lem:large}  with $h=\pi/t$, 
\[
|\hat r_m(ib)-\hat r_m(i(b-\pi/t))|\ll
b^{-(m-\omega)} t^{-(\beta-\eps)}\|v\|_{C^\eta(\tY)}|\partial_u^mw|_{L^\infty(\tY)}.
\]
Suppressing the factor $\|v\|_{C^\eta(\tY)}|\partial_u^mw|_{L^\infty(\tY)}$,
\begin{align*}
\textstyle |2\int_{\delta}^\infty e^{ibt}\hat r_m(ib)\,db| & \textstyle \le \int_{\delta}^\infty|\hat r_m(ib)-\hat r_m(i(b-\pi/t))|\,db
+ \int_{\delta}^{\delta+\pi/t} |\hat r_m(i(b-\pi/t))|\,db \\ &
\textstyle \ll  t^{-(\beta-\eps)}\int_{\delta}^\infty b^{-(m-\omega)}\,db
+O(t^{-1})= O( t^{-(\beta-\eps)}),
\end{align*}
where in the last inequality we have used that $m>\omega+1$.
\end{proof}

\begin{pfof}{Theorem~\ref{thm:rate}}
We let $a=t^\gamma$, where $\gamma\in(0,1)$ is chosen later.
A calculation (see for example~\cite[Proposition~9.5]{MT12}) shows that
for every $j\ge 0$, there exists $C_j\in\C$, with 
$C_0=\int_0^\infty e^{i\sigma}\sigma^{-\beta}\,d\sigma=i\Gamma(1-\beta)e^{-i\beta\pi/2}$,
such that
\begin{align*}
 \int_0^{a/t}b^{-((j+1)\beta-j)}  e^{ibt}db
 -C_j\,t^{-(j+1)(1-\beta)}
&  \ll t^{-(j+1)(1-\beta)}a^{-((j+1)\beta-j)}
\\ & =t^{-1}(a/t)^{j-(j+1)\beta}
\ll  t^{-(1-\beta)}a^{-\beta}.
\end{align*}
By Lemma~\ref{lem:small}(c), writing $c_j'=c_jC_j\int_{\tY}v\,d\mu^\tau \int_{\tY}w\,d\mu^\tau$,
\[
\int_0^{a/t} e^{ibt}\hat\rho(ib)\,db-\sum_{j\ge0} c_j' t^{-(j+1)(1-\beta)}
\ll t^{-(1-\beta)}a^{-\beta} +(a/t)^{1-r},
\]
where $r=\begin{cases}2\beta-q & q<2\beta \\ \eps & q=2\beta\end{cases}$, and
$c_0=c^{-1}c_\beta^{-1}$, $c_\beta=i\int_0^\infty e^{-i\sigma}\sigma^{-\beta}\,d\sigma=\Gamma(1-\beta)e^{i\beta\pi/2}$.

By Lemmas~\ref{lem:DT} and~\ref{lem:infty}, 
\begin{align*}
\int_0^\infty e^{ibt}\hat \rho(ib)\,db-\sum_{j\ge0} c_j'  t^{-(j+1)(1-\beta)} 
& \ll 
(a/t)^{1-r}+ t^{-(1-\beta)}a^{-(2\beta-1)} + t^{-(\beta-\eps)} 
\\ & = t^{-(1-r)\beta/(2\beta-r)}+  t^{-(\beta-\eps)}
\ll t^{-\kappa},
\end{align*}
where we have taken
$a=t^{(\beta-r)/(2\beta-r)}$ and $\kappa$ is as in the statement of the theorem.

The result now follows from Proposition~\ref{prop:0cx} with
\[
\textstyle d_1=\frac{1}{\pi}\Re(c_0C_0)=\frac{1}{c\pi}\Re(ie^{-i\beta\pi})
=\frac{1}{c\pi}\sin\beta\pi,
\]
as desired.
\end{pfof}

\section{Estimates for small $b$: proof of Lemma~\ref{lem:small}}
\label{sec:small}

Throughout this section, we suppose that hypothesis 
(H1) holds and that 
$\mu(\tau>t)=ct^{-\beta}$ where $c>0$, $\beta\in(\frac12,1)$
and $q\in(1,2\beta]$.

\subsection{Representation of $\hat\rho_{v,w}(ib)$ for small $b$}

\paragraph{Transfer operators}
Recall that $R:L^1(Y)\to L^1(Y)$ denotes the transfer operator for $F:Y\to Y$.
Also, for $s\in\barH$, we have the families of operators 
$\hat R(s)v=R(e^{-s\tau}v)$
on $L^1(Y)$,
Note that $\hat R$ is analytic on $\H$ and well-defined on $\barH$.

Given observables $v,w:\tY\to\R$, we define
$v_s,w_s:Y\to\C$ for $s\in\C$, setting
\[
v_s(y)=\int_0^1 e^{su}v(y,u)\,du, \quad
w_s(y)=\int_0^1 e^{-su}w(y,u)\,du.
\]
Also, define
\[
\hat J(s)=-\int_Y\int_0^1\int_0^u e^{su}v(y,u)\, e^{-st}w(y,t) \,dt\,du\,d\mu.
\]

Define $\hat T(s)=(I-\hat R(s))^{-1}$ for $s\in\H$.
Following~\cite{Pollicott85},
\begin{align} \label{eq:rhohat}
\textstyle \hat\rho_{v,w}(s)=\hat J(s)+\int_Y \hat T(s)v_s\, w_s\,d\mu,
\end{align}
for all $v\in L^1(Y)$, $w\in L^\infty(Y)$, 
$s\in\H$.
(See Appendix~\ref{sec:rhohat} for a proof.)

\begin{prop}  \label{prop:v}
$\|v_s\|_{\cB(Y)}\le e^{|\Re s|}\|v\|_{\cB(\tY)}$, for
all $s\in \C$, and 
$\|v_{i(b+h)}-v_{ib}\|_{\cB(Y)}\le h\|v\|_{\cB(\tY)}$
for all $b,h\ge0$.

The same result holds with $\cB$ changed to $L^p$, $1\le p\le \infty$,
and/or $v$ changed to $w$.
\end{prop}

\begin{proof}
We have
$\|v_s\|_{\cB(Y)}\le \int_0^1 e^{\Re s\, u} \|v(\cdot,u)\|_{\cB(Y)}\,du
\le e^{|\Re s|} \|v\|_{\cB(\tY)}$,
and 
$\|v_{i(b+h)}-v_{ib}\|_{\cB(Y)}\le \int_0^1 |e^{i(b+h)u}-e^{ibu}| \|v(\cdot,u)\|_{\cB(Y)}\,du
\le h\int_0^1 u\|v\|_{\cB(\tY)}\,du
\le h \|v\|_{\cB(\tY)}$.
\end{proof}

\begin{prop} \label{prop:hatJ}
\begin{itemize}
\item[(a)]
$|\hat J(s)|\le e^{|\Re s|}|v|_{L^1(\tY)} |w|_{L^\infty(\tY)}$ for all $s\in \C$,
$v\in L^1(\tY)$, $w\in L^\infty(\tY)$, and

\item[(b)] $|\hat J((i(b+h))-\hat J(ib)|\le h|v|_{L^1(\tY)}|w|_{L^\infty(\tY)}$
for all $b,h>0$,
$v\in L^1(\tY)$, $w\in L^\infty(\tY)$.
\end{itemize}
\end{prop}

\begin{proof}
\begin{align*}
|\hat J(s)| & \le 
\int_Y\int_0^1\int_0^1 e^{\Re s\,(u-t)}|v(y,u)||w(y,t)| \,dt\,du\,d\mu
\\ & \le
e^{|\Re s|}|w|_\infty \int_Y\int_0^1 |v(y,u)|\,du\,d\mu
= e^{|\Re s|}|v|_1|w|_\infty.
\end{align*}
Similarly,
\begin{align*}
|\hat J(i(b+h))-\hat J(ib)| & \le h\int_Y\int_0^1\int_0^1 |u-t||v(y,u)||w(y,t)| \,dt\,du\,d\mu  \le h|v|_1|w|_\infty.
\end{align*}
\end{proof}

\subsection{Estimates for $\hat T(s)$}
\label{sec:T}

Viewing $\hat R(s)$, $s\in\barH$, as a family of operators from $\cB(Y)$ to $L^\infty(Y)$, we first study its continuity properties using (H1).
We begin with $\eps,\delta$ small and $p$ large as in (H1).
During the subsection, these values change finitely many times.
Also $C>0$ is a constant whose value changes finitely many times.

\begin{lemma}\label{lem:continf} 
$\|\hat R(s_1)-\hat R(s_2)\|_{\cB(Y)\to L^p(Y)}\le C\, |s_1-s_2|^{\beta-\eps}$
for all $s_1,s_2\in\barH$.
\end{lemma}

\begin{proof}
 Recall that $\hat R(s)v=R(e^{-s\tau}v)$.  Since $R$ is a positive operator,
\begin{align*}
|(\hat R(s_1) -\hat R(s_2))v| 
& \le  R(|e^{-s_1\tau}-e^{-s_2\tau}||v|)
  \le  2|s_1-s_2|^{\beta-\eps}R(\tau^{\beta-\eps} |v|).
\end{align*}
By (H1)(iii), $|(\hat R(s_1) -\hat R(s_2))v|_p\le 2|s_1-s_2|^{\beta-\eps}
|R(\tau^{\beta-\eps} |v|)|_p\ll |s_1-s_2|^{\beta-\eps} \|v\|_{\cB(Y)}$.~
\end{proof}

\begin{lemma}\label{lem:estKL} 
There exists a continuous family $\lambda(s)$, $s\in\barH\cap B_\delta(0)$, of simple eigenvalues for $\hat R(s):\cB(Y)\to\cB(Y)$ with $\lambda(0)=1$.
The corresponding family of spectral projections $P(s)$ are bounded linear operators on $\cB(Y)$ for all
$s\in\barH\cap B_\delta(0)$ and $\sup_{s\in\barH\cap B_\delta(0)}\|P(s)\|_{\cB(Y)}<\infty$.
Moreover, 
\begin{align*}
\|P(s_1)-P(s_2)\|_{\cB(Y)\to L^p(Y)}
\le C |s_1-s_2|^{\beta-\eps}
\quad\text{for all $s_1,s_2\in\barH\cap B_\delta(0)$.}
\end{align*}
\end{lemma}

\begin{proof}
We verify the hypotheses~(2)--(5) of~\cite[Corollary 1]{KellerLiverani99}, thereby obtaining the required estimates for the family $P(s)$.
Simplicity of the family of eigenvalues $\lambda(s)$ is a consequence of
$F$ being mixing.

Hypothesis~(2) is immediate since $\|\hat R(s)\|_{L^p(Y)}\le1$ for all $s\in\barH\cap B_\delta(0)$ and assumption (H1)(ii)
corresponds to hypothesis~(3). 
Hypothesis~(4) follows from (H1)(i),(ii), and hypothesis~(5) follows from
Lemma~\ref{lem:continf}.  
\end{proof}

Let $\zeta(s)$ denote the corresponding family of
eigenfunctions normalized so that $\int_Y \zeta(s)\,d\mu=1$.
In particular, $\zeta(0)\equiv1$ and $P(0)v=\int_Y v\,d\mu$ for all $v\in \cB(Y)$.
Also, define the complementary projections $Q(s)=I-P(s)$.
It is immediate that $\zeta(s)$ and $Q(s)$ inherit the estimate obtained for
$P(s)$.  In particular $|\zeta(s_1)-\zeta(s_2)|_{L^p(Y)}\ll |s_1-s_2|^{\beta-\eps}$.

Following~\cite{Gouezel10b} (a simplification of~\cite{AaronsonDenker01}), 
\begin{align}\label{eq:ev-Gou}
\textstyle \lambda(s) &  =\int_{Y}\lambda(s)\zeta(s)\,d\mu
=\int_{Y}\hat R(s)\zeta(s)\,d\mu
 \\ & =\int_{Y}\hat R(s)\zeta(0)\,d\mu+\int_Y(\hat R(s)-\hat R(0))(\zeta(s)-\zeta(0))\,d\mu
=\int_{Y}e^{-s\tau}\,d\mu +\chi(s), \nonumber
\end{align}
 where 
$\chi(s) =\int_Y (e^{-s\tau}-1)(\zeta(s)-1)\,d\mu$.

\begin{prop} \label{prop:chi}
\begin{itemize}
\item[(a)] $|\chi(s)|\le C|s|^{2\beta-\eps}$ for all $s\in\barH\cap B_\delta(0)$,
\item[(b)] $|\chi(i(b+h))-\chi(ib)|\le C b^\beta h^{\beta-\eps}$ for all $0<h<b<\delta$.
\end{itemize}
\end{prop}

\begin{proof}
Choose $r>1$ such that $(\beta-\eps)r<\beta$ with conjugate exponent $r'$.
Then $\tau^{(\beta-\eps)r}\in L^1$ and it follows from
H\"older's inequality that 
\begin{align*}
\textstyle |\chi(s)| \le 2|s|^{\beta-\eps}| \tau^{\beta-\eps}(\zeta(s)-1)|_1
\le 2|s|^{\beta-\eps}|\tau^{(\beta-\eps)}|_r|\zeta(s)-1|_{r'}
\ll |s|^{2(\beta-\eps)},
\end{align*}
yielding part~(a).   Here we used that $|\zeta(s)-1|_p=O(|s|^{\beta-\eps})$ for $p$ as large as desired.
Similarly,
\begin{align*}
|\chi(i(b+h))-\chi(ib)|
& \le |(e^{i(b+h)\tau}-1)(\zeta(i(b+h))-\zeta(ib))|_1
+ |(e^{ih\tau}-1)(\zeta(ib)-1)|_1
\\ & \ll (b+h)^{\beta-\eps}h^{\beta-\eps}+h^{\beta-\eps}b^{\beta-\eps}
 \ll b^{\beta-\eps}h^{\beta-\eps}
 \le b^{\beta}h^{\beta-2\eps},
\end{align*}
proving part~(b).
\end{proof}

Recall that $c_\beta=i\int_0^\infty e^{-i\sigma}\sigma^{-\beta}\, d\sigma$.

\begin{lemma}[{{cf.~\cite[Lemma 5.5]{MT17}}}]
\label{lem:T} 
For $s\in\barH\cap B_\delta(0)$, 
 \begin{align*}
1-\lambda(s)\sim c_\beta s^\beta\enspace\text{as $s\to0$}, \qquad
  \hat T(s)  =(1-\lambda(s))^{-1}P(0) +E(s),
 \end{align*}
where  
$\|E(s)\|_{\cB(Y)\to L^1(Y)}\le C|s|^{-\eps}$.
\end{lemma}

 \begin{proof}
For similar arguments we refer to, for instance,~\cite{AaronsonDenker01,LT,MT12, MT13, MT17}.

By~\cite[Lemma~2.4]{MT13},
 $1-\int_Y e^{-s\tau}\,d\mu\sim c_\beta s^{\beta}$ as $s\to 0$.
It follows from~\eqref{eq:ev-Gou} and Proposition~\ref{prop:chi} that
$1-\lambda(s)\sim  c_\beta s^{\beta}$.
Next, 
\begin{align*}
\hat T(s) & =(1-\lambda(s))^{-1}P(s)+(I-\hat R(s))^{-1}Q(s)
 =(1-\lambda(s))^{-1}P(0)+E(s),
\end{align*}
where
\[
 E(s)= (1-\lambda(s))^{-1}(P(s)-P(0))+(I-\hat R(s))^{-1}Q(s).
\]
By (H1), $\|(I-\hat R(s))^{-1}Q(s)\|_{\cB(Y)}=O(1)$.
By  Lemma~\ref{lem:estKL}, $\|P(s)-P(0)\|_{\cB(Y)\to L^1(Y)}=O(|s|^{\beta-\eps})$.  Hence
$\|E(s)\|_{\cB(Y)\to L^1(Y)}=O(|s|^{-\eps})$. 
\end{proof}

\begin{lemma}[{{cf.~\cite[Proposition 3.7]{LT}}}]
\label{lem:Tcont} 
 For all  $0<h< b<\delta$,
  \begin{align*}
 \|\hat T(i(b+h)) & -\hat T(ib)\|_{\cB(Y)\to L^1(Y)}
\le C( b^{-2\beta}h^\beta + b^{-\beta} h^{\beta-\eps}).
\end{align*}
\end{lemma}

\begin{proof} 
First,  $|\int_Y (e^{-i(b+h)\tau}-e^{-ib\tau})d\mu|
\ll h^\beta$ by the argument used
in the proof of~\cite[Lemma 3.3.2]{GarsiaLamperti62}.
By~\eqref{eq:ev-Gou} and Proposition~\ref{prop:chi}(b), 
\begin{align} \label{eq:Dlambda}
|\lambda(i(b+h))-\lambda(ib)| \ll 
h^\beta +b^\beta h^{\beta-\eps}.
\end{align}

Next recall as in Lemma~\ref{lem:T} that $\hat T(ib)=A_1(b)+A_2(b)$ where
\[
A_1(b)=(1-\lambda(ib))^{-1}P(ib),
\qquad A_2(b)=(I-\hat R(ib))^{-1}Q(ib).
\]
Using~\eqref{eq:Dlambda} and Lemmas~\ref{lem:estKL} and~\ref{lem:T}, 
\[
\|A_1(b+h)-A_1(b)\|_{\cB(Y)\to L^1(Y)}\ll 
b^{-2\beta}h^\beta + b^{-\beta}h^{\beta-\eps}.
\]
The argument in~\cite[Proposition 3.8]{LT} shows that
$\|A_2(b+h)-A_2(b)\|_{\cB(Y)\to L^1(Y)}\ll h^{\beta-\eps}$, completing the proof.
\end{proof}

\begin{lemma}
\label{lem:T-higherorder}
There are constants $c_j\in\C$ with $c_0=c^{-1}c_\beta^{-1}$ such that
 \begin{align*}
  \hat T(ib) & =\textstyle \sum_jc_j b^{-((j+1)\beta-j)}P(0) +E(ib)
\quad\text{for all $b\in [0,\delta)$}, 
 \end{align*}
where  
$\|E(ib)\|_{\cB(Y)\to L^1(Y)}\le C(b^{-(2\beta-q)}+b^{-\eps})$.
\end{lemma}

\begin{proof} 
Let $G(t)=\mu(\tau\le t)$, so $1-G(t)=ct^{-\beta}+H(t)$ where
$H(t)=O((t+1)^{-q})$.  Then proceeding as in~\cite[Section~5]{AaronsonDenker01} and~\cite[Lemma~3.2]{MT12},
\begin{align*}
\int_Y(e^{ib\tau}-1)\,d\mu & 
 = -\int_0^\infty(e^{ibt}-1)\,d(1-G(t))
 = ib\int_0^\infty e^{ibt}(1-G(t))\,dt
\\ & = icb\int_0^\infty e^{ibt}t^{-\beta}\,dt + 
ib\int_0^\infty H(t)\,dt
+ib\int_0^\infty (e^{ibt}-1)H(t)\,dt
\\ & = icb^\beta\int_0^\infty e^{i\sigma}\sigma^{-\beta}\,d\sigma + ib\int_0^\infty H(t)\,dt+O(b^q).
\end{align*}
Hence
\begin{align} \label{eq:expand}
\textstyle 1-\int_Ye^{-ib\tau}\,d\mu=cc_\beta b^\beta+e_1b+O(b^q),
\end{align}
where $c_\beta=i\int_0^\infty e^{-i\sigma}\sigma^{-\beta}\,d\sigma$ and
$e_1=i\int_0^\infty H(t)\,dt$.
By~\eqref{eq:ev-Gou} and Proposition~\ref{prop:chi}(a),
\begin{align*}
1-\lambda(ib)
& \textstyle
=1-\int_Y e^{-ib\tau}\,d\mu+ O(b^{2\beta-\eps})
\\ &
=cc_\beta b^\beta(1+e_2 b^{1-\beta}+O(b^{q-\beta})+O(b^{\beta-\eps})),
\end{align*}
where $e_2=c^{-1}c_\beta^{-1}e_1$.  Thus,
\[
 (1-\lambda(ib))^{-1}={\textstyle\sum}_j c_j b^{-((j+1)\beta-j)}+ O(b^{-(2\beta-q)})+O(b^{-\eps}),
\]
for constants $c_0,c_1,\ldots\in\C$ with $c_0=c^{-1}c_\beta^{-1}$.
Now apply Lemma~\ref{lem:T}.
\end{proof}

We can now complete the proof of Lemma~\ref{lem:small}.

\begin{pfof}{Lemma~\ref{lem:small}}
\noindent
(a)  By~\eqref{eq:rhohat}, Propositions~\ref{prop:v} and~\ref{prop:hatJ}, and Lemma~\ref{lem:T},
\begin{align*}
|\hat \rho(s)| & \le |\hat J(s)|
+ \|\hat T(s)\|_{\cB(Y)\to L^1(Y)}\|v_s\|_{\cB(Y)}
|w_s|_{L^\infty(Y)}
 \ll |s|^{-\beta} \|v\|_{\cB(\tY)}|w|_{L^\infty(\tY)}.
\end{align*}

\vspace{1ex}
\noindent
(b) 
By~\eqref{eq:rhohat}, Propositions~\ref{prop:v} and~\ref{prop:hatJ}, and Lemmas~\ref{lem:T} and~\ref{lem:Tcont}, 
\begin{align*}
& |\hat\rho  (i(b+h))  -\hat\rho(ib)| \le |\hat J(i(b+h))-\hat J(ib)|
\\ & \qquad\qquad
+ \|\hat T(i(b+h))-\hat T(ib)\|_{\cB(Y)\to L^1(Y)}\|v_{i(b+h)}\|_{\cB(Y)}|w_{i(b+h)}|_{L^\infty(Y)}
\\ & \qquad\qquad
+ \|\hat T(ib)\|_{\cB(Y)\to L^1(Y)}\|v_{i(b+h)}-v_{ib}\|_{\cB(Y)}|w_{i(b+h)}|_{L^\infty(Y)}
\\ & \qquad\qquad
+ \|\hat T(ib)\|_{\cB(Y)\to L^1(Y)}\|v_{ib}\|_{\cB(Y)}|w_{i(b+h)}-w_{ib}|_{L^\infty(Y)}
\\[.75ex] & 
\ll 
 \big\{h+b^{-2\beta}h^\beta+b^{-\beta}h^{\beta-\eps}+
 b^{-\beta}h\}\|v\|_{\cB(\tY)}|w|_{L^\infty(\tY)},
\end{align*}
yielding the required estimate.

\vspace{1ex}
\noindent
(c) By Propositions~\ref{prop:v} and~\ref{prop:hatJ}, and Lemma~\ref{lem:T},
$|\hat J(ib)|\ll |v|_{L^1(\tY)}|w|_{L^\infty(\tY)}$ and
\begin{align*}
 |\hat T(ib)v_{ib} & \,w_{ib} - \hat T(ib)v_0\,w_0|_{L^1(Y)}  
\le 
|\hat T(ib)(v_{ib}-v_0)\,w_{ib}|_{L^1(Y)} 
+ |\hat T(ib)v_0\,(w_{ib}-w_0)|_{L^1(Y)} 
\\ & \le
 \|\hat T(ib)\|_{\cB(Y)\to L^1(Y)}\{\|v_{ib}-v_0\|_{\cB(Y)}|w_{ib}|_{L^\infty(Y)}
+\|v_0\|_{\cB(Y)}|w_{ib}-w_0|_{L^\infty(Y)}\}
\\ & \ll b^{1-\beta}\|v\|_{\cB(\tY)}|w|_{L^\infty(\tY)}.
\end{align*}
Substituting into~\eqref{eq:rhohat}, we obtain 
\begin{align*} 
\textstyle \hat\rho(ib)  =\int_Y \hat T(ib)v_0\,w_0\,d\mu
+O(\|v\|_{\cB(\tY)}|w|_{L^\infty(\tY)}).
\end{align*}

Let $r=\max\{2\beta-q,\eps\}\in[0,1)$.
By Lemma~\ref{lem:T-higherorder},
\begin{align*}
\int_Y \hat T(ib)v_0\,w_0\,d\mu
& = \sum_jc_j b^{-((j+1)\beta-j)}\int_Y P(0)v_0\,w_0\,d\mu+\int_Y E(ib) v_0\,w_0\,d\mu,
\end{align*}
where 
$\int_Y |E(ib)v_0 \,w_0|\,d\mu\ll b^{-r}\|v\|_{\cB(\tY)} |w|_{L^\infty(\tY)}$.
Also, $\int_Y P(0)v_0\,w_0\,d\mu=\int_{Y^\tau}v\,d\mu^\tau \int_{Y^\tau}w\,d\mu^\tau$.
Hence
\[
\hat\rho_{v,w}(ib)= \sum_jc_j b^{-((j+1)\beta-j)}\int_{\tY} v\,d\tilde\mu\int_{\tY}w\,d\tilde\mu+ O(b^{-r}\|v\|_{\cB(\tY)}|w|_{L^\infty(\tY)}),
\]
and the result follows.
\end{pfof}

\section{Estimates for large $b$: proof of Lemma~\ref{lem:large}}
\label{sec:large}

Throughout this section, we suppose that assumptions (H2)--(H4) hold.

\subsection{Lifting from $Y^\tau$ to $Z^\varphi$}

Recall that $F_t:Y^\tau\to Y^\tau$ is the suspension semiflow over $F$ with roof function $\tau$ and
$G_t:Z^\varphi\to Z^\varphi$ is the suspension semiflow over 
$G=F^\sigma$ with roof function $\varphi=\tau_\sigma$.  
Here, $G:Z\to Z$ is a Gibbs-Markov map with partition $\alpha$.
Also $\pi:Z^\varphi\to Y^\tau$ is the semiconjugacy $\pi(z,u)=F_u(z,0)$.

We defined ergodic $F_t$- and $G_t$-invariant measures
$\mu^\tau=\mu\times{\rm Lebesgue}$ on $Y^\tau$
and $\mu_Z^\varphi=(\mu_Z\times{\rm Lebesgue})/\int_Z\sigma\,d\mu_Z$ on $Z^\varphi$.  As promised, we now verify that
$\pi_*\mu_Z^\varphi=\mu^\tau$.

\begin{pfof}{Proposition~\ref{prop:mu}}
The measures $\mu_Z$ and $\mu$ are related by the formula
\[
\int_Y h\,d\mu = \bar\sigma^{-1} \int_Z\sum_{\ell=0}^{\sigma(z)-1}h(F^\ell z)\,d\mu_Z(z)\quad\text{for all $h\in L^1(Y)$},
\]
where $\bar\sigma=\int_Z\sigma\,d\mu_Z$.
For $g\in L^1(Y^\tau)$,
\begin{align*}
& \bar\sigma  \int_{Y^\tau}g\,d\mu^\tau  =
\bar\sigma\int_Y\int_0^{\tau(y)}g(y,u)\,du\,d\mu
  = \int_Z\sum_{\ell=0}^{\sigma(z)-1}\int_0^{\tau(F^\ell z)}\!\!g(F^\ell z,u)\,du\,d\mu_Z(z)
\\ & = \int_Z\sum_{\ell=0}^{\sigma(z)-1}\int_{\tau_\ell(z)}^{\tau_{\ell+1}(z)}\!\!g(F^\ell z,u-\tau_\ell(z))\,du\,d\mu_Z(z)
\\ & 
= \int_Z\sum_{\ell=0}^{\sigma(z)-1}\int_{\tau_\ell(z)}^{\tau_{\ell+1}(z)}\!\!g(F_u(z,0))\,du\,d\mu_Z(z)
= \int_Z\int_0^{\varphi(z)}g(F_u(z,0))\,du\,d\mu_Z(z)
\\ & 
= \int_Z\int_0^{\varphi(z)}g\circ\pi(z,u)\,du\,d\mu_Z(z)
=\bar\sigma \int_{Z^\varphi}g\circ\pi\,d\mu_Z^\varphi
=\bar\sigma \int_{Y^\varphi}g\,d\pi_*\mu_Z^\varphi.
\end{align*}
Hence $\pi_*\mu_Z^\varphi=\mu^\tau$ as required.
\end{pfof}

Define 
\[
\hY=\pi^{-1}(\tY)=\bigcup_{z\in Z}\bigcup_{\ell=0}^{\sigma(z)-1}\{z\}\times[\tau_\ell(z),\tau_\ell(z)+1].
\]
We note that
\begin{align} \label{eq:hatY}
\textstyle \int_0^{\varphi(z)}1_{\hY}(z,u)\,du=\sigma(a) \quad\text{
for all $a\in\alpha$, $z\in a$}.
\end{align}

Observables $v:Y^\tau\to\R$ supported in $\tY$ lift to
observables $\hat v=v\circ\pi:Z^\varphi\to\R$ supported in $\hY$.
Define $\hat v(b):Z\to\C$, $b\in\R$,
\[
\textstyle \hat v(b)(z)=\int_0^{\varphi(z)}e^{ibu}\hat v(z,u)\,du
=\int_0^{\varphi(z)}e^{ibu}(1_{\hY}\hat v)(z,u)\,du.
\]
In this paper, we require estimates for $\hat v$ and $\hat v'$, but it is no extra work to estimate $\hat v^{(k)}$ for all $k\ge0$; this will be used elsewhere.

\begin{prop} \label{prop:vb}
Let $k\ge0$, $\eps>0$. There exists $C>0$ such that
\begin{align*}
& |\hat v^{(k)}(b)(z)|  \le C\sigma(a)(\infa\varphi^k)|v|_\infty,\\
& |\hat v^{(k)}(b)(z)-\hat v^{(k)}(b)(z')|  \le C (b^\eps+1)\sigma(a)(\infa\varphi^{k+\eps})\|v\|_{C^\eta} d_Z(Gz,Gz')^{\eps\eta},
\end{align*}
for all $v\in C^\eta(\tY)$,
$a\in\alpha$, $z,z'\in\alpha$, $b>0$.
\end{prop}

\begin{proof}
By~\eqref{eq:hatY},
$|\hat v^{(k)}(b)(z)|\le |v|_\infty\int_0^{\varphi(z)}u^k1_{\hY}(z,u)\,du\le 
\sigma(a)(\supa\varphi^k)|v|_\infty\ll
\sigma(a)(\infa\varphi^k)|v|_\infty$.
Next,
\begin{align*}
\hat v^{(k)}(b)(z) &
 =
\sum_{\ell=0}^{\sigma(a)-1} \int_{\tau_\ell(z)}^{\tau_\ell(z)+1}(iu)^ke^{ibu}\hat v(z,u)\,du
  \\ & =
i^k\sum_{\ell=0}^{\sigma(a)-1} \int_0^1 (u+\tau_\ell(z))^ke^{ibu}e^{ib\tau_\ell(z)}\hat v(z,u+\tau_\ell(z))\,du
  =
i^k\sum_{\ell=0}^{\sigma(a)-1} I_\ell(z),
\end{align*}
where $I_\ell(z)= \int_0^1 (u+\tau_\ell(z))^ke^{ibu}e^{ib\tau_\ell(z)}v(F^\ell z,u)\,du$.
By (H2)(iv) and (H3),
\begin{align*}
|I_\ell(z)  -I_\ell(z')|  & \le k|\tau_\ell(z)-\tau_\ell(z')|(\supa\varphi)^{k-1}|v|_\infty
\\ & \!\!\!\!\!\!\!\!\!\! + 2b^\eps|\tau_\ell(z)-\tau_\ell(z')|^\eps(\supa\varphi+1)^k|v|_\infty
+(\supa\varphi+1)^k|v|_\eta d_Y(F^\ell z,F^\ell z')^\eta
\\[.75ex] & \ll 
(b^\eps+1)(\infa\varphi^{k+\eps})\|v\|_{C^\eta(\tY)} d_Z(Gz,Gz')^{\eps\eta},
\end{align*}
for $0\le \ell\le\sigma(a)-1$.
The estimate for $|\hat v^{(k)}(b)(z)-\hat v^{(k)}(b)(z')|$ follows.
\end{proof}

For $v:\tZ\to\C$, $\eta\in(0,1)$, define
\[
\|v\|_{C^\eta(\tZ)}=|v|_\infty+|v|_{\eta}, \quad
|v|_{\eta}=\sup_{z,z'\in Z,\, z\neq z'}\sup_{u\in[0,1]}|v(z,u)-v(z',u)|/d_Z(z,z')^\eta.
\]
Define $C^\eta(\tZ)$ to be the space of 
observables $v:\tZ\to\C$ such that $\|v\|_{C^\eta(\tZ)}<\infty$.

\subsection{Representation of $\hat\rho_{v,w}(ib)$ for large $b$}
\label{sec:rhoG}

In this subsection we obtain an expression for $\hat\rho_{v,w}(ib)$ using transfer operators related to the induced suspension semiflow $G_t:Z^\varphi\to Z^\varphi$ over $G=F^\sigma$ with roof function $\varphi=\tau_\sigma$.

Let $L_t^G:L^1(Z^\varphi)\to L^1(Z^\varphi)$ denote the family of transfer operators corresponding to the suspension semiflow $G_t:Z^\varphi\to Z^\varphi$.
Define $\hat L^G(s)= \int_0^\infty L_t^G e^{-st}\,dt$ and note that for $v\in L^1(\tY)$
and $w\in L^\infty(\tY)$,
\begin{align}
\label{eq:hatrholarge}
\hat\rho_{v,w}(s)=\int_{Z^\varphi}1_{\hY}\hat L^G(s)(1_{\hY}\hat v)\,\hat w\,d\mu_Z^\varphi \quad\text{where $\hat v=v\circ\pi$, $\hat w=w\circ\pi$.}
\end{align}

Let  $\tZ=Z\times[0,1]$.
Define 
$T_t^G=1_{\tZ}L^G_t1_{\tZ}$
and let $\hat T^G(s)$ be the corresponding  Laplace transform.
Condition (H4) is required for the next result.

\begin{prop} \label{prop:TG}
Let $\eps\in(0,1)$, $\delta>0$.  Then there exists $C>0$, $\omega>0$ such that
\[
\| \hat T^G(ib)\|_{C^{\eps\eta}(\tZ)\to L^\infty(\tZ)}
\le Cb^\omega, \quad
\|\hat T^G(i(b+h))-\hat T^G(ib)\|_{C^{\eps\eta}(\tZ)\to L^\infty(\tZ)}\le C b^\omega h^{\beta-\eps},
\]
for all $0<h<\delta<b$.
\end{prop}

\begin{proof}  Since $G$ is Gibbs-Markov, the result with $L^\infty(\tZ)$ replaced by $L^1(\tZ)$ is covered by~\cite[Corollary~5.10]{MT17}.
Using~\eqref{eq:GM}, an additional but standard calculation 
(for the term $\hat U$ in~\cite[Lemma~5.8(b)]{MT17}) 
specialised to the Gibbs-Markov setting
yields the $L^\infty$ estimates.~
\end{proof}

In a continuous-time analogue of~\cite{GouezelPhD}, we define 
the operators 
\[
A_t:L^1(\tZ)\to L^1(Z^\varphi),\quad
B_t:L^1(Z^\varphi)\to L^1(\tZ), \quad
E_t:L^1(Z^\varphi)\to L^1(Z^\varphi),
\]
where 
\[
A_t=1_{Z_t}L_t^G1_{\tZ}, \quad B_t=1_{\tZ}L_t^G1_{\Delta_t},
\quad 
E_t=L_t^G((1-\xi(t))v).
\]
Here 
\begin{align*}
Z_t & =\{(z,u)\in Z^\varphi:t\le u\le t+1,\,u>1\}, \\
\Delta_t & =\{(z,u)\in Z^\varphi:\varphi(z)-t\le u \le \varphi(z)-t+1\}, 
\end{align*} 
correspond to certain rows and diagonals in the suspension $Z^\varphi$ and
\[
\xi(t)=\int_0^t\int_0^x 1_{\Delta_y}\,1_{Z_{t-x}}\circ G_t \,dy\,dx\in[0,1].
\]

Given Banach spaces $\cB_1$, $\cB_2$, $\cB_3$ and families of operators
$\alpha_t:\cB_2\to\cB_3$, $\beta_t:\cB_1\to\cB_2$, $t\ge0$,
we define the convolution $(\alpha\star\beta)_t:\cB_1\to\cB_3$, $t\ge0$, by
$(\alpha\star\beta)_t=\int_0^t \alpha_{t-x}\beta_x\,dx$.

\begin{prop} \label{prop:ABEt}
$L_t^G=(A\star T^G\star B)_t+E_t$, for all $t\ge0$.
\end{prop}

\begin{proof}
Let $v\in L^1(Z^\varphi)$, $w\in L^\infty(Z^\varphi)$.
For all $0\le y\le x\le t$,
\begin{align*}
& \int_{Z^\varphi}A_{t-x}T_{x-y}^GB_yv\,w\,d\mu_Z^\varphi
 =
 \int_{Z^\varphi}1_{Z_{t-x}}L^G_{t-x}(1_{\tZ}L_{x-y}^G(1_{\tZ}L^G_y(1_{\Delta_y}v)))\,w\,d\mu_Z^\varphi
\\ & \qquad  =
 \int_{Z^\varphi}1_{Z_{t-x}}L_{t-x}^G(L_{x-y}^G(L^G_y(1_{\Delta_y}v)))\,w\,d\mu_Z^\varphi
 =
 \int_{Z^\varphi}1_{Z_{t-x}}L_t^G(1_{\Delta_y}v)\,w\,d\mu_Z^\varphi
\\ & \qquad  =
 \int_{Z^\varphi}1_{\Delta_y}v\,1_{Z_{t-x}}\circ G_t\, w\circ G_t\,d\mu_Z^\varphi.
\end{align*}
Hence
\begin{align*}
& \int_{Z^\varphi}(A\star T^G\star B)_tv\,w\,d\mu_Z^\varphi
 =\int_{Z^\varphi}\int_0^t\int_0^x A_{t-x}T_{x-y}^GB_yv\,w\,dy\,dx\,d\mu_Z^\varphi
 \\ & =\int_{Z^\varphi}\xi(t)v\,w\circ G_t\,d\mu_Z^\varphi
=\int_{Z^\varphi}L_t^G(\xi(t))v\,w\,d\mu_Z^\varphi
=\int_{Z^\varphi}L_t^Gv\,w\,d\mu_Z^\varphi-
\int_{Z^\varphi}E_tv\,w\,d\mu_Z^\varphi
\end{align*}
as required.
\end{proof}

Let   $\hat A, \hat B, \hat E$ be the Laplace transforms of $A_t, B_t, E_t$. 
By Proposition~\ref{prop:ABEt}, 
\begin{align} \label{eq:ABEt}
\hat L^G(s)=\hat A(s)\hat T^G(s)\hat B(s)+\hat E(s).
\end{align}

Let $\hat B(ib)1_{\hY}\pi^*:L^1(\tY)\to L^1(\tZ)$ denote the operator 
$(\hat B(ib)1_{\hY}\pi^*)(v)=\hat B(ib)(1_{\hY}\hat v)$ where
$\hat v=\pi^*v=v\circ \pi:\hY\to\R$.
Similarly, define 
$1_{\hY}\hat E(ib)1_{\hY}\pi^*:L^1(\tY)\to L^1(\hY)$.
The following result is proved in Subsections~\ref{sec:CE} and~\ref{sec:B}.

\begin{lemma} \label{lem:ABEt}
Let $\eps\in(0,1)$, $\delta>0$.  Then there exists $C>0$ such that
\begin{itemize}
\item[(a)] $\int_{t_0}^\infty \|1_{\hY}A_t\|_{L^\infty(\tZ)\to L^1(\hY)}\,dt\le C {t_0}^{-(\beta-\eps)}$,
\item[(b)] $\int_{t_0}^\infty \|1_{\hY}E_t1_{\hY}\pi^*\|_{L^\infty(\tY)\to L^1(\hY)}\,dt\le C t_0^{-(\beta-\eps)}$,
\item[(c)] $\| \hat B(ib)1_{\hY}\pi^*\|_{C^\eta(\tY)\to C^{\eps\eta}(\tZ)}
\le Cb^\eps$, and 
\newline $\|\hat B(i(b+h))1_{\hY}\pi^*-\hat B(ib)1_{\hY}\pi^*\|_{C^\eta(\tY)\to C^{\eps\eta}(\tZ)}\le C b^\eps h^{\beta-\eps}$,
\end{itemize}
for all $t_0>0$, $0<h<\delta<b$.
\end{lemma}

\begin{pfof}{Lemma~\ref{lem:large}}
It follows from Lemma~\ref{lem:ABEt} that in the appropriate norms,
\begin{alignat*}{2}
& \|1_{\hY}\hat A(ib)\|\le C, \quad
& &\|1_{\hY}\hat A(i(b+h))-1_{\hY}\hat A(ib)\|\le C h^{\beta-\eps}, \\
& \|1_{\hY}\hat E(ib)1_{\hY}\pi^*\|\le C, \quad
&& \|1_{\hY}\hat E(i(b+h))1_{\hY}\pi^*-1_{\hY}\hat E(ib)1_{\hY}\pi^*\|\le C h^{\beta-\eps}, \\
& \|\hat B(ib)1_{\hY}\pi^*\|\le Cb^\eps, \quad
& &\|\hat B(i(b+h))1_{\hY}\pi^*-\hat B(ib)1_{\hY}\pi^*\|\le Cb^\eps h^{\beta-2\eps}.
\end{alignat*}
Combining these with the estimates for $\hat T^G$ in Proposition~\ref{prop:TG} and
substituting into equation~\eqref{eq:ABEt}, there exist (new) constants $C,\omega>0$ such that
\[
\|1_{\hY}\hat L^G(i(b+h))1_{\hY}\pi^*-1_{\hY}\hat L^G(ib))1_{\hY}\pi^*\|_{C^\eta(\tY)\to L^1(\hY)}\le Cb^\omega h^{\beta-2\eps}.
\]
By~\eqref{eq:hatrholarge},
\begin{align*}
|\hat\rho_{v,w}(i(b+h))-\hat\rho_{v,w}(ib)|
& \le 
|\hat w|_\infty ||1_{\hY}\hat L^G(i(b+h))1_{\hY}\pi^*v-1_{\hY}\hat L^G(ib))1_{\hY}\pi^*v|_{L^1(\hY)}
\\ & \le Cb^\omega h^{\beta-2\eps}\|v\|_{C^\eta(\tY)}|w|_\infty,
\end{align*}
for all $v\in C^\eta(\tY)$, $w\in L^\infty(\tY)$.
\end{pfof}

\subsection{Estimates for $1_{\hY}A_t$, $1_{\hY}E_t1_{\hY}\pi^*$: proof of Lemma~\ref{lem:ABEt}(a,b)}
\label{sec:CE}

\begin{pfof}{Lemma~\ref{lem:ABEt}(a)}
Let $t\ge0$.  Then
\begin{align*}
  |1_{\hY}A_tv|_1   & \le \int_{Z^\varphi}1_{\hY}1_{Z_t}L^G_t(1_{\tZ}|v|)\,d\mu_Z^\varphi
\le |v|_\infty\int_{\tZ}1_{\hY}\circ G_t\,1_{Z_t}\circ G_t\,d\mu_Z^\varphi.
\end{align*}

Let $(z,u)\in\tZ$.  If $G_t(z,u)\in Z_t$, then $u+t<\varphi(z)$
and $G_t(z,u)=(z,u+t)$.
Hence
\[
\|1_{\hY}A_t\|\le \int_{\tZ}
1_{\hY}(z,u+t)\,1_{\{u+t\in[0,\varphi(z)]\}}\,d\mu_Z^\varphi.
\]

Now,
\begin{align*}
\int_{t_0}^\infty 1_{\hY}(z,u+t)1_{\{u+t\in[0,\varphi(z)]\}}\,dt
\le 1_{\{\varphi(z)>t_0\}}\int_0^{\varphi(z)} 1_{\hY}(z,t)\,dt
 \le 
1_{\{\varphi(z)>t_0\}}\sigma(z),
\end{align*}
by~\eqref{eq:hatY}, and so
 $\int_{t_0}^\infty \|1_{\hY}A_t\|\,dt
\le  \int_Z1_{\{\varphi>t_0\}}\sigma\,d\mu_Z$.

Choose $\eps'>0$ and $p,q>1$ with $p^{-1}+q^{-1}=1$
such that $(\beta-\eps')/p=\beta-\eps$.   Since $\mu$ is $G$-invariant, by H\"older's inequality,
\begin{align*}
 \int_{t_0}^\infty \|1_{\hY}A_t\|\,dt
  & 
 \le \mu(\varphi>t_0)^{1/p}|\sigma|_q\ll t_0^{-(\beta-\eps)},
\end{align*}
completing the proof.
 \end{pfof}

\begin{pfof}{Lemma~\ref{lem:ABEt}(b)}
It follows from the definitions that
\begin{align*}
|1_{\hY}E_t(1_{\hY}\pi^*v)|_1 & \le \int_{Z^\varphi}1_{\hY}L^G_t(|(1-\xi(t))1_{\hY}\hat v|)\,d\mu_Z^\varphi
\le |v|_\infty\int_{Z^\varphi}1_{\hY}\circ G_t\,|1-\xi(t)|1_{\hY}\,d\mu_Z^\varphi.
\end{align*}

It is convenient to break this expression into terms with the same lap number.
Define $\varphi_k=\sum_{j=0}^{k-1}\varphi\circ G^j$.
Given $(z,u)\in Z^\varphi$ and $t>0$, we define the lap number
$N_t=N_t(z,u)$ to be the integer $N_t=k$ such that
\[
u+t\in[\varphi_k(z),\varphi_{k+1}(z)).
\]
Define $E_{t,k}v=E_t(1_{\{N_t=k\}}v)$.
Then $\|1_{\hY}E_t1_{\hY}\pi^*\|=\sum_{k=0}^\infty 
\|1_{\hY}E_{t,k}1_{\hY}\pi^*\|$
where 
\[
\|1_{\hY}E_{t,k}1_{\hY}\pi^*\|\le 
\int_{Z^\varphi}1_{\{N_t=k\}}1_{\hY}\circ G_t(1-\xi(t))1_{\hY}\,d\mu_Z^\varphi.
\]

The desired result is immediate from the following two claims.
\begin{itemize}
\item[(i)]
For each $k\ge0$, there is a constant $C_k>0$ such that
$\int_{t_0}^\infty \|1_{\hY}E_{t,k}1_{\hY}\pi^*\|\,dt\le C_k t_0^{-(\beta-\eps)}$.
\item[(ii)]
$E_{t,k}=0$ for all $k\ge2$.
\end{itemize}

First, we prove claim (i).
When $N_t=k$, we have 
\[
G_t(z,u)=(G^kz,u+t-\varphi_k(z)) \quad\text{and}\quad u+t-\varphi_k(z)\in[0,\varphi(G^kz)].
\]
Hence,
\begin{align*}
 & \int_{t_0}^{\infty}\|1_{\hY}E_{t,k}1_{\hY}\pi^*\|\,dt
  \le
\int_{Z^\varphi}\int_{t_0}^\infty 1_{\{u+t\in[\varphi_k(z),\varphi_{k+1}(z)]\}}1_{\hY}(G^kz,u+t-\varphi_k(z))\,dt\, 1_{\hY}\,d\mu_Z^\varphi
\\ & \qquad \le
\int_{Z^\varphi}1_{\{\varphi_{k+1}>t_0\}}1_{\hY}\int_0^{\varphi(G^kz)} 1_{\hY}(G^kz,t)\,dt\, d\mu_Z^\varphi
=
\int_{Z^\varphi}1_{\{\varphi_{k+1}>t_0\}}1_{\hY}\sigma(G^kz)\,d\mu_Z^\varphi
\\ & 
\qquad =
\int_Z 1_{\{\varphi_{k+1}>t_0\}}\sigma\circ G^k\int_0^{\varphi(z)}1_{\hY}(z,u)\,du\,d\mu_Z
=\int_Z 1_{\{\varphi_{k+1}>t_0\}}\sigma\circ G^k\sigma\,d\mu_Z,
\end{align*}
where we have used~\eqref{eq:hatY} twice.
Using H\"older's inequality as in the proof of Lemma~\ref{lem:ABEt}(a) completes the proof of claim~(i).

To prove claim~(ii), we show that 
$\xi(t)\equiv1$ when $k\ge2$.
The constraints $(z,u)\in\Delta_y$, $G_t(z,u)\in Z_{t-x}$ can be restated as
\[
\varphi(z)-u<y<\varphi(z)-u+1, \qquad \varphi_k(z)-u<x<\varphi_k(z)-u+1<t.
\]
These conditions alone imply that $x,y>0$, $x<t$.  Since $k\ge2$, we have also that
$y<\varphi(z)-u+1\le \varphi_k(z)-u<x$.  
Hence
\[
\xi(t)=\int_0^t\int_0^x 1_{\Delta_y}1_{Z_{t-x}}\circ G_t\,dy\,dx
=\int_{\varphi_k(z)-u}^{\varphi_k(z)-u+1}\int_{\varphi(z)-u}^{\varphi(z)-u+1}1 \,dy\,dx=1,
\]
as required.
\end{pfof}

\subsection{Estimates for $\hat B(s)1_{\hY}\pi^*$: proof of Lemma~\ref{lem:ABEt}(c)}
\label{sec:B}

For $s\in\C$, define the operator
$\hat D(s):L^1(Z^\varphi)\to L^1(\tZ)$, 
\[
\textstyle (\hat D(s)v)(z,u)=e^{-s(\varphi(z)+u)}\int_0^{\varphi(z)} e^{st} v(z,t)\,dt,\quad (z,u)\in\tZ.
\]
Define $\tilde G:\tZ\to\tZ$, 
by $\tilde G(z,u)=(Gz,u)$,
and let $\tilde R:L^1(\tZ)\to L^1(\tZ)$ denote the corresponding transfer operator.

\begin{prop} \label{prop:hatB}
$\hat B(s)= \tilde R\,\hat D(s)$.
\end{prop}

\begin{proof} 
Let $v\in L^1(Z^\varphi)$, $w\in L^\infty(\tZ)$.  
We can regard $w$ as a function on $Z^\varphi$ supported on $\tZ$.
Since $G_t(z,u)=(Gz,u+t-\varphi(z))$ for $(z,u)\in\Delta_t$,
\begin{align*}
 \int_{\tZ} B_tv\, w\,d\mu_Z^\varphi & =
 \int_{Z^\varphi} L^G_t(1_{\Delta_t}v)\, w\,d\mu_Z^\varphi 
= \int_Z\int_0^{\varphi(z)} 1_{\{0\le u-\varphi(z)+t\le 1\}}(v\, w\circ G_t)(z,u)\,du\,d\mu_Z \\ & =
 \int_Z\int_{t-\varphi(z)}^t 1_{\{0\le u\le 1\}}v(z,u+\varphi(z)-t)\, w\circ \tilde G(z,u)\,du\,d\mu_Z \\ & =
 \int_Z\int_0^1 1_{\{t-\varphi(z)\le u\le t\}}v(z,u+\varphi(z)-t)\, w\circ \tilde G(z,u)\,du\,d\mu_Z.
\end{align*}
Hence
\begin{align*}
\int_{\tZ} \hat B(s) v\,w\, d\mu_Z^\varphi & =
\int_{\tZ} \int_u^{\varphi(z)+u} e^{-st} v(z,u+\varphi(z)-t)\,dt\, w\circ \tilde G(z,u)\,d\mu_Z^\varphi \\ & =
\int_{\tZ}e^{-s(\varphi(z)+u)} \int_0^{\varphi(z)} e^{st} v(z,t)\,dt\, w\circ \tilde G(z,u)\,d\mu_Z^\varphi \\ & =
 \int_{\tZ}\hat D(s)v\,w\circ \tilde G\,d\mu_Z^\varphi
= \int_{\tZ}\tilde R\,\hat D(s)v\,w\,d\mu_Z^\varphi,
\end{align*}
as required.
\end{proof}

For $b\in\R$, define $f(b):L^1(\tY)\to L^1(\tZ)$,
\[
f(b)v=\hat D(ib)(1_{\hY}\pi^*v)=\hat D(ib)\hat v, \quad \hat v=v\circ\pi.
\]

\begin{prop} \label{prop:fa}
 For any $\eps\in (0,\beta]$, $\delta>0$, there exists $C>0$ such that
\begin{itemize}
\item[(a)]  $|(f(b)v)(z,u)|\le \sigma(a)|v|_\infty$, 
\item[(b)] 
$|(f(b)v)(z,u)- (f(b)v)(z',u))| \le Cb^\eps\sigma(a)(\infa\varphi^\eps)\|v\|_{C^\eta(\tY)}d_Z(Gz,Gz')^{\eps\eta}$, 
\item[(c)] 
$|\{f(b+h)v-f(b)v\}(z,u)|
\le Ch^{\beta-\eps}\sigma(a)(\infa\varphi^{\beta-\eps})|v|_\infty$,
\item[(d)] 
$|\{f(b+h)v-f(b)v\}(z,u)
-\{f(b+h)v-f(b)v\}(z',u)|
 \vspace*{1ex}
\newline \mbox{} \qquad\qquad\qquad \le Cb^\eps h^{\beta-2\eps}\sigma(a)(\infa\varphi^{\beta-\eps})\|v\|_{C^\eta(\tY)}d_Z(Gz,Gz')^{\eps\eta}$, 
\end{itemize}
for all $a\in\alpha$, $z,z'\in a$, $u\in[0,1]$,
$0<h<\delta<b$, and $v\in C^\eta(\tY)$.
\end{prop}

\begin{proof}
We establish the inequalities (a) and (b) together with
\begin{itemize}
\item[(e)] 
 $|(f'(b)v)(z,u)|\le C\sigma(a)(\infa\varphi)|v|_\infty$.
\item[(f)]
$|(f'(b)v)(z,u)- (f'(b)v)(z',u)|
\le Cb^\eps\sigma(a)(\infa\varphi^{1+\eps})\|v\|_{C^\eta(\tY)}
d_Z(Gz,Gz')^{\eps\eta}$.
\end{itemize}

By the mean value theorem, it follows from~(e) that
$|\{f(b_1)v-f(b_2)v\}(z,u)| \ll \sigma(a)(\infa\varphi)|v|_\infty|b_1-b_2|$.
Combining this with~(a), we obtain
\begin{align*}
|\{f(b_1)v-f(b_2)v\}(z,u)|
& \ll \sigma(a)|v|_\infty\min\{1,(\infa \varphi)|b_1-b_2|\}
 \\ & 
\le \sigma(a)|v|_\infty (\infa \varphi^{\beta-\eps})|b_1-b_2|^{\beta-\eps},
\end{align*}
proving (c).  Similarly, using~(b) and~(f),
\begin{align*}
|\{f(b_1)v & -f(b_2)v\}(z,u)-
\{f(b_1)v-f(b_2)v\}(z',u)|
\\ & \ll (|b_1|^\eps+|b_2|^\eps)\sigma(a)(\infa\varphi^\eps)\|v\|_{C^\eta(\tY)} d_Z(Gz,Gz')^{\eps\eta} \min\{1, (\infa \varphi)|b_1-b_2|\}
\\ & \le (|b_1|^\eps+|b_2|^\eps)\sigma(a)(\infa\varphi^{\beta-\eps})\|v\|_{C^\eta(\tY)} d_Z(Gz,Gz')^{\eps\eta} |b_1-b_2|^{\beta-2\eps},
\end{align*}
proving~(d).

It remains to carry out the estimates~(a),~(b),~(e),~(f).
In the notation of Proposition~\ref{prop:vb},
\[
(f(b)v)(z,u)  
=e^{-ib(\varphi(z)+u)} \int_0^{\varphi(z)}e^{ibt}(1_{\hY}\hat v)(z,t)\,dt
=e^{-ib(\varphi(z)+u)}\hat v(b)(z).
\]

\vspace{1ex}
\noindent(a)
By Proposition~\ref{prop:vb},
$|(f(b)v)(z,u)|= |\hat v(b)(z)|\le  \sigma(a)|v|_\infty$.

\vspace{1ex}
\noindent(e)  Write
\begin{align} \label{eq:f'}
(f'(b)v)(z,u) & =-i(\varphi(z)+u)(f(b)v)(z,u)+e^{-ib(\varphi(z)+u)}\hat v'(b)(z).
\end{align}
By part (a) and Proposition~\ref{prop:vb},
\[
|(f'(b)v)(z,u)|\le 
 (\supa\varphi+1)|(f(b)v)(z,u)|+|\hat v'(b)(z)|
\ll\sigma(a)(\infa\varphi)|v|_\infty.
\]

\vspace{1ex}
\noindent(b) 
Write $(f(b)v)(z,u)-(f(b)v)(z',u)= I_1+I_2$ where 
\begin{align*}
I_1 & = e^{-iu}( e^{-ib\varphi(z)}-e^{-ib\varphi(z')})\hat v(b)(z), \qquad
I_2   =  e^{-iu}e^{-ib\varphi(z')}(\hat v(b)(z)-\hat v(b)(z')).
\end{align*}
By (H3) and Proposition~\ref{prop:vb},
\begin{align*}
|I_1|  & \le 2b^\eps|\varphi(z)-\varphi(z')|^\eps |v|_\infty \sigma(a) 
\ll b^\eps|v|_\infty\sigma(a)(\infa\varphi^\eps)d_Z(Gz,Gz')^{\eps\eta}, \\
|I_2| & = |\hat v(b)(z)-\hat v(b)(z')|\ll b^\eps\sigma(a)(\infa\varphi^\eps)\|v\|_{C^\eta(\tY)} d_Z(Gz,Gz')^{\eps\eta}.
\end{align*}

\vspace{1ex}
\noindent(f) 
Using~\eqref{eq:f'}, estimating 
 $(f'(b)v)(z,u)-  (f'(b)v)(z',u)$ reduces to estimating the four terms
\begin{align*}
& (\varphi(z)-\varphi(z'))(f(b)v)(z,u), \quad
(\varphi(z')+u)((f(b)v)(z,u)-(f(b)v)(z',u)), \quad
\\ & (e^{-ib\varphi(z)}-e^{-ib\varphi(z')})\hat v'(b)(z), \quad
e^{-ib\varphi(z')}(\hat v'(b)(z)-\hat v'(b)(z')).
\end{align*}
By (H3) and Proposition~\ref{prop:vb}, we obtain the same estimates
as in part (b) with an extra factor of $\infa\varphi$.
\end{proof}

\begin{prop}   \label{prop:rhophi}
$\int_Z\sigma\,\varphi^{\beta-\eps}\,d\mu_Z<\infty$ for any $\eps>0$.
\end{prop}

\begin{proof}
Let $\delta>0$, $q\ge1$.  Then
\begin{align*}
\mu_Z(\sigma\,\varphi^{\beta-\eps}>t) & \le 
\mu_Z(\sigma\,\varphi^{\beta-\eps}>t,\,\sigma\le q\log t)+\mu_Z(\sigma>q\log t)=g(t)+O(t^{-cq}),
\end{align*}
where 
\begin{align*}
g(t) & = \mu_Z(\sigma\,\varphi^{\beta-\eps}>t,\,\sigma\le q\log t)
 \le \mu_Z(\varphi>ct^{(1-\delta)/(\beta-\eps)},\,\sigma\le q\log t)
\\ & \le \mu_Z(\tau_{[q\log t]}>ct^{(1-\delta)/(\beta-\eps)})
 \le \sum_{j=0}^{[q\log t]-1}\mu_Z(\tau\circ F^j>c't^{(1-2\delta)/(\beta-\eps)})
\\ & = [q\log t]\,\mu_Z(\tau>c't^{(1-2\delta)/(\beta-\eps)})
 \ll t^{-(\beta-\delta)(1-2\delta)/(\beta-\eps)}.
\end{align*}
(Here, $c$, $c'>0$ are constants.)
Choosing $q$ sufficiently large and $\delta$ sufficiently small, we obtain that
$\mu_Z(\varphi^{\beta-\eps}\sigma>t)=O(t^{-r})$ for some $r>1$ and the result follows.
\end{proof}

\begin{pfof}{Lemma~\ref{lem:ABEt}(c)}
Using the formula for $\tilde R$ in Remark~\ref{rmk:GM}, we obtain
$(\hat B(s)v)(z,u)=\sum_{a\in\alpha}\xi(z_a)(\hat D(s)v)(z_a,u)$.
In particular, for $v\in C^\eta(\tY)$, $\hat v=v\circ \pi$,
\begin{align*} 
(\hat B(ib)1_{\hY}\pi^*v)(z,u)=(\hat B(ib)\hat v)(z,u)=\suma \xi(z_a)(f(b)v)(z_a,u).
\end{align*}
By~\eqref{eq:GM} and Propositions~\ref{prop:fa}(a,b) and~\ref{prop:rhophi}, for all $z,z'\in \tZ$, $u\in[0,1]$,
\[
|(\hat B(ib)\hat v)(z,u)|\le \suma\xi(z_a)|(f(b)v)(z_a,u)| \ll |v|_\infty\suma \mu_Z(a)\sigma(a) \ll |v|_\infty,
\]
and
\begin{align*}
& |(\hat B(ib)\hat v)(z,u)
-(\hat B(ib)\hat v)(z',u)|
\\ & \quad \le \suma |\xi(z_a)-\xi(z_a')||(f(b)v)(z_a,u)| 
+
 \suma  \xi(z_a')|(f(b)v)(z_a,u)-(f(b)v)(z_a',u)| 
 \\ & \quad \ll \suma \mu_Z(a)d_Z(Gz_a,Gz_a')^\eta\sigma(a)|v|_\infty\,
+\suma \mu_Z(a)b^\eps\sigma(a)(\infa\varphi^\eps)\|v\|_{C^\eta}d_Z(Gz_a,Gz_a')^{\eps\eta}
 \\ & \quad \ll b^\eps\|v\|_{C^\eta}d_Z(z,z')^{\eps\eta}.
\end{align*}
Hence
$\|\hat B(ib)\hat v\|_{C^{\eps\eta}(\tZ)}\ll b^\eps\|v\|_{C^\eta(\tY)}$.

Similarly, using~\eqref{eq:GM} and Propositions~\ref{prop:fa}(e,f) and~\ref{prop:rhophi},  
\begin{align*}
|\{\hat B(i(b+h))\hat v-
\hat B(ib)\hat v\}(z,u)| 
& \le \suma \xi(z_a)|\{f(b+h)v-f(b)v\}(z_a,u)|
\\ &
\ll \suma \mu_Z(a)h^{\beta-\eps}\sigma(a)(\infa\varphi^{\beta-\eps})|v|_\infty
\ll h^{\beta-\eps}|v|_\infty,
\end{align*}
and
\begin{align*}
 |\{\hat  B &(i(b+h))\hat v-  \hat B(ib)\hat v\}(z,u)
-\{\hat B(i(b+h))\hat v- \hat B(ib)\hat v\}(z',u)| 
\\ &
\le \suma |\xi(z_a)-\xi(z_a')| |\{f(b+h)v-f(b)v\}(z,u)|
\\ &\qquad  +
\suma \xi(z_a') 
|\{f(b+h)v-f(b)v\}(z,u)-
\{f(b+h)v-f(b)v\}(z',u)|
\\ &
\ll \suma \mu_Z(a)d_Z(Gz_a,Gz_a')^\eta h^{\beta-\eps}\sigma(a)(\infa\varphi^{\beta-\eps})|v|_\infty 
\\ &\qquad  +
\suma \mu_Z(a)
b^\eps h^{\beta-2\eps}\sigma(a)(\infa\varphi^{\beta-\eps})\|v\|_{C^\eta}d_Z(Gz_a,Gz_a')^{\eps\eta}
\\ &
\ll 
b^\eps h^{\beta-2\eps}\|v\|_{C^\eta}d_Z(z,z)^{\eps\eta}.
\end{align*}
Hence
$\|\hat B(i(b+h))\hat v- \hat B(ib)\hat v\|_{C^{\eps\eta}(\tZ)}\ll 
b^\eps h^{\beta-2\eps}\|v\|_{C^\eta(\tY)}$.
\end{pfof}

\section{Proof of Theorem~\ref{thm:rateA}}
\label{sec:rateA}

In this section, we consider the proof of Theorem~\ref{thm:rateA}.
The arguments are similar to those in Section~\ref{sec:small}, with hypothesis (H1) replaced by hypothesis (A1).
Our conventions regarding $\eps$, $\delta$, $p$, $C$ are the same as in Section~\ref{sec:T}.
Let  $c_0=c^{-1}c_\beta^{-1}$
where $c_\beta=i\int_0^\infty e^{-i\sigma}\sigma^{-\beta}\,d\sigma$.

Lemma~\ref{lem:continf} is replaced by:

\begin{lemma}\label{lem:continfA}
$\|\hat R(s_1)-\hat R(s_2)\|_{\cB(Y)\to L^1(Y)}\le C\, |s_1-s_2|^{\beta-\eps}$
for all $s_1,s_2\in\barH$,
\end{lemma}

\begin{proof}
Since $R:L^1(Y)\to L^1(Y)$ is a contraction,
\begin{align*}
|(\hat R(s_1) -\hat R(s_2))v|_1 
& \le  |(e^{-s_1\tau}-e^{-s_2\tau})v|_1
  \le  2|s_1-s_2|^{\beta-\eps}|\tau^{\beta-\eps} v|_1.
\end{align*}
Choose $r>1$ such that $(\beta-\eps)r<\beta$ with conjugate exponent $r'$.
By H\"older's inequality and (A1)(i), 
$|\tau^{\beta-\eps} v|_1\le |\tau^{\beta-\eps}|_r|v|_{r'}\ll \|v\|_{\cB(Y)}$
completing the proof.
\end{proof}

Next, we proceed as in Lemma~\ref{lem:estKL},
using Lemma~\ref{lem:continfA} and hypotheses~(A1)(i),(ii).
This leads to continuous families of eigenvalues $\lambda(s)$, projections $P(s)$ and eigenfunctions $\zeta(s)$ with the same properties as before except that $L^p(Y)$ is replaced by $L^1(Y)$.
The main difference is Proposition~\ref{prop:chi} which is replaced by the following.

\begin{prop} \label{prop:chiA}
\begin{itemize}
\item[(a)] $|\chi(s)|\le C|s|^{\beta_+}$ for all $s\in\barH\cap B_\delta(0)$,
\item[(b)] $|\chi(i(b+h))-\chi(ib)|\le C h^{\beta-\eps}$ for all 
$0<h<b<\delta$. 
\end{itemize}
\end{prop}

\begin{proof}
Part~(a) is immediate from (A1)(iii).  
Since $\|\zeta(s)\|_{\cB(Y)}$ is bounded, it follows from H\"older's inequality as in the proof of Lemma~\ref{lem:continfA} that
\begin{align*}
|\chi(i(b+h))-\chi(ib)|
& \le |(e^{i(b+h)\tau}-1)(\zeta(i(b+h))-\zeta(ib))|_1
+|(e^{ih\tau}-1)(\zeta(ib)-1)|_1  \\
& \le 2|\zeta(i(b+h))-\zeta(ib)|_1+
2h^{\beta-\eps} |\tau^{\beta-\eps}|_r|(\zeta(ib)-1)|_{r'}
 \ll h^{\beta-\eps}.
\end{align*}
Hence we obtain part~(b).
\end{proof}

The statement and proof of Lemma~\ref{lem:T} goes through unchanged since $\beta_+>\beta$.
Lemma~\ref{lem:Tcont} becomes:

\begin{lemma} \label{lem:TcontA}
 $\|\hat T(i(b+h)) -\hat T(ib)\|_{\cB(Y)\to L^1(Y)}
\le Cb^{-2\beta} h^{\beta-\eps}$ for all  $0<h< b<\delta$.
\end{lemma}

\begin{proof}
By Proposition~\ref{prop:chiA}(b), in place of~\eqref{eq:Dlambda} we have
$|\lambda(i(b+h))-\lambda(ib)| \ll h^{\beta-\eps} $.
Now proceed as before.
\end{proof}

Lemma~\ref{lem:T-higherorder} becomes:
\begin{lemma}
\label{lem:T-higherorderA}
  $\hat T(ib)  =c_0 b^{-\beta}P(0) +E(ib)$
for all $b\in [0,\delta)$, 
where  
$\|E(ib)\|_{\cB(Y)\to L^1(Y)}\le Cb^{-(2\beta-\beta_+)}$.
\end{lemma}

\begin{proof} 
By~\eqref{eq:ev-Gou},~\eqref{eq:expand} and Proposition~\ref{prop:chiA}(a),
\[
1-\lambda(ib) =cc_\beta b^\beta(1+O(b^{\beta_+-\beta})), \qquad
 (1-\lambda(ib))^{-1}=c_0b^{-\beta}+O(b^{-(2\beta-\beta_+)}).
\]
The rest of the proof is unchanged.
\end{proof}

Lemma~\ref{lem:small} is replaced by the following result.

\begin{lemma} \label{lem:smallA}
Let $v\in \cB(\tY)$, $w\in L^\infty(\tY)$.
\begin{itemize}
\item[(a)]
$|\hat\rho(s)|  \le C |s|^{-\beta}\|v\|_{\cB(\tY)}\,|w|_{L^\infty(\tY)}$
for all $s\in\barH\cap B_\delta(0)$.
\item[(b)]
$|\hat \rho(i(b+h))-\hat \rho(ib)|\le C b^{-2\beta} h^{\beta-\eps} 
\|v\|_{\cB(\tY)}|w|_{L^\infty(\tY)}$
for all $0<h<b<\delta$.
\item[(c)] 
For all $a\in(0,\delta t)$,
$\int_0^{a/t}  e^{ibt}\hat\rho(ib)\,db
 =c_0 \int_0^{a/t}  b^{-\beta} e^{ibt}\,db\int_{\tY}v\,d\mu^\tau \int_{\tY}w\,d\mu^\tau 
+ O\big((a/t)^{1-2\beta+\beta_+}\|v\|_{\cB(\tY)}\,|w|_{L^\infty(\tY)}\big)$.
\end{itemize}
\end{lemma}

\begin{proof}
The proofs proceed exactly as before.
\end{proof}

\begin{lemma} \label{lem:DTA}
For all $a\ge1$, $t>(a+\pi)/\delta$, $v\in \cB(\tY)$, $w\in L^\infty(\tY)$,
\[
\Big|\int_{a/t}^{\delta}  e^{ibt}\hat\rho(ib)\,db\Big|
\le C t^{-(1-\beta-\eps)}a^{-(2\beta-1)} \|v\|_{\cB(\tY)}|w|_{L^\infty(\tY)}.
\]
\end{lemma}

\begin{proof} 
We use the same decomposition 
$2I=I_1 + I_2 +  I_3$ as in the proof of Lemma~\ref{lem:DT}.
The estimates for $I_1$ and $I_2$ are unchanged.
By Lemma~\ref{lem:smallA}(b) with $h=\pi/t$,
\begin{align*}
|I_3| & \ll 
t^{-(\beta-\eps)}\int_{a/t}^\infty b^{-2\beta}\,db
\ll
t^{-(1-\beta-\eps)}a^{-(2\beta-1)},
\end{align*}
as required.
\end{proof}

\begin{pfof}{Theorem~\ref{thm:rateA}}
We argue as in the proof of Theorem~\ref{thm:rate}.
By Lemmas~\ref{lem:smallA}(c),~\ref{lem:DTA} and~\ref{lem:infty} (which is unchanged), 
\begin{align*}
\int_0^\infty e^{ibt}\hat \rho(ib)\,db-c_0'  t^{-(1-\beta)} 
& \ll (a/t)^{1-2\beta+\beta_+} +t^{-(1-\beta-\eps)}a^{-(2\beta-1)}.
\end{align*}
The result follows with $a=t^{(\beta_+-\beta)/\beta_+}$.
\end{pfof}

\section{Piecewise $C^{1+\eta}$ semiflows satisfying UNI}
\label{sec:UNI}

As mentioned in the introduction, exponential decay of correlations has been verified for a very restricted class of continuous time dynamical systems, whereas superpolynomial decay of correlations is better understood.   Hence, we have chosen in this paper to focus on the dynamical systems that are amenable to superpolynomial-type techniques (modulo the nonintegrability of the roof functions $\tau$ and $\varphi$).  The price we pay for this extra generality is the smoothness requirements on the observable $w$.

In this section, we consider the more restrictive class of dynamical systems for which exponential-type techniques are available.  Consequently, 
the smoothness assumptions on $w$ are relaxed.

We continue to assume the tail conditions on $\tau:Y\to\R^+$, $\varphi:Z\to\R^+$ and $\sigma:Z\to\Z^+$ as well as assumption (H1).
 As usual we require that $\inf\varphi\ge2$.

Assumption (H2) is replaced by the following:
Fix $\eta\in(0,1]$.    Let $\{(c_m,d_m)\}$ be a countable partition mod $0$
of $Z=[0,1]$ and 
suppose that $G:Z\to Z$ is $C^{1+\eta}$ on each subinterval $(c_m,d_m)$ and extends to a homeomorphism from $[c_m,d_m]$ onto $Z$.
(In~\cite{AraujoMapp}, the map is denoted by $F$ and is $C^{1+\alpha}$ but it is convenient here to use $\eta$ instead of $\alpha$.)
Suppose that 
$\sigma$ is constant on partition elements of $Z$ and that
the roof function $\varphi:Z\to\R^+$ is $C^1$ on partition elements.

Let $\cG_n$ denote the set of inverse branches for $G^n$ and write $\cG=\cG_1$.
We assume that there are constants $C>0$, $\rho_0\in(0,1)$ such that
\begin{itemize}
\item $|g'|_\infty\le C\rho_0^n$ for all $g\in\cG_n$,
\item $|\log|g'||_\eta\le C$ for all $g\in\cG$,
\item $|(\varphi\circ g)'|_\infty \le C$ for all $g\in\cG$.
\end{itemize}
These correspond to conditions (i)--(iii) from~\cite{AraujoMapp}.
In place of (H3), we require that 
\begin{itemize} 
\item $|\tau_\ell\circ g|_\eta \le C|\varphi\circ g|_\infty$
for all $\ell\le\sigma\circ g$ and all $g\in\cG$.
\end{itemize}
Finally, (H4) is replaced by the uniform nonintegrability condition
\begin{itemize}
\item[(UNI)]  For all $n_0\ge1$, there exists $n\ge n_0$ and $g_1,g_2\in\cG_n$ such that $\psi=
\varphi_n\circ g_1- \varphi_n\circ g_2$ satisfies
$\inf |\psi'|>0$.
(Here, $\varphi_n=\sum_{j=0}^{n-1}\varphi\circ G^j$.)
\end{itemize}

The main result in this section is:

\begin{theorem} \label{thm:UNI}
Under the above condition, Theorems~\ref{thm:rate} holds with $m=2$.
\end{theorem}

\begin{rmk}  
The assumption that $G$ is a one-dimensional map can be relaxed somewhat.  We have used the 
set-up in~\cite{AraujoMapp} which extends results of~\cite{BalVal} from $C^2$ maps to $C^{1+\alpha}$ maps.   In both references, the map $G$ is one-dimensional.
Higher-dimensional $C^2$ maps are treated in~\cite{AGY} but under very restrictive smoothness assumptions on the boundaries of the partition elements, and this also extends to $C^{1+\alpha}$ maps~\cite{BW}.  Hence Theorem~\ref{thm:UNI} can be shown to hold for semiflows over $C^{1+\alpha}$ piecewise expanding maps in arbitrary dimensions, subject to this restriction on the partition elements.
\end{rmk}

The proof of Theorem~\ref{thm:UNI} occupies the remainder of the section.
First, we define the family of equivalent norms on $C^\eta(Z)$:
\[
\|v\|_b=\max\{|v|_\infty,|v|_\eta/(1+|b|^\eta)\}, \quad b\in\R.
\]
If $L$ is an operator on $C^\eta(Z)$, we write
$\|L\|_b=\sup_{v\in C^\eta,\,\|v\|_b=1}\|Lv\|_b$.

Let $R^G$ denote the transfer operator corresponding to $G:Z\to Z$,
and define $\hat R^G(s)v=R^G(e^{-s\varphi}v)$.

The are various families of transfer operators 
$P_s$, $L_s$ and $Q_s$ in~\cite{AraujoMapp}.
We are only interested in the range $s\in\barH$.
Moreover, we are particularly interested in $s=ib$ imaginary.
It is easily checked that $L_{ib}=Q_{ib}=\hat R^G(ib)$.
In what follows, we use the notation $\hat R^G(ib)$.
Also, the operator $P_s$ in~\cite{AraujoMapp} satisfies 
$P_s=f_0\hat R^G(ib)f_0^{-1}$ where $f_0$, $f_0^{-1}\in C^\eta$.

The estimates in~\cite{AraujoMapp} are carried out for $\Re s\ge-\eps$; 
it is easy to check that the estimates for $s\in\barH$ hold even though
condition~(iv) from~\cite{AraujoMapp} is not assumed.  

\begin{prop} \label{prop:bound}
There exists $C>0$ such that
$\|\hat R^G(ib)^n\|_b\le C$ for all $b\in\R$, $n\ge1$.
\end{prop}

\begin{proof}
See~\cite[Corollary~2.8]{AraujoMapp}.
\end{proof}

\begin{lemma} \label{lem:AM}
There exist constants $A,D>0$ and $\gamma\in(0,1)$ such that
\[
\|\hat R^G(ib)^n\|_b\le \gamma^n\;
\text{for all $n\ge A\ln b$, $b\ge D$}.
\]
\end{lemma}

\begin{proof}  This is contained in the second paragraph of the proof
of~\cite[Theorem~2.16]{AraujoMapp}.~
\end{proof}

\begin{cor} \label{cor:AM}
For any $\delta>0$,
there exists $C>0$ such that
\[
\|(I-\hat R^G(ib))^{-1}\|_b\le C\ln b\;
\text{for all $b>\delta$}.
\]
\end{cor}

\begin{proof}  
Let $A,D$ be as in Lemma~\ref{lem:AM}, increased if necessary so that 
$A\ln D\ge1$.

First suppose that $b\ge D$.
Let $m=m(b)=\lceil A\ln b\rceil\ge1$.
By Lemma~\ref{lem:AM}, $\|\hat R^G(ib)^m\|_b\le\gamma^m\le\gamma$ and so
$\|(I-\hat R^G(ib)^m)^{-1}\|_b\le (1-\gamma)^{-1}$.

Next we use the identity
$(I-\hat R^G)^{-1}=(I+\hat R^G+\dots+(\hat R^G)^{m-1})(I-(\hat R^G)^m)^{-1}$
and Proposition~\ref{prop:bound} to conclude that
$\|(I-\hat R^G(ib))^{-1}\|_b\le mC(1-\gamma)^{-1}\ll \ln b$.

To deal with the range $\delta<b<D$, we note by the proof of~\cite[Lemma~2.22]{AraujoMapp} that $1\not\in\spec \hat R^G(ib)$ for any $b>0$.
Hence by continuity of the family $b\mapsto\hat R^G(ib)$ and compactness of
$[\delta,D]$ it follows that there is a constant $C>0$ such that
$\|(I-\hat R^G(ib))^{-1}\|_b\le C$ for all $b\in[\delta,D]$.
\end{proof}

\begin{prop}   \label{prop:DR}
Let $\eps\in(0,\beta)$.  There exists $C>0$ such that
$\|\hat R^G(i(b+h))-\hat R^G(ib)\|_b\le Ch^{\beta-\eps}$ for all
$b,h>0$.
\end{prop}

\begin{proof}  
The unnormalized transfer operator has the form
\[
\textstyle P(ib)=\sum_{g\in\cG} A_g(b),\quad  A_g(b)v=e^{-ib\varphi\circ g}|g'|v\circ g.
\]
Now,
\[
|A_g(b+h)v- A_g(b)v|_\infty\le h^{\beta-\eps}|\varphi^{\beta-\eps}\circ g|_\infty|g'|_\infty|v|_\infty.
\]
It follows from the assumptions on $G$ and $\varphi$, together with Proposition~\ref{prop:rhophi}, that $\sum_{g\in\cG}|\varphi^{\beta-\eps}\circ g|_\infty |g'|_\infty \ll \int_Z\varphi^{\beta-\eps}\,d\mu_Z<\infty$ and
so $|P(i(b+h))-P(ib)|_\infty\ll h^{\beta-\eps}$.

Next, the proof of~\cite[Proposition~2.5]{AraujoMapp} shows that
$|A_g(b)v|_\eta \ll |g'|_\infty\{(1+|b|^\eta)|v|_\infty +|v|_\eta\}$.
Also, $A_g'(b)=-i(\varphi\circ g) A_g(b)$, so
\begin{align*}
|A_g'(b)v|_\eta & \le |\varphi\circ g|_\infty
|A_g(b)v|_\eta
+|\varphi\circ g|_\eta|A_g(b)v|_\infty
\\ &
\ll |\varphi\circ g|_\infty|g'|_\infty\{(1+|b|^\eta)|v|_\infty +|v|_\eta\}.
\end{align*}
Therefore,
\begin{align*}
|(A_g(b+h)v)-(A_g(b)v)|_\eta
& \ll |g'|_\infty\{(1+|b|^\eta)|v|_\infty +|v|_\eta\}\min\{1,|\varphi\circ g|_\infty h\}
\\
& \le |g'|_\infty\{(1+|b|^\eta)|v|_\infty +|v|_\eta\}|\varphi\circ g|_\infty^{\beta-\eps} h^{\beta-\eps},
\end{align*}
and so
\[
|P(i(b+h))v-P(ib)v|_\eta \ll
\{(1+|b|^\eta)|v|_\infty +|v|_\eta\}h^{\beta-\eps}.
\]

It follows from these estimates that
$\|P(i(b+h))-P(ib)\|_b\le Ch^{\beta-\eps}$.
Finally, $\hat R^G(ib)=f_0^{-1} P(ib)f_0$
where $f_0,\,f_0^{-1}\in C^\eta$ and the result follows.
\end{proof}

\begin{prop} \label{prop:B}
$\| \hat B(ib)1_{\hY}\pi^*\|_{C^\eta(\tY)\to C^{\eps\eta}(\tZ)}
\le Cb^\eps$ and
$$
\|\hat B(i(b+h))1_{\hY}\pi^*-\hat B(ib)1_{\hY}\pi^*\|_{C^\eta(\tY)\to C^{\eps\eta}(\tZ)}\le C b^\eps h^{\beta-2\eps},
$$
for all $t_0>0$, $0<h<\delta<b$.
\end{prop}

\begin{proof}
This is almost identical to the argument in Section~\ref{sec:large}.
The only difference is that in the final arguments in the proof of
Lemma~\ref{lem:ABEt}(c), we have $g'=p\circ g$ and the estimate
$|\xi(z)-\xi(z')|\ll \mu(a)d_Z(Gz,Gz')$ for $a\in\alpha$, $z,z'\in\alpha$
is replaced by 
$|g'z-g'z'|\ll |g'z||z-z'|^\eta$ which holds for all $z,z'\in Z$.
\end{proof}

\begin{pfof}{Theorem~\ref{thm:UNI}}
It follows from Corollary~\ref{cor:AM} and Proposition~\ref{prop:DR}
(with $\eta$ replaced by $\eps\eta$) that
\[
\|(I-\hat R^G(i(b+h)))^{-1}-
(I-\hat R^G(ib))^{-1}\|_b\ll (\ln b)^2h^{\beta-\eps}.
\]
Since $|v|_\infty\le \|v\|_b\le \|v\|_\eta$,
\[
\|(I-\hat R^G(i(b+h)))^{-1}-
(I-\hat R^G(ib))^{-1}\|_{C^{\eps\eta}(Z)\to L^\infty(Z)}\ll (\ln b)^2h^{\beta-\eps}.
\]
Similarly,
$\|(I-\hat R^G(i(b+h)))^{-1}\|_{C^{\eps\eta}(Z)\to L^\infty(Z)}\ll \ln b$.
As in~\cite{MT17}, we obtain that
\[
\|\hat T^G(ib)\|_{C^{\eps\eta}(\tilde Z)\to L^\infty(\tilde Z)}\ll \ln b, \quad
\|\hat T^G(i(b+h))-
\hat T^G(ib)\|_{C^{\eps\eta}(\tilde Z)\to L^\infty(\tilde Z)}\ll (\ln b)^2h^{\beta-\eps}.
\]

We now proceed exactly as in Section~\ref{sec:large}, with
Lemma~\ref{lem:ABEt}(c) replaced by Proposition~\ref{prop:B},
to conclude that Lemma~\ref{lem:large} holds with $\omega\in(0,1)$.
(In fact, $\omega$ can be taken arbitrarily small.)
Hence in the proof of Lemma~\ref{lem:infty}, we can take $m=2$.
\end{pfof}

\section{Semiflows modelled by Young towers}
\label{sec:Young}

Let $f_t:M\to M$ be a semiflow on 
a bounded finite-dimensional Riemannian manifold $(M,d_M)$, with codimension one cross-section~$X$, first hit time $\tau_0:X\to \R^+$, and Poincar\'e map $f:X\to X$.
Here $\tau_0(x)>0$ is least such that $f_{\tau_0(x)}(x)\in X$ and
$f(x)=f_{\tau_0(x)}(x)$.  
In contrast to Section~\ref{sec:ambient}, we now suppose that 
$f:X\to X$ is a finite measure system with $\tau_0$ nonintegrable.
Since there is a natural semiconjugacy from $X^{\tau_0}$ to $M$, we focus attention on the semiflow $f_t:X^{\tau_0}\to X^{\tau_0}$.

Now suppose that $f:X\to X$ is modelled by a (one-sided) Young tower with exponential tails~\cite{Young98,Young99}.
That is, there is a subset $Z\subset X$ and a return time $\sigma:Z\to\Z^+$
(not necessarily the first return time)
such that
$G=f^\sigma:Z\to Z$ is a full branch Gibbs-Markov map.
Denote the corresponding partition and
ergodic $G$-invariant probability measure by $\alpha$ and $\mu_Z$ respectively.
We require that $\sigma:Z\to\Z^+$ is constant on partition elements and as usual that $\sigma$ has exponential tails: $\mu_Z(\sigma>n)=O(e^{-dn})$ for some $d>0$.

For $z,z'\in Z$, we
let $s(z,z')$ denote the least $n\ge0$ such that
$G^nz$ and $G^nz'$ lie in distinct elements of $\alpha$.
For $\theta\in(0,1)$ fixed, we
define the symbolic metric $d_\theta(z,z')=\theta ^{s(z,z')}$ on $Z$.
Our final assumption on the map $f$ is that there exists
$\theta_0\in(0,1)$ and $C>0$ such that
$d_M(f^\ell z,f^\ell z')\le Cd_{\theta_0}(z,z')$
for all $z,z'\in Z$, $0\le\ell<\sigma(z)-1$.

Define the Young tower map $F:Y\to Y$,
\[
Y=\{(z,\ell)\in Z\times\Z:0\le \ell\le \sigma(z)-1\},
\quad
F(z,\ell)=\begin{cases} (z,\ell+1) & \ell\le \sigma(z)-2 \\
(Gz,0) & \ell=\sigma(z)-1 \end{cases},
\]
with ergodic $F$-invariant probability measure
$\mu=(\mu_Z\times{\rm counting})/\bar\sigma$ where $\bar\sigma=\int_Z\sigma\,d\mu_Z$.
Let $\pi_Y:Y\to X$ be the semiconjugacy $\pi_Y(z,\ell)=f^\ell z$.
Then $\mu_X=(\pi_Y)_*\mu$ is an ergodic $f$-invariant probability measure on $X$.

Let $\tau=\tau_0\circ\pi_Y:Y\to\R^+$ be the lifted roof function.
Note that
$\mu(\tau>t)=\mu_X(\tau_0>t)$ and that
$\tau_\ell=\sum_{j=0}^{\ell-1}\tau_0\circ f^j = 
\sum_{j=0}^{\ell-1}\tau\circ F^j$ on $Z$.
Hence various assumptions are expressed equally well in terms of $\tau_0$ and $\tau$.
Define $\varphi=\tau_\sigma:Z\to\R^+$.
As usual, we assume $\essinf\tau_0>1$ and the tail condition
\begin{itemize}
\item[]
 $\mu_X(\tau_0>t)=ct^{-\beta}+O(t^q)$ where $c>0$, $\beta\in(\frac12,1)$
and $q\in(1,2\beta]$.
\end{itemize}
Also, we assume
(H3) for some $\theta_0\in(0,1)$,
namely,
\begin{itemize}
\item[]
There exists $C>0$ such that
$|\tau_\ell(z)-\tau_\ell(z')|\le C(\infa\varphi)d_{\theta_0}(z,z')$
for all $a\in\alpha$, $z,z'\in a$, $\ell\le \sigma(a)$.
\end{itemize}
Hypothesis (H2) is automatically satisfied in the above set-up, and
as usual hypothesis (H4) is satisfied in typical examples.

It remains to address hypothesis (H1) or alternatively (A1).
In fact, hypothesis (H1)(iii) is easily seen to fail for Young towers (for reasonable examples of roof functions $\tau_0$) and also Lemma~\ref{lem:continf} fails.
This is due to the lack of smoothing properties of the transfer operator $R$ during excursions up the tower.  Hence in the remainder of this section we discuss  hypothesis (A1).

\subsection{Hypotheses (A1)(i),(ii) for Young towers}

Young~\cite{Young98} introduced  a function space for Young towers with exponential tails and proved that this function space satisfies hypothesis (A1)(i) and (ii) with $s=0$.   B\'alint \& Gou\"ezel~\cite{BalintGouezel06} enlarged this function space and proved (A1)(i) and (ii) for $s\in\barH$.  
However, they require a strengthened form of (H3) which is not satisfied in Example~\ref{ex:Collet}.   In general it is easy to see that $\hat R(s)$ is an unbounded operator for the function spaces in~\cite{BalintGouezel06,Young98} for all $s\neq0$.
Hence, we define a new function space as follows.

Let $Y_\ell=\{(z,j)\in Y:j=\ell\}$ denote the $\ell$'th level of the tower.
Fix $\eps>0$ (to be specified later), $\beta'\in(0,\beta)$ and $\theta\in[\theta_0^{\beta'},1)$.
For $v:Y\to\R$, define the weighted $L^\infty$ norm
$\|v\|_{w,\infty}=\sup_{\ell\ge0} e^{-\ell\eps}|1_{Y_\ell}v|_\infty$
as in~\cite{Young98}.  Also, set
\[
\|v\|_\theta=\sup_{\ell\ge0}e^{-\eps\ell} \sup_{\stackrel{a\in \alpha}{\sigma(a)>\ell}} 
\sup_{\stackrel{z,z'\in a}{z\neq z'}} \varphi(z)^{-\beta'}|v(z,\ell)-v(z',\ell)|/d_\theta(z,z').
\]
Let $\cB(Y)$ denote the Banach space of functions $v:Y\to\R$ with
$\|v\|_{\cB(Y)}=\max\{\|v\|_{w,\infty},\|v\|_\theta\}<\infty$.

\begin{rmk}
Following~\cite{BalintGouezel06}, we have used the separation time in terms of iterates of $G$ as in~\cite{Young99} rather than iterates of $F$ as in~\cite{Young98}.   This increases the class of observables and more importantly the class of allowable roof functions $\tau_0$.
If $\beta'=0$, then $\cB(Y)$ is precisely the function space
defined in~\cite{BalintGouezel06}.

Since $\cB(Y)$ is larger than the spaces in~\cite{BalintGouezel06,Young98,Young99},
it is standard that H\"older observables $v:X\to\R$ lift to observables $v\circ\pi_Y\in\cB(Y)$.  Hence from now on we can work with $Y$ and $\tau$ instead of $X$ and $\tau_0$.
\end{rmk}

\begin{lemma}  \label{lem:Young}
Let $\beta'\in(0,\beta)$, $\theta\in[\theta_0^{\beta'},1)$.
For $\eps>0$ sufficiently small, conditions (A1)(i),(ii) hold for $\cB(Y)$.
\end{lemma}

\begin{proof}
If $\tau$ satisfies the strengthened (H3) condition
$|\tau_\ell(z)-\tau_\ell(z')|\le Cd_{\theta_0}(z,z')$ for $a\in\alpha$, $z,z'\in a$, $\ell\le \sigma(a)$, 
then this is proved in~\cite[Propositions~3.5 and~3.7]{BalintGouezel06}.
The general case is proved in Appendix~\ref{app:Young}.
\end{proof}

\subsection{A reformulation of (A1)(iii) for Young towers}

In this subsection, we obtain a useful reformulation of (A1)(iii) in the special case of Young towers, using ideas in~\cite[Section~3]{BalintGouezel06}.
The results only make use of the norm $\|\;\|_{w,\infty}$ so the extra factor of $\varphi^{\beta'}$ in the definition of $\|\;\|_{\cB(Y)}$ does not affect these arguments.
Define
\[
\textstyle \psi:Y\to\R^+, \qquad \psi(z,\ell)=\tau_\ell(z)=\sum_{j=0}^{\ell-1}\tau(F^jz).
\]

\begin{lemma}  \label{lem:zeta}
Let $\beta_-\in(0,\beta)$, 
$p\in[1,\infty)$. 
There exists $\delta>0$ such that $|\zeta(s)-e^{-s\psi}|_p=O(|s|^{\beta_-})$
for all $s\in\barH\cap B_\delta(0)$.
\end{lemma}

\begin{proof}  
We sketch the main steps following~\cite{BalintGouezel06}.
For notational convenience, write $\lambda_s=\lambda(s)$ and $\zeta_s=\zeta(s)$.
First,~\cite[Lemmas~3.12, 3.13 and~3.14]{BalintGouezel06}
go through unchanged except that $t$, $w_t$ and $\Delta_0$ are now $s$, $\zeta_s$ and $Y_0$, and error rates $|t|$ are replaced by $|s|^{\beta_-}$.  As in~\cite[Lemma~3.15]{BalintGouezel06}, we obtain
that $|1_{Y_0}(\zeta_s-c_s)|_\infty=O(|s|^{\beta_-})$ where
$c_s\in\R$ and $c_s\to1$.  Since $|\zeta_s-1|_1=O(|s|^{\beta_-})$ (see the discussion after Lemma~\ref{lem:continfA}) it follows easily from the proof of~\cite[Lemma~3.15]{BalintGouezel06} that
$c_s=1+O(|s|^{\beta_-})$ and hence that
\begin{align} \label{eq:base}
|1_{Y_0}(\zeta_s-1)|_\infty=O(|s|^{\beta_-}).
\end{align}

Now $\hat R(s)\zeta_s=\lambda_s\zeta_s$, so for $(z,\ell)\in Y\setminus Y_0$,
\[
\zeta_s(z,\ell)=\lambda_s^{-1}\big(R(e^{-s\tau}\zeta_s)\big)(z,\ell)
=\lambda_s^{-1}e^{-s\tau(z,\ell-1)}\zeta_s(z,\ell-1).
\]
Inductively,
$\zeta_s(z,\ell) =\lambda_s^{-\ell}e^{-s\psi(z,\ell)}\zeta_s(z,0)$
for all $(z,\ell)\in Y$.   
Hence
\begin{align*}
|\zeta_s(z,\ell)-e^{-s\psi(z,\ell)}|
& \le |\lambda_s^{-\ell}\zeta_s(z,0)-1|
\le |\lambda_s^{-\ell}-1||\zeta_s(z,0)|+|\zeta_s(z,0)-1|.
\end{align*}
Now $|\zeta_s(z,0)|\le \|\zeta_s\|_{w,\infty}$ which is bounded
and $|\zeta_s(z,0)-1|=O(|s|^{\beta_-})$ by~\eqref{eq:base}.
Hence
\begin{align*}
|\zeta_s(z,\ell)-e^{-s\psi(z,\ell)}|
& \ll |\lambda_s^{-\ell}-1|+|s|^{\beta_-}
 \ll |\lambda_s^{\ell}-1|+|s|^{\beta_-}
\\ & \ll \big((1+O(|s|^{\beta_-}))^\ell-1\big)+|s|^{\beta_-}
\ll |s|^{\beta_-}\big\{\ell(1+O(|s|^{\beta_-}))^\ell+1\big\}.
\end{align*}
For $|s|<\delta$ sufficiently small,
it now follows from exponential tails that
\[
|\zeta_s-e^{-s\psi}|_p^p
\le |s|^{p{\beta_-}}\textstyle \sum_{\ell\ge 0} \mu(Y_\ell)\big\{\ell^p(1+O(|s|^{\beta_-}))^{p\ell}+1\big\}\ll |s|^{p{\beta_-}},
\]
as required.
\end{proof}

\begin{cor} \label{cor:zeta}
Let $\beta_+\in(\beta,2\beta)$.  Then (A1)(iii) is equivalent to the
condition
\begin{align} \label{eq:zeta}
\textstyle \big|\int_Y(e^{-s\tau}-1)(e^{-s\psi}-1)\,d\mu\big| =O(|s|^{\beta_+})
\quad\text{for $s\in\barH\cap B_\delta(0)$.}
\end{align}
\end{cor}

\begin{proof}
Let $\beta_-\in(\frac12\beta_+,\beta)$.
By H\"older's inequality and Lemma~\ref{lem:zeta},
for all $r,r'>1$ with $r^{-1}+r'^{-1}=1$,
\[
|(e^{-s\tau}-1)(\zeta(s)-e^{-s\psi})|_1
\le 
2|s|^{\beta_-}|\tau^{\beta_-}|_r |\zeta(s)-e^{-s\psi}|_{r'}
\ll |s|^{2\beta_-}|\tau^{\beta_-}|_r
\le |s|^{{\beta_+}}|\tau^{\beta_-}|_r.
\]
Taking $r$ sufficiently close to $1$, we can ensure that
$|\tau^{\beta_-}|_r<\infty$.
\end{proof}

\subsection{Sufficient conditions for (A1)(iii)}

In this subsection, we give sufficient conditions for~\eqref{eq:zeta} and hence (A1)(iii).

Define $I(s)=(e^{-s\tau}-1)(e^{-s\psi}-1)$.
By Corollary~\ref{cor:zeta}, it suffices to show that $|I(s)|_1=O(|s|^{\beta_+})$ for some $\beta_+>\beta$.
To achieve this, we consider decompositions of the form
\[
Y=\{\tau<\psi^{1-q}\} \cup \{\sigma\ge d_0\log \psi\} \cup
\{\sigma< d_0\log\psi\;\text{and}\;\tau\ge\psi^{1-q}\},
\]
for appropriate choices of $q$ and $d_0$.

Recall that $\psi\in L^p$ for all $p<\beta$.  Hence, for any $q\in(0,1)$,
\begin{align} \label{eq:Yn1}
|1_{\{\tau<\psi^{1-q}\}}I(s)|_1
\le 4|s|^{\beta_+}|1_{\{\tau<\psi^{1-q}\}}\tau^{\beta_+}|_1
\le 4|s|^{\beta_+}|\psi^{\beta_+(1-q)}|_1
\ll |s|^{\beta_+},
\end{align}
for all $\beta_+\in(\beta,1]$ with $\beta_+<\beta/(1-q)$.
Also, since $\mu_Z(\sigma>n)=O(e^{-dn})$,
\begin{align} \label{eq:Yn2}
|1_{\{\sigma\ge d_0\log \psi\}}I(s)|_1
\le 4|s||1_{\{\sigma\ge d_0\log\psi\}}\psi|_1
\le 4|s||e^{d_0^{-1}\sigma}|_1 \ll |s|,
\end{align}
for all $d_0>d^{-1}$.

It remains to estimate
$|1_{\{\sigma< d_0\log\psi\}}1_{\{\tau\ge\psi^{1-q}\}}I(s)|_1$.
Define $\psi_*(z,\ell)=\max_{0\le j\le\ell-1}\tau(F^jz)$.  Then
$\psi_*(z,\ell)\le\psi(z,\ell)\le\sigma(z)\psi_*(z,\ell)$.

\begin{prop} \label{prop:d}
There exists $M_1,M_2>0$ with the property that
if $\sigma < d_0\log\psi$ then $\sigma <M_1\log\psi_*+M_2$.
\end{prop}

\begin{proof}
Since $\sigma < d_0\log\psi$, it follows that
$\sigma < d_0\log(\sigma\psi_*)=d_0\log\sigma+d_0\log\psi_*$.  Hence
$\sigma(1-d_0(\log\sigma)/\sigma) < d_0\log\psi_*$ and the result follows.
\end{proof}

From now on, we fix $M_1,M_2>0$ sufficiently large as in Proposition~\ref{prop:d}.  Also we fix $d_*>0$ such that
$M_1\log n+M_2\le d_*\log n$ for all $n\ge2$.  
For $q>0$, define
\[
X_q(n)=\big\{x\in X:[\tau_0(x)]=n\enspace \text{and}\enspace \tau_0(f^jx)\ge n^{1-q} \enspace\text{for some $1\le j< d_*\log n$}\big\}.
\]

\begin{prop} \label{prop:Xn}
Assume that the density $d\mu_X/d\mu_Z$ is bounded below on $Z$.
Suppose that there exists $\beta'\in(\beta,1]$, $q\in(0,1)$,  such that
$\sum_{n=1}^\infty n^{\beta'} \mu_X(X_q(n))<\infty$.
Then (A1)(iii) holds for any 
$\beta_+\in(\beta,\min\{\beta',\beta/(1-q)\})$.
\end{prop}

\begin{proof}
For $n\ge2$, define
\[
Y(n)=\big\{(z,\ell)\in Y: [\psi_*(z,\ell)]=n
\enspace\text{and}\enspace
 \sigma(z)<d_*\log n\enspace\text{and}\enspace \tau(z,\ell)\ge n^{1-q} 
\big\}.
\]
Also, set $Y(1)=\{(z,\ell)\in Y:[\psi_*(z,\ell)]=1\;\text{and}\;\sigma(z)< M_1\log 2+M_2\}$.
By Proposition~\ref{prop:d},
\begin{align*}
& |1_{\{\sigma<d_0\log\psi\}}  1_{\{\tau\ge\psi^{1-q}\}}I(s)|_1
  \le |1_{\{\sigma<M_1\log\psi_*+M_2\}}1_{\{\tau\ge\psi_*^{1-q}\}}I(s)|_1
\\ & \qquad  \le 
\textstyle |1_{\{[\psi_*]=1\}}1_{\{\sigma< M_1\log2+M_2\}}I(s)|_1+ \sum_{n\ge 2} |1_{\{[\psi_*]=n\}}1_{\{\sigma<d_*\log n\}}1_{\{\tau\ge n^{1-q}\}}I(s)|_1
\\ & \qquad \textstyle  =\sum_{n\ge 1} |1_{Y(n)} I(s)|_1.
\end{align*}
Let $\beta''<\beta'$.  For $y=(z,\ell)\in Y(n)$,
\begin{align*}
|I(s)(y)|  \le 4|s|^{\beta''}\psi(y)^{\beta''}
& \le 4|s|^{\beta''}\sigma(z)^{\beta''}\psi_*(y)^{\beta''}
\ll |s|^{\beta''}(\log (n+1))^{\beta''}n^{\beta''}.
\end{align*}
It follows that
\[
|1_{\{\sigma<d_0\log\psi\}}  1_{\{\tau\ge\psi^{1-q}\}}I(s)|_1
 \ll 
\textstyle |s|^{\beta''}\sum_{n\ge 1} (\log(n+1))^{\beta''} n^{\beta''} \mu(Y(n)).
\]
Provided
$\sum_{n\ge1} (\log (n+1))^{\beta''}n^{\beta''} \mu(Y(n))<\infty$, we obtain from~\eqref{eq:Yn1} and~\eqref{eq:Yn2} the desired estimate
for $\int_Y I(s)\,d\mu$.

It remains to show that 
$\sum_{n\ge 1} n^{\beta''} \mu(Y(n))<\infty$ for all $\beta''<\beta'$.
For this, it is enough to show for $n\ge2$ that $\mu(Y(n))\ll (\log n)^2\mu_X(X_q(n))$.
Let $y=(z,\ell)\in Y(n)$. Then
$\tau_0(f^\ell z)\ge n^{1-q}$ and
$[\tau_0(f^j z)]= n$ for some $j$ with $0\le j<\ell<d_*\log n$.
Hence $f^jz\in X_q(n)$, and it follows that
\[
Y(n)\subset\big\{(z,\ell)\in Y:z\in A_n,\, 1\le\ell<d_*\log n\big\},
\]
where $A_n=\bigcup_{j=0}^{[d_*\log n]-1}f^{-j}X_q(n)$.
Since $d\mu_Z/d\mu_X\in L^\infty(Z)$,
\begin{align*}
\mu(Y(n)) &  = \bar\sigma^{-1}\int_Z\sum_{\ell=0}^{\sigma(z)-1}1_{\{z\in Z:(z,\ell)\in Y(n)\}}\,d\mu_Z
\le \int_Z\sum_{1\le\ell<d_*\log n}1_{\{z\in Z:z\in A_n\}}\,d\mu_Z
\\ & \le d_*(\log n)\mu_Z(A_n)
 \ll (\log n)\mu_X(A_n)
  \\ & \le (\log n)\sum_{j=0}^{[d_*\log n]-1}\mu_X(f^{-j}X_q(n))
  <d_*(\log n)^2\mu_X(X_q(n)),
\end{align*}
as required.
\end{proof}

\begin{cor} \label{cor:SV}
Assume that the density $d\mu_X/d\mu_Z$ is bounded below on $Z$.
Suppose that there are constants $C>0$, $p,q\in(0,1)$, such that
\begin{align} \label{eq:SV}
& \mu_X(X_q(n)) \le Cn^{-p}\mu_X(x\in X:[\tau_0(x)]=n)
\quad\text{for all $n\ge1$.}
\end{align}
Then (A1)(iii) holds for any  $\beta_+\in(\beta,\min\{1,\beta+p,\beta/(1-q)\})$.
\end{cor}

\begin{proof}
Let $\beta'\in(\beta,\beta+p)$.  Then
\[
\textstyle \sum_{n\ge 1} n^{\beta'}\mu(X_q(n))
\le C\sum_{n\ge 1} n^{\beta'-p}\mu_X([\tau_0]=n)
<\infty,
\]
since $\mu_X(\tau_0>n)=O(n^{-\beta})$.
\end{proof}

Condition~\eqref{eq:SV} has been verified in the setting of infinite horizon planar dispersing billiards by~\cite[Lemma 16]{SzaszVarju07} (see also~\cite[Lemma 5.1]{ChernovZhang08}).
For Example~\ref{ex:Collet}, it is easier to work with the following related condition.
Define the two new subsets
$X_q^+(n)\subset X$ and $X_q'(n)\subset X$,
\begin{align*}
X_q^+(n) & =\big\{[\tau_0]\ge n\enspace \text{and}\enspace \tau_0\circ f^j\ge n^{1-q} \enspace\text{for some $1\le j<d_*\log n$}\big\}.
\\
X_q'(n) & =\big\{[\tau_0]\ge n\enspace \text{and}\enspace \tau_0\circ f^j\ge n^{1-q} \enspace\text{for some $d_*\log n \le   j<d_*\log (n+1)$}\big\}.
\end{align*}

\begin{prop} \label{prop:A1iii}
Assume that the density $d\mu_X/d\mu_Z$ is bounded below on $Z$.
Suppose that there are constants $C>0$, $p,q\in(0,1)$, such that
\begin{align} \label{eq:A1iii}
 \mu_X(X^+_q(n)) \le Cn^{-(\beta+p)}
\quad\text{and}\quad
 \mu_X(X'_q(n)) \le Cn^{-(\beta+p)},
\end{align}
for all $n\ge1$.
Then (A1)(iii) holds for any  $\beta_+\in(\beta,\min\{1,\beta+p,\beta/(1-q)\})$.
\end{prop}

\begin{proof}
First note that $X_q^+(n+1)\subset X_q^+(n)\cup X_q'(n)$.
Hence
\begin{align*}
\mu_X(\{X_q^+(n)\cup X_q'(n)\}\setminus X_q^+(n+1))
 & =\mu_X(X_q^+(n)\cup X_q'(n))-\mu(X_q^+(n+1)
\\ & \le \mu_X(X_q^+(n)) -\mu(X_q^+(n+1)) +\mu(X_q'(n)).
\end{align*}
Also, $X_q(n)\subset 
X_q^+(n)\setminus X_q^+(n+1)\subset 
\{X_q^+(n)\cup X_q'(n)\}\setminus X_q^+(n+1)$,
so
\begin{align*} 
\mu_X(X_q(n)) \le \mu_X(X_q^+(n)) -\mu(X_q^+(n+1)) +\mu(X_q'(n)).
\end{align*}

We claim that the series
\[
\textstyle S_1=\sum_{n\ge1} n^{\beta'}\{\mu_X(X_q^+(n)) -\mu(X_q^+(n+1))\}
\quad\text{and}\quad
S_2=\sum_{n\ge1} n^{\beta'}\mu(X_q'(n))
\]
converge for $\beta'\in(\beta,\beta+p)$.
The result then follows from Proposition~\ref{prop:Xn}.

It remains to verify the claim.
Convergence of $S_1$ follows
since $\mu_X(X_q^+(n))\le Cn^{-(\beta+p)}$.
Next, we note that $X_q'(n)=\emptyset$ unless there exists an
integer $j\in[d_*\log n,d_*\log(n+1))$, in which case
$n=n(j)=[e^{j/d_*}]$.
Hence
\begin{align*}
S_2  = \textstyle \sum_{j\ge1}n(j)^{\beta'}\mu(X'_q(n(j)))
& \textstyle \ll \sum_{j\ge1}n(j)^{\beta'}n(j)^{-(\beta+p)}
\\ & \textstyle  = \sum_{j\ge1}[e^{j/d_*}]^{-(\beta+p-\beta')} 
\\ & \ll \sum_{j\ge1}e^{-j(\beta+p-\beta')/d_*} <\infty,
\end{align*}
completing the proof.
\end{proof}

\section{Verification of hypotheses for Example~\ref{ex:AFN}}
\label{sec:AFN}

In this section, we return to 
Example~\ref{ex:AFN} and show that Theorem~\ref{thm:rate}
applies in this case under the assumption that $\tau_0$ is both of bounded variation and H\"older continuous.

The first step is to pass from the original suspension semiflow on
$X^{\tau_0}$ to a suspension of the form
$Y^\tau$ where $Y$ is a probability space and $\tau$ is an nonintegrable roof function.

We take $Y$ to be the interval of domain of the rightmost branch of $f$.
Define $F = f^{\sigma_0}:Y \to Y$ for $\sigma_0 = \min\{ n \ge1 : f^ny \in Y\}$.
Then $\mu=(\mu_X|Y)/\mu_X(Y)$ is an acip for $F$.
The corresponding roof function $\tau:Y \to\R^+$ is given by
$\tau(y) = \sum_{\ell=0}^{\sigma_0(y)-1} \tau_0(f^\ell y)$.
Let $F_t:Y^\tau\to Y^\tau$ be the corresponding suspension semiflow
with infinite invariant measure $\mu^\tau$.

In \cite[Section 9]{BT15}, taking $Z=Y$, a reinduced full branch Gibbs-Markov map $G = F^\sigma:Z\to Z$ is obtained, with inducing
time $\sigma:Z \to \Z^+$ having exponential tails.
Let $\mu_Z$ denote the unique acip.

\begin{prop} \label{prop:AFN1}
Suppose that $\tau_0$ is H\"older with exponent
$\eta\in(0,1)$.  Then
$\mu(\tau > t) = c t^{-\beta} + O(t^{-(1+\beta)\eta})$ for some $c>0$.
\end{prop}

\begin{proof}
Let $\sigma_Z(y) = \sum_{j=0}^{\sigma(y)-1} \sigma_0(F^jy)$.
By~\cite[Lemma 9.1]{BT15},
the tails of $\sigma_Z$ satisfy
\[
\mu_Z(\sigma_Z > n) = c_Z n^{-\beta} + O(n^{-2\beta}) \text{ for some } c_Z > 0.
\]
By~\cite[Lemma 9.2]{BT15},
$\mu(\sigma_0>n)=\mu_Z(\sigma_Z>n)+ O(n^{-(1+\beta)})$,
and hence
\[
\mu(\sigma_0 > n) = c_F n^{-\beta} + O(n^{-2\beta}) \text{ for some } c_F > 0.
\]
This corresponds to~\cite[Condition~(2.3) in the proof of Proposition 2.7]{MT17} and we can now proceed as in the remainder of the proof of~\cite[Proposition 2.7]{MT17}.
\end{proof}

The map $F$ is uniformly expanding, but nonMarkov in general.
Hence we take $\cB(Y) = BV(Y)$
with norm $\|v\|=|v|_\infty+\Var v$
where $\Var$ denotes the variation on $Y$.
In particular, $\cB(Y)$ is embedded in $L^\infty(Y)$.

\begin{prop} \label{prop:AFN2}
Suppose that $\tau_0$ is of bounded variation and $C^\eta$ for some $\eta>0$.
Then conditions (H1)--(H3) are satisfied.
\end{prop}

\begin{proof}
(H2): The properties of the reinduced map $G$ are given in \cite[Section 9]{BT15}. For instance, (H2)(ii) and (H2)(iv) follow
from uniform expansion of $F$, and (H2)(iii) from uniform expansion of $F$ combined with
Adler's condition.

\vspace{1ex}
(H3): 
Let $y,y'\in Y$ with $\sigma_0(y)=\sigma_0(y')$.
Since $f'\ge1$ and $\tau_0\ge2$,
\begin{align} \label{eq:tau} \nonumber
|\tau(y)-\tau(y')|
& \le |\tau_0|_\eta \sum_{\ell=0}^{\sigma_0(y)-1}|f^\ell y-f^\ell y'|^\eta
\\ & \le|\tau_0|_\eta \sigma_0(y)|Fy-Fy'|^\eta
\le {\textstyle\frac12} |\tau_0|_\eta \tau(y)|Fy-Fy'|^\eta.
\end{align}

Next let, $z,z'\in a$, $a\in\alpha$.
It follows from~\cite[Section~9]{BT15} that
$\sigma_0(F^\ell z)=\sigma_0(F^\ell z')$ for all 
$0\le \ell\le\sigma(a)-1$.
Hence for $1\le \ell<\sigma(a)$,
\begin{align*}
|\tau_\ell(z) & -\tau_\ell(z')|   \le \sum_{j=0}^{\ell-1}|\tau(F^jz)-\tau(F^jz')|
  \le |\tau_0|_\eta  \sum_{j=0}^{\sigma(a)-1}\tau(F^jz)|F^{j+1}z-F^{j+1}z'|^\eta
 \\ &  \le |\tau_0|_\eta  \sum_{j=0}^{\sigma(a)-1}\tau(F^jz)|Gz-Gz'|^\eta
 = 
|\tau_0|_\eta  \varphi(z) |Gz-Gz'|^\eta
\ll  \infa\varphi |Gz-Gz'|^\eta.
\end{align*}
This is condition (H3).

\vspace{1ex}
(H1): It is standard that $\cB(Y) = BV(Y)$ is compactly embedded in $L^1(Y)$.
Since $\cB(Y)$ is embedded in $L^\infty(Y)$ it follows immediately that
$\cB(Y)$ is compactly embedded in $L^p(Y)$ for all $p$.
Indeed suppose that $v_n\in\cB(Y)$ with $\|v_n\|_{\cB(Y)}=1$.
There exists $v\in\cB(Y)$ with $\|v\|_{\cB(Y)}=1$ and a subsequence $n_k$ such that $v_{n_k}\to v$ in $L^1(Y)$.
Then $|v_{n_k}-v|_\infty\ll \|v_{n_k}-v\|_{\cB(Y)}\le 2$, so
for any $p\ge1$,
\[
\textstyle \int_Y|v_{n_k}-v|^p\,d\mu\le |v_{n_k}-v|^{p-1}_\infty
\int_Y|v_{n_k}-v|\,d\mu
\ll |v_{n_k}-v|_1\to0,
\]
establishing compactness in $L^p(Y)$.
 Also it is standard that for $s=0$ the Lasota-Yorke inequality holds for $p=1$ and hence all $p\ge1$, completing the verification of~(H1)(i).

Let $P:L^1(Y)\to L^1(Y)$ be the ``unnormalized'' transfer operator
given by $\int_Y Pv\,w\,d\Leb=\int_Y v\,w\circ F\,d\Leb$.
The density $d\mu/d\Leb$ lies in $\cB(Y)$ and is bounded above and below.
Recall that
\[
Pv=\sumI \{Jv\}\circ\psi_I 1_{FI},
\]
where $\{I\}$ is the collection of domains corresponding to branches of $F$,
$J=1/|F'|$ and $\psi_I=(F|_I)^{-1}$.

By standard arguments,
\begin{align*}
\|Pv\|_{\cB(Y)} & \le |Pv|_\infty+\Var(Pv)\le \sumI(3\supI(Jv)+\Var_I(Jv))
\\ & \ll \sumI \mu(I)(\supI v+\Var_Iv),
\end{align*}
where $\Var_Iv$ denotes the variation of $v$ on $I$.
Hence
\begin{align} \label{eq:sub} \nonumber
\|(\hat R(s)-\hat R(0))v\|_{\cB(Y)}
& =\|R((e^{-s\tau}-1)v)\|_{\cB(Y)}
\ll \|P((e^{-s\tau}-1)v)\|_{\cB(Y)}
\\ \nonumber & \ll \sumI \mu(I)\{\supI ((e^{-s\tau}-1))v+\Var_I((e^{-s\tau}-1)v)\}
\\ & \ll \|v\|_{\cB(Y)}\sumI \mu(I)\{\supI (e^{-s\tau}-1)+\Var_I(e^{-s\tau}-1)\}.
\end{align}

Since the images $f^\ell I$ are disjoint for $\ell=0,1,\dots, \sigma_0(I)-1$,
it follows that $\Var_I\tau\le \Var\tau_0$.
Also, $|e^{-s\tau(y)}-e^{-s\tau(y')}|\le |s||\tau(y)-\tau(y')|$ for all
$y,y'\in Y$, $s\in\barH$, so
\[
\Var_I(e^{-s\tau}-1)=
\Var_I(e^{-s\tau})\le |s|\Var_I\tau
\le |s|\Var\tau_0,
\quad\text{for all $s\in\barH$.}
\]

Also, it follows as in~\eqref{eq:tau} that
$\supI\tau-\infI\tau\ll \sigma_0|_I\ll \infI\tau$, so
$\supI(e^{-s\tau}-1)\le2
|s|^{\beta-\eps}\supI\tau^{\beta-\eps}
\ll |s|^{\beta-\eps}\infI\tau^{\beta-\eps}$.
Substituting into~\eqref{eq:sub}, for $s\in\barH$ bounded,
\begin{align*}
\|\hat R(s)-\hat R(0)\|_{\cB(Y)}
 & \ll |s|^{\beta-\eps}\sumI \mu(I)\{\infI \tau^{\beta-\eps}+\Var\tau_0\}
 \\ & \textstyle
\le |s|^{\beta-\eps}\{\int_Y \tau^{\beta-\eps}\,d\mu+\Var\tau_0\}
\ll |s|^{\beta-\eps}.
\end{align*}
 Condition~(H1)(ii) now follows from Remark~\ref{rmk:H1}.

Similarly,
\begin{align*}
|R(\tau^{\beta-\eps}v)|_\infty & \ll
|P(\tau^{\beta-\eps}|v|)|\le \sumI \supI(J\tau^{\beta-\eps}|v|)
\\ & \textstyle \ll|v|_\infty\sumI\mu(I) \infI\tau^{\beta-\eps} \le |v|_\infty\int_Y\tau^{\beta-\eps}\,d\mu \ll \|v\|_{\cB(Y)},
\end{align*}
verifying (H1)(iii).
\end{proof}

By Propositions~\ref{prop:AFN1} and~\ref{prop:AFN2},
we can apply Theorem~\ref{thm:rate} with $\cB(Y)=BV(Y)$
and $q=(1+\beta)\eta$ whenever (H4) is satisfied.
If we suppose moreover that $\tau_0$ is $C^1$ and satisfies (UNI),
then the improved result in Theorem~\ref{thm:UNI} applies with $q=1+\beta$.

\section{Verification of hypotheses for Example~\ref{ex:Collet}}
\label{sec:Collet}

In this section, we return to 
Example~\ref{ex:Collet} and show how to apply Theorem~\ref{thm:rateA}.

Recall that $X=[0,1]$ and $f:X\to X$ is a $C^2$ unimodal map, so there is a unique critical point $x_0\in(0,1)$ and $f$ is strictly increasing on $[0,x_0]$ and strictly decreasing on $[x_0,1]$.  We suppose for convenience that
the critical point $x_0$ is nondegenerate, but the general case of non-flat critical points works similarly.
We suppose further that $f$ is {\em Collet-Eckmann}~\cite{ColletEckmann83}: there are constants $C>0$, $\lambda_{{\textsc{\tiny CE}}}>1$ such that $|(f^n)'(fx_0)|\ge C\lambda_{{\textsc{\tiny CE}}}^n$ for all $n\ge1$.
It follows~\cite{Jak81} that there is a unique acip $\mu_X$ that is mixing up to a finite cycle.  We restrict to the case when $\mu_X$ is mixing.
Finally, we suppose that $x_0$ satisfies {\em slow recurrence} in the sense
that $\lim_{n\to\infty}n^{-1}\log |f^nx_0-x_0|=0$.

\begin{rmk}
There are many maps within the logistic family $f_a:x \mapsto ax(1-x)$
satisfying the above conditions.
By~\cite{Jak81}, the set $S\subset [0,4]$ of parameters such that 
$f_a$ does not have an attracting periodic orbit has positive Lebesgue measure.
By~\cite[Theorems A and B]{AM05},
the Collet-Eckmann and slow recurrence conditions hold for a.e.\ $a\in S$.
\end{rmk}

By eg.~\cite{BLS03,Young98}, $f$ is modelled by a Young tower $\pi_Y:Y=Z^\sigma\to X$ as described in Section~\ref{sec:Young}, where the inducing time $\sigma:Z\to\Z^+$ has exponential 
tails.\footnote{This is not the first return time to $Z$ except when $x_0$ is 
non-recurrent, i.e., the Misiurewicz condition holds.} 

Next, recall from Example~\ref{ex:Collet} that the roof function $\tau_0:X\to\R^+$ has the form $\tau_0(x)=g(x)|x-x_0|^{-1/\beta}$ where $\beta\in(\frac12,1)$ and $g:[0,1]\to(1,\infty)$ is differentiable.

\begin{lemma} \label{lem:Collet}
For maps $f$ and roof functions $\tau_0$ as defined above,
\begin{itemize}
\item[(a)] 
$\mu_X(\tau_0 > t) = c t^{-\beta}+O(t^{-2\beta})$
where $c>0$ is a constant given explicitly below.
\item[(b)] (H2) and (H3) hold.
\item[(c)] Hypothesis (A1) holds for the function space defined in Section~\ref{sec:Young}.  In (A1)(iii), we can take any $\beta_+$ in the range $\beta<\beta_+<\min\big\{1,\frac12(\sqrt 5+1)\beta\big\}$.
\end{itemize}
\end{lemma}

Hence, the required rates of mixing in Example~\ref{ex:Collet} follow from Theorem~\ref{thm:rateA} provided (H4) is satisfied,
with $d_1=\frac{1}{c\pi}\sin\beta\pi$.

We begin by proving Lemma~\ref{lem:Collet}(a).
The density $h=d\mu_X/d\Leb$ is bounded below and has the form
\begin{align}\label{eq:hCE}
\textstyle h(x) = \sum_{n\ge 1} h_n(x) |(f^n)'(fx_0)|^{-1/2} |x-f^nx_0|^{-1/2},
\end{align}
where the $h_n$ are $C^1$ functions supported on one-sided neighbourhoods of $f^nx_0$,
and are uniformly bounded in $n$. 
An upper bound of this form was obtained by~\cite{Now93}.  The simplest way to obtain the precise expression for $h$ is to observe that
\[
 \textstyle h(x) = \sum_{n \ge 2} \tilde h_n(x) \,  |(f^n)'(fx_0)|^{-1/2} (|x-f^nx_0|^{-1/2} + |x-f^{b(n)-1}x_0|^{-1/2})
 \]
 where $\tilde h_n(x) |(f^n)'(fx_0)|^{-1/2} (|x-f^nx_0|^{-1/2} + |x-f^{b(n)-1}x_0|^{-1/2})$
 is the density of $\mu_X$ lifted to the $n$-th level of the Hofbauer tower,
 cf.\ Keller \cite{Kel89}. (Here $b(n) < n$ is an integer determined by the combinatoric properties of $f$.)

Lemma~\ref{lem:Collet}(a)
now follows from the choice of roof function $\tau_0$ and~\eqref{eq:hCE} with
\[
\textstyle c = g(x_0)^\beta  \sum_{n \ge 1} h_n(x_0) |(f^n)'(fx_0)|^{-1/2} |x_0-f^nx_0|^{-1/2} .
\]
(Convergence of this series follows from the Collet-Eckmann and slow recurrence conditions.)

\begin{lemma}
(H2) and (H3) are satisfied.
\end{lemma}

\begin{proof}
The existence of inducing schemes for Collet-Eckmann maps is well-known,
see e.g.\ \cite{Jak81, BLS03}, and (H2) follows automatically, but we review the particular construction following \cite{Bruin95}
to ensure that (H3) holds.

Let $Z,\,Z'$ be intervals with $x_0\subset \operatorname{Int}Z$  and $\bar Z\subset \operatorname{Int}Z'$ such that 
trajectories starting in $f(\partial Z)\cup\partial Z'$ do not intersect $\operatorname{Int}Z'$, cf.~\cite[p.~330]{dMvS}.
For a.e.\ $z \in Z$, one can find a minimal integer $\sigma(z) \ge 1$
for which
there are intervals $a\subset a'\subset Z$ with $z\in\operatorname{Int}a$ 
such that
$f^{\sigma(z)}:a' \to Z'$ is a diffeomorphism and $f^{\sigma(z)}(a) = Z$.
Moreover, for any two points, the corresponding intervals $a$ coincide or are disjoint (so $\sigma$
is constant on each $a$).
Define $\alpha=\{a\}$ and $G=f^\sigma$.
Standard Koebe distortion arguments
(see \cite[Chapter IV, Theorem 1.2]{dMvS}) ensure that
\[
\Big| \frac{G'(z)}{G'(z')} - 1 \Big| \le C|Gz-Gz'|
\quad\text{for all $z,z'\in a$,}
\]
for a constant $C$ depending only on $f$ and the intervals $Z,\,Z'$.
It follows from the Collet-Eckmann condition that $\Leb(z\in Z:\sigma(z)>n)=O(e^{-dn})$ for some $d>0$, see e.g.\ \cite{BLS03}.
Now $G= f^\sigma:Z \to Z$ is the required induced Gibbs-Markov map,
and there exists an acip $\mu_Z$ with density above and below.
Moreover, there exists $\theta_0 \in (0,1)$ such that
$|Gz-Gz'|\ll d_{\theta_0}(Gz,Gz')=\theta_0^{-1}d_{\theta_0}(z,z')$
for all $z,z'\in a$, $a\in\alpha$.

The choices of $Z\subset Z'$ and $a \subset a'$
allow us to show (again using standard Koebe distortion estimates)
that there exists $\delta > 0$
such that the components of $f^\ell(a'\setminus a)$ have length at least
$\delta |f^\ell a|$ for each $a \in \alpha$ and $0 \le \ell < \sigma(a)$.
Let $z\in a$.
Since $f^{\sigma(a)-\ell}:f^\ell a'\to Z'$ is a diffeomorphism, $x_0\not\in f^\ell a'$.  Also $f^\ell z\in f^\ell a$ and $\overline{f^\ell a}\subset \operatorname{Int}f^\ell a'$.
Hence the interval
$[f^\ell z,x_0]$ contains a component of $f^\ell(a'\setminus a)$, so
$|f^\ell z-x_0|\ge \delta|f^\ell a|$.  Using Koebe distortion once more for 
$f^{\sigma(a)-\ell}:f^\ell a\to Z$,
\begin{align*} 
 \frac{|f^\ell z-f^\ell z'|}{|f^\ell z-x_0|} 
 \le \frac{|f^\ell z-f^\ell z'|}{\delta|f^\ell a|}
\ll
 \frac{|Gz-Gz'|}{|Z|} \ll
 d_{\theta_0}(z,z')
\enspace\text{for all $z,z'\in a$, $0\le\ell<\sigma(a)$.}
\end{align*}
Also, for $x,y\in X$ with $|x-x_0|\ge |y-y_0|$,
\begin{align*}
|\tau_0(x) & - \tau_0(y)|  \le 
|g(x)-g(y)||x-x_0|^{-1/\beta}\,+\,|g(y)|\big||x-x_0|^{-1/\beta}-|y-x_0|^{-1/\beta}\big|
\\ & \ll |x-x_0|^{-1/\beta}|x-y|\,+\,|x-x_0|^{-1/\beta}|y-x_0|^{-1/\beta}\big||y-x_0|^{1/\beta}-|x-x_0|^{1/\beta}\big|
\\ & \ll \tau_0(x)\big\{|x-y|\,+\,|y-x_0|^{-1/\beta}\big||y-x_0|^{1/\beta}-|x-x_0|^{1/\beta}\big|\big\}
\\ & \ll 
 \tau_0(x)\big\{ |x-y|+|y-x_0|^{-1/\beta}|y-x_0|^{1/\beta-1}|x-y|\big\}
 \ll \tau_0(x)|x-y||y-x_0|^{-1}.
\end{align*}
Hence $|\tau_0(f^\ell z)-\tau_0(f^\ell z')| \ll 
\tau_0(f^\ell z)|f^\ell z-f^\ell z'|/|f^\ell z'-x_0|\ll\tau_0(f^\ell z)d_{\theta_0}(z,z')$, and
\begin{align*}
 |\tau_\ell(z) - \tau_\ell(z')| &\le  \sum_{j=0}^{\ell-1} |\tau_0(f^j z) - \tau_0(f^j z')| 
 \ll \sum_{j=0}^{\ell-1} \tau_0(f^j z) d_{\theta_0}(z,z')
  \le \varphi(z) d_{\theta_0}(z,z').
 \end{align*}
This is (H3).
\end{proof}

\begin{lemma}\label{lem:SV}
Let $p,q\in(0,1)$ with $(1-q)\beta>p$.  Then there exists $C>0$ such that the
conditions in~\eqref{eq:A1iii} hold.
\end{lemma}

\begin{proof}
Choose $\alpha>0$ small so that  $p'=(1-q) \beta-\alpha d_*> p$.
 Let $U_n = \{ x \in X : [\tau_0(x)] \ge n\}$, so $U_n$ is an interval containing $x_0$ of radius $O(n^{-\beta})$.

We claim that there is a constant $C>0$ such that
\[
\mu_X(U_n\cap f^{-j}U_{n^{1-q}})\le Cn^{-(\beta+p')}
\quad\text{for all $n\ge1$, $j< d_*\log n$.}
\]
Then
$ \mu_X(X_q^+(n))\le 
\sum_{1\le j< d_*\log n}\mu_X(U_n\cap f^{-j}U_{n^{1-q}}) \ll n^{-(\beta+p)}$,
 and
$\mu_X(X_q'(n))\le \sum_{d_*\log n\le j<d_*\log(n+1)}\mu_X(U_n\cap f^{-j}U_{n^{1-q}}) \ll n^{-(\beta+p)}$ as required.

It remains to verify the claim.
 Fix $j < d_* \log n$, and let $J$ be a maximal interval in
$U_n\cap f^{-j} U_{n^{1-q}}$ on which $f^j$ is monotone. 
For each such $J$ let $J'$ be the maximal interval, $J\subset J'\subset X$,
 on which $f^j$ is monotone.
Note that the intervals $J'$ obtained in this way are disjoint.

Since $f^jJ\subset U_{n^{1-q}}$ and $J\subset J'\cap U_n$, it follows that
$f^j(J'\cap U_n)$ intersects $U_{n^{1-q}}$.
Hence we can choose $y\in f^j(J'\cap U_n)$ with $|y-x_0|\ll n^{-(1-q)\beta}$.

Since $J'$ is a maximal interval of monotonicity for $f^j$, 
there exist $k_1, k_2 \le j$ such that $f^jJ' = [f^{k_1}x_0, f^{k_2}x_0]$.
Note that $f^j$ is not monotone on $U_n$ so without loss $f^{k_1}x_0\in f^j U_n$.
By slow recurrence, 
$|f^{k_1}x_0-x_0| \gg e^{-\alpha d_* \log n} = n^{-\alpha d_*}$.
Hence 
\[
|f^j(J'\cap U_n)|\ge |f^{k_1}x_0-y|\ge |f^{k_1}x_0-x_0|-|y-x_0|\gg  n^{-\alpha d_*} -n^{-(1-q)\beta} \gg n^{-\alpha d_*}.
\]
By the Koebe distortion lemma, there is a constant $C$ such that
\begin{align*}
\frac{|J|}{|J'\cap U_n|} \le C \frac{|f^jJ|}{|f^j(J'\cap U_n)|} \ll
 \frac{n^{-(1-q) \beta}}{n^{-\alpha d_*}}=n^{-p'}.
\end{align*}
Hence
\[
\textstyle |U_n\cap f^{-j}U_{n^{1-q}}|=\sum_J|J|\ll n^{-p'}\sum_J |J'\cap U_n|
\le n^{-p'}|U_n|\ll n^{-(\beta+p')}.
\]

Let $I=U_n\cap f^{-j}U_{n^{1-q}}$.
We have shown that $|I|\ll n^{-(\beta+p')}$ and we require that
$\mu_X(I)\ll n^{-(\beta+p')}$.
By~\eqref{eq:hCE} and the Collet-Eckmann condition,
\begin{align*}
\mu_X(I)=\int_I h(x)\,dx \ll 
\sum_{i\ge 1} \lambda_{{\textsc{\tiny CE}}}^{-i/2} \int_I|x-f^ix_0|^{-1/2}dx.
\end{align*}
Note that $\int_I|x-f^ix_0|^{-1/2}dx\le \min\{ |I|d(f^ix_0, I)^{-1/2},\, 2|I|^{1/2}\}$,
where $d(f^ix_0, I)$ denotes the distance between $f^ix_0$ and $I$.
Fix $k>1$ large.  Then $\mu_X(I)\ll S_1+S_2$ where
\begin{align*}
S_1 & = |I|\sum_{1\le i\le k\log n}\lambda_{{\textsc{\tiny CE}}}^{-i/2}d(f^ix_0, I)^{-1/2}
\ll n^{-(\beta+p')}\sum_{1\le i\le k\log n}\lambda_{{\textsc{\tiny CE}}}^{-i/2}d(f^ix_0, I)^{-1/2},
\\
S_2 & =|I|^{1/2}\sum_{i>k\log n}\lambda_{{\textsc{\tiny CE}}}^{-i/2}\ll |I|^{1/2}n^{-(k/2)\log\lambda_{{\textsc{\tiny CE}}}} \ll n^{-(\beta+p')/2}n^{-(k/2)\log\lambda_{{\textsc{\tiny CE}}}}.
\end{align*}
Clearly, $S_2\ll n^{-(\beta+p')}$ for $k$ sufficiently large.

Choose $\delta>0$ small.  (In fact, the choices $k=(\beta+p')/\log\lambda_{{\textsc{\tiny CE}}}$, $\delta=\beta/(2k)$ suffice.)  By slow recurrence, $|f^ix_0-x_0|\gg e^{-\delta i}$.
Also $x_0\in U_n$, $I\subset U_n$ and $|U_n|\ll n^{-\beta}$, so $d(f^ix_0,I)\gg e^{-\delta i}-n^{-\beta}$.
For $i\le k\log n$ with $k$ sufficiently large, $d(f^ix_0,I)\gg e^{-\delta i}$.
Hence $S_1
\le n^{-(\beta+p')}\sum_{i\ge1} (\lambda_{{\textsc{\tiny CE}}}^{-1}e^{\delta })^{i/2}\ll n^{-(\beta+p')}$.
We conclude that $\mu_X(I)\ll n^{-(\beta+p')}$ and the proof is complete.
\end{proof}

Now take $p=\frac12(\sqrt 5-1)\beta$ and $(1-q)\beta =p-\eps$.
Then Lemma~\ref{lem:SV} applies and the range 
$\beta<\beta_+<\min\big\{1,\frac12(\sqrt 5+1)\beta\big\}$
in Lemma~\ref{lem:Collet}(c) follows from Proposition~\ref{prop:A1iii}.

\appendix

\section{Correlation function of a suspension semiflow}
\label{sec:rhohat}

In this appendix, we verify formula~\eqref{eq:rhohat} following~\cite{Pollicott85}.

\begin{prop} \label{prop:rhohat}
$\hat\rho_{v,w}(s)  =\hat J(s)+\int_Y \hat T(s)v_s\,w_s\,d\mu$
for all $s\in\H$.
\end{prop}

\begin{proof}
First observe that $\rho_{v,w}(t)=\sum_{n=0}^\infty K_n(t)$, where
\begin{align*}
K_n(t) & =\int_{Y^\tau} 1_{\{\tau_n(y)<t+u<\tau_{n+1}(y)\}}v(y,u)\,w\circ F_t(y,u)\,d\mu^\tau 
\\ & = \int_{Y^\tau} 1_{\{\tau_n(y)<t+u<\tau_{n+1}(y)\}}v(y,u)w(F^ny,t+u-\tau_n(y))\,d\mu^\tau.
\end{align*}
For all $n\ge0$,
\begin{align*}
\hat K_n(s) & 
= \int_0^\infty e^{-st}\int_{Y^\tau} 1_{\{\tau_n(y)-u<t<\tau_{n+1}(y)-u\}}v(y,u)w(F^ny,t+u-\tau_n(y))\,d\mu^\tau\,dt.
\end{align*}

When $n\ge1$, we have $u<\tau(y)\le \tau_n(y)$ for all $(y,u)\in Y^\tau$, and hence
\begin{align*}
\hat K_n(s) & 
 =\int_Y\int_0^{\tau(y)}\int_{\tau_n(y)-u}^{\tau_{n+1}(y)-u}e^{-st}v(y,u)w(F^ny,t+u-\tau_n(y))\,dt\,du\,d\mu.
\end{align*}
The substitution $u'=t+u-\tau_n(y)$ yields
\begin{align*}
\hat K_n(s) & =\int_Y\Bigl(\int_0^{\tau(y)}e^{su}v(y,u)\,du\Bigr)
\Bigl(\int_0^{\tau(F^ny)}e^{-su'}w(F^ny,u')\,du'\Bigr)e^{-s\tau_n(y)}\,d\mu
\\ & 
= \int_Y e^{-s\tau_n}v_s\, w_s\circ F^n\,d\mu
= \int_Y \hat R(s)^nv_s\, w_s\,d\mu.
\end{align*}
Also,
\begin{align*}
\hat K_0(s) & 
 =\int_Y\int_0^{\tau(y)}\int_0^{\tau(y)-u}e^{-st}v(y,u)w(y,t+u)\,dt\,du\,d\mu \\
 & =\Big(\int_Y\int_0^{\tau(y)}\int_0^{\tau(y)}
 -\int_Y\int_0^{\tau(y)}\int_0^u\Big)e^{su}v(y,u)e^{-st}w(y,t)\,dt\,du\,d\mu 
 \\
& = \int_Y v_s\,w_s\,d\mu + \hat J(s).
\end{align*}
Hence
$\hat\rho_{v,w}(s)=\hat J(s)+\sum_{n=0}^\infty \int_Y \hat R(s)^nv_s\,w_s\,d\mu
=\hat J(s)+\int_Y (I-\hat R(s))^{-1}v_s\,w_s\,d\mu$.
\end{proof}

\section{Quasicompactness for Young towers}
\label{app:Young}

In this section, we show that (A1)(i),(ii) are satisfied for the function space
$\cB(Y)$ defined in Section~\ref{sec:Young}.
We use the facts that (i) $\tau^p\in L^1$ for all $p<\beta$, 
(ii) $\sigma$ has exponential tails, (iii) $G$ is a full branch Gibbs-Markov map, (iv) $\tau$ satisfies (H3).

Recall that $\cB(Y)$ depends on $\beta',\,\theta,\,\eps$.
Throughout, we fix $\beta'\in(0,\beta)$ and $\theta\in[\theta_0^{\beta'},1)$.

\renewcommand{\thesubsection}{\Alph{section}.\arabic{subsection}}

\subsection{Compact embedding}

In this subsection, we verify (A1)(i).

\begin{prop} \label{prop:AA}
Let $p\in(1,\infty)$. 
Choose $\eps>0$ so that $\int_Z e^{\eps p\sigma}\,d\mu_Z<\infty$.
Then $\cB(Y)$ is compactly embedded in $L^p(Y)$.
\end{prop}

\begin{proof}
We have
\begin{align*}
|v|_p^p & \ll \sum_{a\in\alpha}\int_a\sum_{\ell=0}^{\sigma(a)-1}|1_{a\times\{\ell\}}v|_\infty^p\,d\mu_Z
\le \|v\|_{w,\infty}^p\sum_a\mu_Z(a)e^{\eps p \sigma(a)}\ll \|v\|_{w,\infty}^p
\le \|v\|_{\cB(Y)}^p.
\end{align*}
Hence $\cB(Y)$ is embedded in $L^p(Y)$.
Compactness of the embedding is a standard Arzel\`a-Ascoli argument.
Since the setting is nonstandard, we provide the details.

For $k\ge1$, define the partition $\beta_k=\{a\times\{\ell\}: a\in\alpha_k,\,0\le\ell\le\sigma(a)-1\}$ of $Y$.
For each $b\in\bigcup_{k\ge1}\beta_k$,
choose $y_b\in b$.

Let $v_n\in\cB(Y)$ with $\|v_n\|_{\cB(Y)}\le1$.
For each $b=a\times\{\ell\}$, the real
sequence $v_n(y_b)$ is bounded by $e^{\eps\ell}$ and hence has a convergent subsequence.
By a diagonal argument, we can suppose without loss that
$v_n(y_b)$ is convergent for all $b\in \bigcup_{k\ge1}\beta_k$.

Let $y\in Y_\ell$ and choose $b=a\times\{\ell\}\in\beta_k$ containing $y$.
Then $d_\theta(y,y_b)\le\diam a=\theta^k$ and
\begin{align} \label{eq:emb}
|v_n(y) & -v_m(y)|  \le |v_n(y)-v_n(y_b)| + |v_n(y_b)-v_m(y_b)|
+|v_m(y_b)-v_m(y)|
\\ & \le 2e^{\eps\ell}\varphi(y)^{\beta'}d_\theta(y,y_b)+|v_n(y_b)-v_m(y_b)|
\le 2e^{\eps\ell}\varphi(y)\theta^k+|v_n(y_b)-v_m(y_b)|. \nonumber
\end{align}
It follows that $\limsup_{m,n\to\infty}|v_n(y)-v_m(y)|\le 2e^{\eps\ell}\varphi(y)\theta^k$.
Since $k$ is arbitrary, $v_n(y)$ is a Cauchy sequence for each $y$.
Hence there is a function $v:Y\to\R$ such that $v_n\to v$ pointwise.
Moreover, for $y,y'\in Y_\ell$ it is immediate that $v$ inherits from $v_n$ the properties $|v(y)|\le e^{\eps\ell}$ and $|v(y)-v(y')|\le e^{\eps\ell}
\varphi(y)^{\beta'}d_\theta(y,y')$, so $\|v\|_{\cB(Y)}\le1$.

It remains to show that $|v_n-v|_p\to0$.  Let $\delta>0$ and fix $k\ge1$ such that
$\diam b<\delta$ for all $b\in\beta_k$.
By~\eqref{eq:emb},
\[
|v_n(y)-v(y)|\le 2e^{\eps\ell}\varphi(y) \delta+|v_n(y_b)-v(y_b)|
\quad\text{for all $y\in b=a\times\{\ell\}$, $a\in\alpha_k$, $\ell,\,n\ge1$.}
\]
Choose finitely many cylinders $b_j=a_j\times\{\ell_j\}$, $j=1,\dots,m$, with $a_j\in \alpha_k$ such that
$\mu(b_1\cup\cdots\cup b_m)>1-\delta$ and let $\ell_*=\max\{\ell_1,\dots,\ell_m\}$,
$\varphi_*=\sup_{b_1\cup\cdots\cup b_m}\varphi$.
Then for $n$ sufficiently large,
\[
|v_n(y)-v(y)|\le 3e^{\eps \ell_*}\varphi_* \delta
\quad\text{for all $y\in b_1\cup\cdots \cup b_m$}.
\]
Also
\[
|v_n(y)-v(y)|\le |v_n(y)|+|v(y)|\le 2e^{\eps\ell}\le 2e^{\eps\sigma(y)}
\quad\text{for all $y\in b=a\times\{\ell\}\in\beta_k$.}
\]
Hence, writing $Y'=Y\setminus(b_1\cup\cdots\cup b_m)$,
\begin{align*}
|v_n-v|_p^p &
\le \sum_{j=1}^m \int_{b_j}|v_n-v|^p\,d\mu+
\sum_{b\in Y'}\int_b|v_n-v|^p\,d\mu
 \le
(3e^{\eps \ell_*}\varphi_* \delta)^p+2^p
\!\int_{Y'} e^{\eps p\sigma}\,d\mu.
\end{align*}
Since $\mu(Y')<\delta$
and $e^{\eps p\sigma}\in L^1(Y)$,
we have that $\int_Y|v_n-v|^p\,d\mu\ll
\delta^p+g(\delta)$ where $\lim_{\delta\to0}g(\delta)=0$.
Hence $|v_n-v|_p\to0$ as $n\to\infty$.
\end{proof}

\subsection{Lasota-Yorke inequality for $\hat R(s)$}

In this subsection, we verify (A1)(ii).
Choose $\eps>0$ so that $\int_Z e^{\eps \sigma}\varphi^{\beta'}\,d\mu_Z<\infty$.

\begin{theorem} \label{thm:LY}
There exists $\gamma_0\in(0,1)$ and $C>0$ such that
\[
\|\hat R(s)^nv\|_{\cB(Y)}\le C(|v|_{L^1(Y)}+\gamma_0^n \|v\|_{\cB(Y)}),
\]
for all $s\in \barH\cap B_1(0)$, $v\in \cB(Y)$, $n\ge1$.
\end{theorem}

Define 
\[
\hat T(s)_n:L^1(Z)\to L^1(Z), \qquad \hat T(s)_nv=1_Z\hat R(s)^n(1_Zv).
\]
Let $\cF_\theta(Z)$ denote the usual space of observables $v:Z\to\R$ that are $d_\theta$-Lipschitz with norm
$\|v\|_{\cF_\theta(Z)}=|v|_\infty+|v|_{\cF_\theta(Z)}$ where
$|v|_{\cF_\theta(Z)}=\sup_{z\neq z'}|v(z)-v(z')|/d_\theta(z,z')$.

The key step is to estimate $\hat T(s)_n:\cF_\theta(Z)\to\cF_\theta(Z)$.  
For this, we need the following technical lemma.

\begin{lemma} \label{lem:tech}
There exists $\kappa>0$ and $\gamma_0\in(0,1)$ such that
\begin{itemize}
\item[(a)]
$\sum_{m=1}^n\theta^m\mu_Z(\sigma_m=n)=O(\gamma_0^n)$.
\hspace{1em} (b) $\mu_Z(\sigma_{\kappa n}\ge n)=O(\gamma_0^n)$.
\end{itemize}
\end{lemma}

\begin{proof} (a) 
Let $g=1_Z$, so $g_n(z)=\sum_{j=0}^{n-1}g(F^jz)$ denotes the number of returns to $Z$ by time~$n$.
By~\cite[Lemma~3.6]{BalintGouezel06},
$\int_{F^{-n}Z} \theta^{g_n(z)}e^{\eps\ell} d\mu(z,\ell)
=O(\gamma_0^n)$ for $\eps$ sufficiently small.  In particular,
$\int_{Z\cap F^{-n}Z} \theta^{g_n}\,d\mu_Z=O(\gamma_0^n)$.
On $Z\cap F^{-n}Z$, we have $g_n\in\{1,\dots,n\}$, and $g_n=m$ if and only if $\sigma_m=n$.  Hence
\[
\int_{Z\cap F^{-n}Z} \theta^{g_n}\,d\mu_Z=\sum_{m=1}^n
\int_{Z\cap F^{-n}Z} 1_{\{\sigma_m=n\}} \theta^m\,d\mu_Z
=\sum_{m=1}^n\theta^m\mu_Z(\sigma_m=n),
\]
and the result follows.

\vspace{1ex}
\noindent (b)
This is the main estimate obtained in the proof of~\cite[Lemma~3.6]{BalintGouezel06} (see the estimate for the set $\Gamma_n$ defined therein).
\end{proof}

\begin{lemma} \label{lem:LY}
There exists $\gamma_0\in(0,1)$ and $C>0$ such that
\[
\|\hat T(s)_nv\|_{\cF_\theta(Z)}\le C(|v|_{L^1(Z)}+\gamma_0^n \|v\|_{\cF_\theta(Z)}),
\]
for all $s\in \barH\cap B_1(0)$, $v\in \cF_\theta(Z)$, $n\ge1$.
\end{lemma}

\begin{proof}
The entire proof is on $Z$ so we write $|v|_1$, $|v|_\theta$ and $\|v\|_\theta$
instead of $|v|_{L^1(Z)}$, $|v|_{\cF_\theta(Z)}$, and
$\|v\|_{\cF_\theta(Z)}$.

Let $\alpha_{m,n}=\{a\in\alpha_m:\sigma_m(a)=n\}$.
Then 
$Z\cap F^{-n}Z=\bigcup_{m=1}^n\bigcup_{a\in\alpha_{m,n}}a$.
Hence for  $z\in Z$,
\begin{align*}
(\hat T(s)_nv)(z) & =\big(\hat R(s)^n(1_Zv)\big)(z)=\sum_{z'\in Z:F^nz'=z}J_n(z')e^{-s\tau_n(z')}v(z')
\\ & =\sum_{m=1}^n \sum_{a\in\alpha_{m,n}} \xi_m(z_a)e^{-s\tau_n(z_a)}v(z_a).
\end{align*}
(Here $J$ is the Jacobian for $F$ and $J_n=\prod_{j=0}^{n-1}J\circ F^j$.
Recall from Remark~\ref{rmk:GM} that $\xi$ is the Jacobian for $G$ and
$\xi_m=\prod_{j=0}^{m-1}\xi\circ G^j$.)

It follows from~\eqref{eq:GM} and 
Lemma~\ref{lem:tech}(a) that
\begin{align*}
|\hat T(s)_nv|_\infty & \le 
\sum_{m=1}^n \sum_{a\in\alpha_{m,n}} \supa\xi_m\, \supa |v|
 \ll \sum_{m=1}^n \sum_{a\in\alpha_{m,n}} \mu_Z(a) (\infa |v|+\theta^m|v|_{\theta})
\\ & \le \sum_{m=1}^n |1_{\{\sigma_m=n\}}v|_1
+ |v|_{\theta} \sum_{m=1}^n\theta^m \mu_Z(\sigma_m=n)
 \ll |v|_1
+ \gamma_0^n |v|_{\theta}.
\end{align*}

Next,
$|(\hat T(s)_nv)(z)- (\hat T(s)_nv)(z')| \le \sum_{m=1}^n\sum_{a\in\alpha_{m,n}}(I_1+I_2+I_3)$,
where
\begin{align*}
|I_1| & = |\xi_m(z_a)-\xi_m(z_a')|\supa|v|, \qquad
|I_2|  = \supa\xi_m|v(z_a)-v(z_a')|, \\
|I_3| & = \supa\xi_m\,\supa|v| |e^{-s\tau_n(z_a)}-e^{-s\tau_n(z_a')}|.
\end{align*}
By~\eqref{eq:GM},
\begin{align*}
|I_1| & \ll  \mu_Z(a)d_\theta(z,z')(\infa|v|+\theta^m|v|_{\theta}), \qquad
 |I_2|  \ll  \mu_Z(a)|v|_{\theta}\theta^m d_\theta(z,z').
\end{align*}
By Lemma~\ref{lem:tech}(a), these contribute 
$d_\theta(z,z')(|v|_1+\gamma_0^n|v|_{\theta})$ and
$d_\theta(z,z')\gamma_0^n|v|_{\theta}$ respectively to the sum.
Also,
\begin{align*}
 |I_3| & \ll  
\mu_Z(a)(\infa|v|+|v|_{\theta}\theta^m d_\theta(z,z'))
|e^{-s\tau_n(z_a)}-e^{-s\tau_n(z_a')}|=I_3'+I_3'',
\end{align*}
where
$I_3'  \le 2 \mu_Z(a)|v|_{\theta}\theta^m d_\theta(z,z')$
which contributes
 $d_\theta(z,z')\gamma_0^n|v|_{\theta}$ to the sum, 
and
\[
I_3''  =  
\mu_Z(a)\infa|v| |e^{-s\tau_n(z_a)}-e^{-s\tau_n(z_a')}|
\le 2\mu_Z(a)\infa|v| |s|^{\beta'}|\tau_n(z_a)-\tau_n(z_a')|^{\beta'}.
\]

We show below that 
$S=\sum_{m=1}^n \sum_{a\in\alpha_{m,n}} I_3''\ll d_\theta(z,z')(|v|_1+\gamma_0^n\|v\|_{\theta})$.  Then $|\hat T(s)_nv|_{\theta}\ll |v|_1+\gamma_0^n\|v\|_{\theta}$ and the result follows.

It remains to estimate $S$.
For $x,y\in a$, $a\in\alpha_{m,n}$,
\begin{align*} 
 |\tau_n(x) &  -\tau_n(y)|
  =|\varphi_m(x)-\varphi_m(y)|
\le \sum_{j=0}^{m-1}|\varphi(G^jx)-\varphi(G^jy)|
\\ & \ll \sum_{j=0}^{m-1}\infa(\varphi\circ G^j)\,d_{\theta_0}(G^jx,G^jy)
= \sum_{j=0}^{m-1}\infa(\varphi\circ G^j)\,{\theta_0}^{m-j}d_{\theta_0}(F^nx,F^ny).
\end{align*}
Recall that $\theta\ge\theta_0^{\beta'}$, so
$|\tau_n(x) -\tau_n(y)|^{\beta'}\ll \sum_{j=0}^{m-1}\infa(\varphi^{\beta'}\circ G^j)\,{\theta}^{m-j}d_{\theta}(F^nx,F^ny)$.
Hence
\begin{align*}
S & 
 \ll |s|^{\beta'} \sum_{m=1}^n\sum_{a\in\alpha_{m,n}} \mu_Z(a)\infa|v|
\sum_{j=0}^{m-1}\infa(\varphi^{\beta'}\circ G^j)\,\theta^{m-j}d_{\theta}(z,z')
 \le d_{\theta}(z,z')Q,
\end{align*}
where 
\begin{align*}
Q & = \sum_{m=1}^n\sum_{j=0}^{m-1} \theta^{m-j}
\big| 1_{\{\sigma_m=n\}}|v| \varphi^{\beta'}\circ G^j\big|_1
  = \sum_{m=1}^n\sum_{j=1}^{m} \theta^j
\big|1_{\{\sigma_m=n\}}|v| \varphi^{\beta'}\circ G^{m-j}\big|_1
\\  & = \sum_{j=1}^n\sum_{m=j}^{n} Q_{j,m}, \qquad Q_{j,m}=
\theta^j \big|1_{\{\sigma_m=n\}}|v| \varphi^{\beta'}\circ G^{m-j}\big|_1.
\end{align*}
We claim that there exists $\kappa >0$ and $\gamma_0\in(0,1)$ such that
\begin{itemize}
\item[(i)] $Q_{j,m}=O(\theta^{\kappa  n}|v|_\infty)$
for $j\ge \kappa  n$,
\hspace{1.6em}(ii) $Q_{j,m}=O(\gamma_0^n|v|_\infty)$ for $m\le 2\kappa  n$,
\item[(iii)] $Q_{j,m}=O\big(n\theta^{\kappa  n}|v|_{\theta}
+\theta^j\sum_{k=0}^n |1_{\{\sigma_{m-j}=k\}}v|_1
|1_{\{\sigma_j=n-k\}}\varphi^{\beta'}|_1\big)
$ for $m-j\ge \kappa  n$.
\end{itemize}
It then follows that
\begin{align*}
Q & \ll \gamma_1^n\|v\|_{\theta}
+\sum_{j=1}^n\theta^j\sum_{k=0}^n\sum_{m=j}^{n} |1_{\{\sigma_{m-j}=k\}}v|_1
|1_{\{\sigma_j=n-k\}}\varphi^{\beta'}|_1
\\ & \ll \gamma_1^n\|v\|_{\theta}+|v|_1|\varphi^{\beta'}|_1
 \ll \gamma_1^n\|v\|_{\theta}+|v|_1,
\end{align*}
for some $\gamma_1\in(0,1)$,
yielding the desired estimate for $S$.

It remains to verify the claim.
Choose $r>1$ with $r\beta'<\beta$ and conjugate exponent~$r'$.
By H\"older's inequality, 
\[
\big|1_{\{\sigma_m=n\}}|v| \varphi^{\beta'}\circ G^{m-j}\big|_1
\ll \mu_Z(\sigma_m=n)^{1/r'}|v|_\infty|\varphi^{\beta'}|_r
\ll \mu_Z(\sigma_m=n)^{1/r'}|v|_\infty.
\]
Estimate (i) is immediate, and 
estimate (ii) follows by Lemma~\ref{lem:tech}(b) for $\kappa $ sufficiently small.
Let $q=m-j$ and write
\begin{align*}
\big|1_{\{\sigma_m=n\}}|v| \varphi^{\beta'}\circ G^q\big|_1
& = \big|1_{\{\sigma_q+\sigma_j\circ G^q=n\}}|v| \varphi^{\beta'}\circ G^q\big|_1
\\ & \le \sum_{k=0}^n 
\big|1_{\{\sigma_q=k\}}1_{\{\sigma_j=n-k\}}\circ G^q|v| \varphi^{\beta'}\circ G^q\big|_1.
\end{align*}
Let $R^G$ be the transfer operator for $G$.  Then 
\begin{align*}
\big|1_{\{\sigma_q=k\}}1_{\{\sigma_j=n-k\}}\circ G^q|v| \varphi^{\beta'}\circ G^q\big|_1 & = 
\big|(R^G)^q (1_{\{\sigma_q=k\}}|v|) 1_{\{\sigma_j=n-k\}}\varphi^{\beta'}\big|_1
\\ & \le  
|(R^G)^q(1_{\{\sigma_q=k\}}|v|)|_\infty |1_{\{\sigma_j=n-k\}}\varphi^{\beta'}|_1.
\end{align*}
We have
\begin{align*}
|(R^G)^q & (1_{\{\sigma_q=k\}}  |v|)|_\infty
 \le \textstyle \sum_{a\in\alpha_q} \supa\xi_q 1_{\{\sigma_q(a)=k\}}\supa |v|
\\ & \ll \textstyle \sum_{a\in\alpha_q} \mu_Z(a) 1_{\{\sigma_q(a)=k\}}
(\infa |v|+\theta^q|v|_{\theta})
 \le |1_{\{\sigma_q=k\}}v|_1+\theta^q|v|_{\theta}.
\end{align*}
For $q=m-j\ge\kappa  n$, it follows that
\begin{align*}
\big|1_{\{\sigma_m=n\}}|v| \varphi^{\beta'}\circ G^j\big|_1
&\textstyle  \le \sum_{k=0}^n \big(|1_{\{\sigma_{m-j}=k\}}v|_1 +\theta^{\kappa  n}|v|_{\theta}\big)
|1_{\{\sigma_j=n-k\}}\varphi^{\beta'}|_1
\\ &\textstyle  \ll 
n\theta^{\kappa  n}|v|_{\theta}
+\sum_{k=0}^n |1_{\{\sigma_{m-j}=k\}}v|_1
|1_{\{\sigma_j=n-k\}}\varphi^{\beta'}|_1.
\end{align*}
Hence we obtain estimate (iii) completing the proof of the claim.
\end{proof}

Following~\cite{GouezelPhD}, we have the decomposition\footnote{The continuous time analogue of this was used in Section~\ref{sec:rhoG}.}
\begin{align} \label{eq:op}
\hat R(s)^n=\sum_{n_1+n_2+n_3=n}\hat A(s)_{n_1}\hat T(s)_{n_2}\hat B(s)_{n_3} + \hat E(s)_n,
\end{align}
where the operators
\begin{align*}
& \hat A(s)_n:\cF_\theta(Z)\to \cB(Y), \qquad
  \hat B(s)_n:\cB(Y)\to \cF_\theta(Z), \qquad
\hat E(s)_n: \cB(Y)\to \cB(Y),
\end{align*}
are given by
\begin{alignat*}{2}
(\hat A(s)_nv)(x) & =\sum_{\stackrel{F^n\!y=x,\,y\in Z}{Fy\dots,F^n\!y\not\in Z}}\omega(s,n)(y),
& \qquad
& (\hat E(s)_nv)(x)   =\sum_{\stackrel{F^n\!y=x}{ y,\dots, F^n\!y\not\in Z}}\omega(s,n)(y),
\\
(\hat B(s)_nv)(x) & =\sum_{\stackrel{F^n\!y=x,\, F^n\!y\in Z}{ y,\dots,F^{n-1}\!y\not\in Z}}\omega(s,n)(y),
& \qquad
& \omega(s,n)  =J_n e^{-s\tau_n}v.
\end{alignat*}

\newpage
\begin{prop} \label{prop:ABE}
For all $s\in\barH\cap B_1(0)$, $n\ge1$, 
\begin{itemize}
\item[(i)] $\|\hat A(s)_n\|_{\cF_\theta(Z)\to\cB(Y))}=O(e^{-\eps n})$.
\hspace{1em} (ii) $\|\hat E(s)_n\|_{\cB(Y)\to\cB(Y)}=O(e^{-\eps n})$.
\item[(iii)] $\|\hat B(s)_n\|_{\cB(Y)\to\cF_\theta(Z)}=O(e^{-\eps n})$.
\hspace{1em} (iv)
$\sum_{n=0}^\infty |\hat B(s)_n|_{L^1(Y)\to L^1(Z)}\le \bar\sigma$.
\end{itemize}
\end{prop}

\begin{proof}
\noindent(i)
Note that $(\hat A(s)_n v)(z,\ell)\equiv0$ for all $\ell\neq n$
and $(\hat A(s)_nv)(z,n)=e^{-s\tau_n(z,0)}v(z)$.
Let $y=(z,n),\,y'=(z',n)\in a\times\{n\}$, $a\in\alpha$.
Then $\tau_n(z,0)=\sum_{\ell=0}^{n-1}\tau_0(f^\ell z)$ and
$|\tau_n(z,0)-\tau_n(z',0)|\ll \varphi(z)d_{\theta_0}(z,z')$ by (H3) since $n<\sigma(a)$.
Hence,
\[
|(\hat A(s)_nv)(y)|=|v(z)|\le |v|_\infty \le \|v\|_{\cF_\theta(Z)},
\]
and
\begin{align*}
|(\hat A(s)_nv)(y)-(\hat A(s)_nv) &(y')|
  \le |e^{-s\tau_n(z,0)}-e^{-s\tau_n(z',0)}||v(z)|+|v(z)-v(z')| \\ & \le 
 2|s|^{\beta'}|\tau_n(z,0)-\tau_n(z',0)|^{\beta'}|v(z)|+
|v|_{\cF_\theta(Z)} d_\theta(z,z') \\ & \ll 
\varphi(z)^{\beta'}d_\theta(z,z')\|v\|_{\cF_\theta(Z)}
=\varphi(y)^{\beta'}d_\theta(y,y')\|v\|_{\cF_\theta(Z)}.
\end{align*}
It follows that $\|\hat A(s)_nv\|_{\cB(Y)}\ll e^{-\eps n}\|v\|_{\cF_\theta(Z)}$.

\vspace{1ex}
\noindent(ii)
Note that $(\hat E(s)_nv)(z,\ell)\equiv0$ if $\ell\le n$, while 
$(\hat E(s)_nv)(z,\ell)=e^{-s\tau_n(z,\ell-n)}v(z,\ell-n)$ for $\ell>n$.
Let $y=(z,\ell),\,y'=(z',\ell)\in a\times\{\ell\}$, where $a\in\alpha$ and $n<\ell<\sigma(a)$.
Then
\[
|(\hat E(s)_nv)(y)|=|v(z,\ell-n)|\le \|v\|_{w,\infty} e^{\eps(\ell-n)}.
\]
Also, $\tau_n(z,\ell-n)=\sum_{j=\ell-n}^{\ell}\tau_0(f^jz)$, so
$|\tau_n(z,\ell-n)-\tau_n(z',\ell-n)|\ll \varphi(z)d_{\theta_0}(z,z')$ by (H3) since
$0\le\ell-n<\ell<\sigma(a)$.  Hence
\begin{align*}
|(\hat E(s)_nv)(y)- & (\hat E(s)_nv)(y')|
 \le 2|s|^{\beta'}|\tau_n(z,\ell-n)-\tau_n(z',\ell-n)|^{\beta'}|v(z,\ell-n)|
\\
& +|v(z,\ell-n)-v(z',\ell-n)|
 \ll e^{\eps(\ell-n)} \varphi(z)^{\beta'} d_\theta(y,y')\|v\|_{\cB(Y)},
\end{align*}
and so $\|\hat E(s)_nv\|_{\cB(Y)}\ll e^{-\eps n}\|v\|_{\cB(Y)}$.

\vspace{1ex}
\noindent(iii)
Note that 
$(\hat B(s)_nv)(z)=\sum \xi(z_a)e^{-s\tau_n(z_a,\sigma(a)-n)}v(z_a,\sigma(a)-n)$,
 where
the sum is over $a\in\alpha$ with $\sigma(a)>n$.
By Remark~\ref{rmk:GM}, since $\int_Z e^{\eps\sigma}\,d\mu_Z<\infty$,
\[
|(\hat B(s)_nv)(z)|\ll 
\textstyle \sum_{a\in\alpha} \mu_Z(a)\|v\|_{w,\infty}e^{\eps (\sigma(a)-n)}
\ll e^{-\eps n}\|v\|_{w,\infty}.
\]
Also,
\[
|(\hat B(s)_nv)(z)-(\hat B(s)_nv)(z')| \le I_1+I_2+I_3,
\]
where
\begin{align*}
I_1 & = \sum (\xi(z_a)-\xi(z'_a))e^{-s\tau_n(z_a,\sigma(a)-n)}v(z_a,\sigma(a)-n), \\
I_2 & = \sum \xi(z'_a)
(e^{-s\tau_n(z_a,\sigma(a)-n)}- e^{-s\tau_n(z'_a,\sigma(a)-n)}) v(z_a,\sigma(a)-n), \\
I_3 & = \sum \xi(z'_a)e^{-s\tau_n(z'_a,\sigma(a)-n)}
(v(z_a,\sigma(a)-n) -v(z'_a,\sigma(a)-n)).
\end{align*}
Now, $\tau_n(z_a,\sigma(a)-n)=\sum_{\ell=\sigma(a)-n}^{\sigma(a)-1}\tau_0(f^\ell z)$ so
$|\tau_n(z_a,\sigma(a)-n)-\tau_n(z'_a,\sigma(a)-n)|\ll \infa\varphi\, d_{\theta_0}(z,z')$ by (H3).
Since $\int_Z e^{\eps \sigma}\varphi^{\beta'}\,d\mu_Z<\infty$,
\begin{align*}
|I_1| & \ll \sum \mu_Z(a)d_\theta(z,z')\|v\|_{w,\infty}e^{\eps(\sigma(a)-n)}\ll d_\theta(z,z')\|v\|_{w,\infty}e^{-\eps n}, \\
|I_2| & \ll \sum \mu_Z(a)\infa\varphi^{\beta'}d_\theta(z,z')\|v\|_{w,\infty}e^{\eps(\sigma(a)-n)} \ll d_\theta(z,z')\|v\|_{w,\infty}e^{-\eps n}, \\
|I_3| & \ll \sum \mu_Z(a)\|v\|_{\theta}e^{\eps(\sigma(a)-n)}\infa\varphi^{\beta'} d_\theta(z,z') \ll d_\theta(z,z')\|v\|_{\theta}e^{-\eps n},
\end{align*}
so
\[
|(\hat B(s)_nv)(z)-(\hat B(s)_nv)(z')| \ll e^{-\eps n}\|v\|_{\cB(Y)}d_\theta(z,z'),
\]
and $\|\hat B(s)_nv\|_{\cF_\theta(Z)}\ll e^{-\eps n}\|v\|_{\cB(Y)}$.

\vspace{1ex}
\noindent(iv)
Define $\psi:Y\to\N$, $\psi(y)=\min\{n\ge0:F^ny\in Z\}$.
Then $|\hat B(s)_nv|=|\hat R(s)^n(1_{\{\psi=n\}}v)|\le R^n(1_{\{\psi=n\}}|v|)$.
Hence
$|\hat B(s)_nv|_{L^1(Z)}\le \bar\sigma |1_{\{\psi=n\}}|v||_{L^1(Y)}$.
\end{proof}

\begin{pfof}{Theorem~\ref{thm:LY}}
By Proposition~\ref{prop:ABE} and Lemma~\ref{lem:LY},
it follows from~\eqref{eq:op} that
\begin{align*}
\| & \hat R(s)^nv\|_{\cB(Y)}
 \le \sum_{n_1+n_2+n_3=n} e^{-\eps n_1}\|T_{n_2}B_{n_3}v\|_{\cF_\theta(Z)}+e^{-\eps n}\|v\|_{\cB(Y)}
\\ & \le C\sum_{n_1+n_2+n_3=n} e^{-\eps n_1}|B_{n_3}v|_{L^1(Z)}\,
+C\sum_{n_1+n_2+n_3=n} e^{-\eps n_1}\gamma_0^{n_2}\|B_{n_3}v\|_{\cF_\theta(Z)}
+e^{-\eps n}\|v\|_{\cB(Y)}
\\ & \ll \sum_{n_1+n_3\le n} e^{-\eps n_1}|B_{n_3}v|_{L^1(Z)}\,
+\sum_{n_1+n_2+n_3=n} e^{-\eps n_1}\gamma_0^{n_2}e^{-\eps n_3}\|v\|_{\cB(Y)}
+e^{-\eps n}\|v\|_{\cB(Y)}
\\ & \ll \textstyle  \sum_{n_3\ge 0}|B_{n_3}v|_{L^1(Z)}\, +\gamma_1^{n}\|v\|_{\cB(Y)},
\end{align*}
for some $\gamma_1\in(0,1)$.
Finally, it follows from Proposition~\ref{prop:ABE}(iv) that
$\sum_{n_3\ge 0}|B_{n_3}v|_{L^1(Z)}\ll |v|_{L^1(Y)}$.
\end{pfof}

\paragraph{Acknowledgements}
The research of IM was supported in part by a
European Advanced Grant {\em StochExtHomog} (ERC AdG 320977).
We are grateful for the support of the Erwin Schr\"odinger International Institute 
for Mathematical Physics at the University of Vienna, where part of this research was carried out. We are very grateful to the referee for pointing out numerous inaccuracies and typos, and for very helpful suggestions that greatly enlarged the scope of the paper.


\begin{thebibliography}{99}

\bibitem{Aaronson}
J.~Aaronson. \emph{An Introduction to Infinite Ergodic Theory}. Math. Surveys
  and Monographs \textbf{50}, Amer. Math. Soc., 1997.

\bibitem{AaronsonDenker01}
J.~Aaronson, M.~Denker.
{Local limit theorems for partial sums of stationary
  sequences generated by Gibbs-Markov maps.} 
\emph{Stoch.\ Dyn.} \textbf{1} (2001) 193--237.

\bibitem{ABV16}
V.~Ara{\'u}jo, O.~Butterley, P.~Varandas. Open sets of axiom {A} flows with
  exponentially mixing attractors. \emph{Proc. Amer. Math. Soc.} \textbf{144}
  (2016) 2971--2984.

\bibitem{AraujoMapp}
V.~Ara{\'u}jo, I.~Melbourne. Exponential decay of correlations for
  nonuniformly hyperbolic flows with a $C^{1+\alpha}$ stable foliation,
  including the classical Lorenz attractor. \emph{Annales Henri Poincar\'e}
  \textbf{17} (2016) 2975--3004.

\bibitem{AM05} A.\ Avila, C.\ Moreira.
Statistical properties of unimodal maps: the quadratic family.
{\em Ann.\ of Math.} {\bf 161} (2005) 831--881.

\bibitem{AGY} A.\ Avila, S.\ Gou\"ezel, J.-C.\ Yoccoz.
{Exponential mixing for the Teichm\"uller flow.} 
\emph{Publ.\ Math.\ Inst.\ Hautes \'Etudes Sci.} {\bf 104} (2006) 143--211. 

\bibitem{BDL}
V. Baladi, M. Demers, C. Liverani.
Exponential decay of correlations for finite horizon Sinai billiard flows.
{\em Invent. Math.} \textbf{211} (2018) 39--177.

\bibitem{BL14}
V.~Baladi, C.~Liverani.
{Exponential decay of correlations for piecewise cone hyperbolic 
contact flows.} 
\emph{Comm.\ Math.\ Phys.} \textbf{314} (2012) 689--773.

\bibitem{BalVal}
V. Baladi, B. Vall\'ee.
{Exponential decay of correlations for surface semi-flows without finite Markov partitions.}
\emph{Proc.\ Amer.\ Math.\ Soc.} \textbf{133} (2005) 865–-874.

\bibitem{BalintGouezel06}
P.~B{\'a}lint, S.~Gou{\"e}zel. Limit theorems in the stadium billiard.
  \emph{Comm. Math. Phys.} \textbf{263} (2006) 461--512.

\bibitem{BC85} M.\ Benedicks, L.\ Carleson.
On the iterations of $x \mapsto a-x^2$ on $(-1,1)$.
\emph{Ann.\ of Math.} {\bf 122} (1985) 1--25

\bibitem{Bruin95} H.\ Bruin.
Induced maps, Markov extensions and invariant measures in one-dimensional dynamics.
\emph{Comm.\ Math.\ Phys.} {\bf 168} (1995) 571--580. 

\bibitem{BLS03} H.\ Bruin, S.\ Luzzatto, S.\ van Strien.
{Decay of correlations in one-dimensional dynamics.}
 \emph{Ann.\ Sci.\ Ec.\ Norm.\ Sup.} {\bf 36} (2003) 621--646.

\bibitem{BT15} H.\ Bruin, D.\ Terhesiu.
{Upper and lower bounds for the correlation function via inducing with general return times}.
\emph{Ergodic Theory Dynam.\ Systems} \textbf{38} (2018) 34--62.

\bibitem{BMMW}
K.\ Burns, H.\ Masur, C.\ Matheus, A.\ Wilkinson.
Rates of mixing for the Weil-Petersson geodesic flow II: exponential mixing in exceptional moduli spaces.  
\emph{Geom. Funct. Anal. (GAFA)}  \textbf{27} (2017) 240--288. 

\bibitem{BW}
O.\ Butterley, K.\ War.
Open sets of exponentially mixing Anosov flows.
To appear in {\em J. Eur. Math. Soc.}

\bibitem{ChernovZhang08}
N.~I. Chernov, H.-K. Zhang. {Improved estimates for correlations in
  billiards}. \emph{Comm. Math. Phys.} \textbf{77} (2008) 305--321.

\bibitem{ColletEckmann83}
P.~Collet, J.-P. Eckmann, Positive {L}iapunov exponents and absolute
  continuity for maps of the interval, \emph{Ergodic Theory Dynam. Systems}
  \textbf{3} (1983) 13--46.


\bibitem{Dolgopyat98a} D.~Dolgopyat.
{On the decay of correlations in Anosov flows}.
\emph{Ann.\ of Math.} \textbf{147} (1998) 357--390.

\bibitem{Dolgopyat98b} D.~Dolgopyat. 
{Prevalence of rapid mixing in hyperbolic flows}. 
\emph{Ergodic Theory Dynam.\ Systems} \textbf{18} (1998) 1097--1114.

\bibitem{FMT07}
M.~J. Field, I.~Melbourne, A.~T{\" o}r{\" o}k. {Stability of mixing and
  rapid mixing for hyperbolic flows}. \emph{Ann.\ of Math.} \textbf{166} (2007)
  269--291.

\bibitem{GarsiaLamperti62} A.~Garsia, J.~Lamperti. 
A discrete renewal theorem with infinite mean.
  \emph{Comment.\ Math.\ Helv.} \textbf{37} (1962/1963) 221--234.

\bibitem{GouezelPhD}
S.~Gou{\"e}zel. {Vitesse de d\'ecorr\'elation et th\'eor\`emes limites pour les
  applications non uniform\'ement dilatantes}. {\mbox{Ph.\ D.} Thesis}. Ecole
  Normale Sup\'erieure, 2004.

\bibitem{Gouezel04} S.~Gou{\"e}zel. 
{Sharp polynomial estimates for the decay of correlations}.
  \emph{Israel J.\ Math.} \textbf{139} (2004) 29--65.

\bibitem{Gouezel10b} S.~Gou{\"e}zel. 
Characterization of weak convergence of {B}irkhoff sums for
 {G}ibbs-{M}arkov maps. \emph{Israel J. Math.} \textbf{180} (2010) 1--41.

\bibitem{Gouezel11}
S.~Gou{\"e}zel. Correlation asymptotics from large deviations in dynamical
  systems with infinite measure. \emph{Colloq. Math.} \textbf{125} (2011)
  193--212.

\bibitem{Jak81} M.\ Jakobson.
{Absolutely continuous invariant measures for one-parameter families of one-dimensional maps.}
\emph{Comm.\ Math.\ Phys.} {\bf 81} (1981) 39--88.


\bibitem{Kel89} G.\ Keller.
Lifting measures to Markov extensions.  \emph{Monatsh. Math. }
\textbf{108} (1989) 183--200.

\bibitem{KellerLiverani99} G.~Keller, C.~Liverani. 
Stability of the spectrum for transfer operators.
\emph{Annali della Scuola Normale Superiore di Pisa, Classe di Scienze} \textbf{19} (1999) 141--152.

\bibitem{Liverani04} C.~Liverani. {On contact Anosov flows.}
\emph{Ann.\ of Math.} {\bf 159}  (2004) 1275--1312.

\bibitem{LT} C.\ Liverani, D.\ Terhesiu.  
{Mixing for some non-uniformly hyperbolic systems}. 
\emph{Annales Henri Poincar{\'e}} {\bf 17} (2016) 179--226.

\bibitem{Mar94} M.\ Martens.
{Distortion results and invariant Cantor sets of unimodal maps.}
\emph{Ergodic Theory Dynam.\ Systems} {\bf 14} (1994) 331--349.

\bibitem{M07}
I.~Melbourne. {Rapid decay of correlations for nonuniformly hyperbolic flows}.
  \emph{Trans.\ Amer.\ Math.\ Soc.} \textbf{359} (2007) 2421--2441.

\bibitem{M09} I.\ Melbourne.
{Decay of correlations for slowly mixing flows.}
\emph{Proc.\ London Math.\ Soc.} {\bf 98} (2009) 163--190.

\bibitem{rapid} 
I.~Melbourne. Superpolynomial and polynomial mixing for semiflows and flows. 
\emph{Nonlinearity} {\bf 31} (2018) R268--R316.

\bibitem{MT12} I.~Melbourne, D.~Terhesiu. 
{Operator renewal theory and mixing rates for dynamical
systems with infinite measure}.  
\emph{Invent.\ Math.} \textbf{189} (2012) 61--110.

\bibitem{MT13}
I.~Melbourne, D.~Terhesiu. {First and higher order uniform dual ergodic
  theorems for dynamical systems with infinite measure}. \emph{Israel J. Math.}
  \textbf{194} (2013) 793--830.

\bibitem{MT17} I.\ Melbourne, D.\ Terhesiu. 
{Operator renewal theory for continuous time dynamical systems with finite and infinite measure}. 
{\em Monatsh. Math.} \textbf{182} (2017) 377--432.

\bibitem{MTtoralsub} I.\ Melbourne, D.\ Terhesiu.
{Mixing properties for toral extensions 
of slowly mixing dynamical systems  with finite and infinite measure}.
 To appear in {\em J. Mod. Dyn.}

\bibitem{MTsub} I.\ Melbourne, D.\ Terhesiu.
Renewal theorems and mixing for non Markov flows with infinite measure.
Preprint 2017: arXiv:1701.08440

\bibitem{dMvS} W.\ de Melo, S.\  van Strien.
{\em One-Dimensional Dynamics.}
Springer, Berlin Heidelberg New York, (1993).

\bibitem{Now93} T.\ Nowicki.
Some dynamical properties of S-unimodal maps.
\emph{Fund.\ Math.} {\bf 142} (1993) 45--57.

\bibitem{Pollicott85}
M.~Pollicott. {On the rate of mixing of Axiom A flows}. \emph{Invent. Math.}
  \textbf{81} (1985) 413--426.

\bibitem{PomeauManneville80}
Y.\ Pomeau, P.\ Manneville. Intermittent transition to turbulence in
  dissipative dynamical systems. \emph{Comm. Math. Phys.} \textbf{74} (1980)
  189--197.

\bibitem{Sarig02}
O.~M.\ Sarig. {Subexponential decay of correlations}. 
\emph{Invent.\ Math.} \textbf{150} (2002) 629--653.

\bibitem{SzaszVarju07}
D.~Sz{\'a}sz, T.~Varj{\'u}. Limit laws and recurrence for the planar
  {L}orentz process with infinite horizon. \emph{J. Stat. Phys.} \textbf{129}
  (2007) 59--80.

\bibitem{Terhesiu12}
D.~Terhesiu. {Improved mixing rates for infinite measure preserving transformations}.
 \emph{Ergodic Theory Dynam.\ Systems} {\bf 35} (2015) 585--614.

\bibitem{Tsujii}
M. Tsujii.  Exponential mixing for generic volume-preserving Anosov flows in dimension three.  \emph{J. Math. Soc. Japan} {\bf 70} (2018) 757--821.

\bibitem{Young98} L.-S.\ Young.
Statistical properties of dynamical systems with some
hyperbolicity. 
\emph{Ann.\ of Math.} {\bf 147} (1998) 585--650.

\bibitem{Young99}
L.-S. Young. Recurrence times and rates of mixing. \emph{Israel J. Math.}
  \textbf{110} (1999) 153--188.

\bibitem{Zweimuller98}
R.~Zweim{\"u}ller. Ergodic structure and invariant densities of non-{M}arkovian
  interval maps with indifferent fixed points. \emph{Nonlinearity} \textbf{11}
  (1998) 1263--1276.


\end{thebibliography}
\end{document}